\documentclass[12pt,a4paper]{article}
\usepackage[OT1]{fontenc}
\usepackage[latin1]{inputenc}
\usepackage[english]{babel}
\usepackage{amsfonts,amsmath,relsize,amsthm,units,amssymb,bbm,newtxtext,newtxmath}
\usepackage[nottoc]{tocbibind}
\usepackage[footskip=.5in]{geometry}

\newtheorem*{theorem*}{Theorem}
\newtheorem{teo}{Theorem}[section]
\newtheorem{prop}[teo]{Proposition}
\newtheorem{lem}[teo]{Lemma}
\newtheorem{coro}[teo]{Corollary}
\theoremstyle{definition}
\newtheorem{defi}[teo]{Definition}
\newtheorem{rem}[teo]{Remark}

\def\gl{\textrm{GL}}

\def\dm{\textrm{Diff}(M)}
\def\ds1{\textrm{Diff}(S^1)}
\def\h{\mathcal H}

\def\kh{{\mathcal K}(\mathcal H)}

\DeclareMathOperator\ad{ad}
\DeclareMathOperator\Ad{Ad}
\DeclareMathOperator\lie{Lie}
\DeclareMathOperator\re{Re}
\DeclareMathOperator\im{Im}
\DeclareMathOperator\LG{Length_{\mu}}
\DeclareMathOperator\Le{Length}

\DeclareMathOperator\di{dist}
\DeclareMathOperator\dg{dist_{\mu}}
\DeclareMathOperator\bh{\mathcal L(H)}
\DeclareMathOperator\tr{Tr}
\DeclareMathOperator\difm{Diff(M)}
\DeclareMathOperator\Exp{Exp}

\begin{document}

\title{\vspace*{-2cm}The metric geometry of infinite dimensional  Lie \newline groups and their homogeneous spaces\footnote{2010 MSC Primary 22E65, 58B20; Secondary 53C22, 47B10, 58D05}.}
%22E65 Infinite-dimensional Lie groups and their Lie algebras (in Topological Groups and Lie groups)
%58B20 Riemannian, Finsler and other geometric structures
%53C22 Geodesics (in Global differential Geometry)
%47B10 Operators belonging to operator ideals (in Operator theory)
%58D05 Groups of diffeomorphisms and homeomorphisms as manifold (in Global Analysis)
\date{}
\author{Gabriel Larotonda\footnote{Supported by CONICET and ANPCyT, Argentina},\footnote{Instituto Argentino de Matemática "Alberto P. Calderón" and Facultad de Ciencias Exactas y Naturales, Universidad de Buenos Aires. e-mail: glaroton@dm.uba.ar}}

\maketitle

%\begin{spacing}{.95}
\abstract{\footnotesize{\noindent We study the geometry of Lie groups $G$ with a continuous Finsler metric, assuming the existence of a subgroup $K$ such that the metric is right-invariant for the action of $K$. We present a systematic study of the metric and geodesic structure of homogeneous spaces $M$ obtained by the quotient $M\simeq G/K$. Of particular interest are left-invariant metrics of $G$ which are then bi-invariant for the action of $K$. We then focus on the geodesic structure of groups $K$ that admit bi-invariant metrics, proving that one-parameter groups are short paths for those metrics, and characterizing all other short paths. We provide applications of the results obtained, in two settings: manifolds of Banach space linear operators, and groups of maps from compact manifolds. }}
%\end{spacing}

\setlength{\parindent}{0cm} %% para que no indente los parrafos nuevos

%\tableofcontents

\section{Introduction}

The purpose of this paper is the study of the geometry of Lie groups $G$ with a continuous Finsler metric, assuming the existence of a subgroup $K$ such that the metric is right-invariant  for the action of $K$. In the first part we are concerned with the metric and geodesic structure of homogeneous spaces $M$ obtained by the quotient $M\simeq G/K$; we prove a structural result concerning the descent of the metric to the quotient, showing that a path $\gamma\subset G$ is short when pushed to the quotient if and only if $\gamma$ realizes distance among fibers in $G$ (Theorem \ref{elcorodelq}). This theorem contains, as particular instances, many results of the Finsler and Riemannian geometry of homogeneous spaces in finite and infinite dimensions: for instance geometries of the group of positive invertible operators, and Bures' metric in state space \cite{acl,bjl,cm,dittmann}. Of particular interest are left-invariant metrics of $G$ which are then bi-invariant for the action of $K$. In the second part we focus on the geodesic structure of groups $K$ that admit bi-invariant metrics, proving that one-parameter groups are short paths for those metrics in locally exponential lie groups, and characterizing all short paths (Theorem \ref{nonstrict}), using the faces of the unit ball in $\lie(K)$. This theorem has applications to several settings, and provides a unified treatment of known results from the Finsler geometry, both in finite and infinite dimension \cite{alhr,alr,alv,atkin1}. The third part is concerned with examples and applications of these theorems.

\smallskip 

Since we want to give a systematic framework to the theory of intrinsic distances derived from tangent metrics, in the setting of Lie groups and their homogeneous spaces,  we will discuss a bit further this circle of ideas in this introduction. The emphasis is put in infinite-dimensional Lie groups and manifolds, modeled by locally convex topological vector spaces, and metrics defined as continuous distributions of tangent Finsler norms or semi-norms (all the metrics considered in this paper are  non-negative); our metrics are only continuous and not necessarily smooth. There are technical and fundamental differences with the classical theory of metrics on Lie groups, coming from two directions: the first one is going from finite dimensional manifolds to infinite dimensional manifolds, and the second one is going from Riemannian (or even classical Finsler metrics, which are smooth and have an auxiliary Riemannian metric given in terms of the Hessian of the metric) to only continuous tangent metrics. However, it is more often than not that both difficulties arise simultaneously, since infinite dimensional vector spaces come with intrinsic and interesting continuous norms or semi-norms.

\smallskip

Many of the ideas proposed here can be extended to principal $G$  bundles, but we preferred to keep the exposition in the realm of homogeneous spaces for clarity. Let us recall now some standard results of the classical theory of Riemannian homogeneous spaces, to put in context what we are trying to achieve here. Let $\pi:G\times M\to M$ be a smooth action of a finite dimensional Lie group $G$ in a smooth finite dimensional manifold $M$, and let $K$ be the stabilizer of the action at $p\in M$ (i.e. $K$ is the isotropy group of $\pi_p=\pi(\cdot,p)$). Denoting $\pi(u,p)=u\cdot p$ for $u\in G$, then $u\lie(K)\subset T_uG\simeq u\lie(G)$ is the stabilizer of $u\cdot p\in M$. By choosing in $G$ a Riemannian metric $(G,g_0)$ \textit{which is right-invariant for the action of $K$}, one obtains a reductive connection, decomposing, for each $u\in G$, $
T_uG\simeq u\lie(K)\oplus \mathfrak h_u$ where $\mathfrak h_u$ is the orthogonal of $u\lie(K)$, called the \textit{horizontal distribution}. Thus for $V\in T_{u\cdot p}M$, denoting $\kappa_{u\cdot p}(V)\in \mathfrak h_g$ the unique \textit{horizontal lift} of $V$, one obtains a smooth section $\kappa:TM\to \mathfrak h\subset TG$, and if $X$ is a smooth vector field $M$, we denote $X^H=\kappa\circ X$ the smooth horizontal lift of $X$ in $G$. The distribution $\Ad_G\lie(K)=\ker \pi_*$ is also known as the \textit{vertical space}, and vectors in this kernel are called \textit{vertical vectors}.

\smallskip

If $(M,g)$ is a Riemannian metric in $M$, and we fix $p\in M$, the map $\pi=\pi_p:G\to M$ is a \textit{Riemannian submersion} if it is a submersion and its differential is isometric restricted to the horizontal distribution. Reciprocally, we can endow $M$ with a Riemannian metric $g$ that makes of $\pi$ a Riemannian submersion by declaring that $\pi_*u:T_uG\to T_{u\cdot p}M$  are isometric for all $u\in G$. That is, if $V=\pi_{*u}v,W=\pi_{*u}v\in T_{u\cdot p}M$, then $\langle V,W\rangle_{u\cdot p}=\langle V^H,W^H\rangle_u\quad \forall\, u\in G$. Note that this can be restated, using the orthogonality of the distribution, as
\begin{equation}\label{qme}
\|V\|_u=\inf\{ \|v-uz\|_u: z\in \lie(K)\}.
\end{equation}
For $\pi: (G,g_0)\to (M,g)$ a Riemannian submersion, let $X,Y$ be smooth vector fields in $(M,g)$, and let $\Gamma$ be a geodesic of $(G,g_0)$. The folllowing well-know results \cite[2.90, 3.55,2.109]{gallot} are at the core of our (non-Riemannian) point of view:
\begin{itemize}
\item[-] If $\nabla^0$ is the Levi-Civita connection of $(G,g_0)$ and $\nabla$ is the Levi-Civita connection of $(M,g)$, then for all $u\in G$ we have $
(\nabla_XY)_{\pi(u)}=\pi_{*u}(\nabla^0_{X^H}Y^H)$. 
\item[-] If $\Gamma'(0)$ is horizontal, then $\Gamma'(t)$ is horizontal for all $t$, and $\gamma=\pi(\Gamma)$ is a geodesic of $(M,g)$ with the same length. Conversely, if $\gamma$ is a geodesic of $(M,g)$ with $\gamma(0)=\pi(u)$, there exists a unique horizontal lift $\Gamma$ of $\gamma$ such that $\Gamma(0)=u$ and $\Gamma$ is a geodesic of $(G,g_0)$.
\item[-] When the metric $g_0$ on $G$ is bi-invariant, it is also well-known that geodesics $\delta$ of $(G,g_0)$ are left-translations of one-parameter groups: $\delta(t)=ue^{tv}$.
\end{itemize}
In particular, the Riemannian exponential maps are thus related by $\Exp^M(V)=\pi\circ\Exp^G(V^H)$, and by the Hopf-Rinow theorem, if $(G,g_0)$ is complete, then $(M,g)$ is complete. For bi-invariant metrics geodesics of $(M,g)$ are then given by $t\mapsto \pi(ue^{tz})$ where $z\in \lie(G)$ is a horizontal vector, and as a consequence $(M,d_g)$ is a complete metric space.

\medskip

In the presentation and results just mentioned for Riemannian manifolds, there are several assertions that fail for non-Riemannian metrics; they fail for Finsler norms that are not Riemannian, they fail for Riemannian metrics that are only continuous with respect to the topology of the locally convex spaces $E$ modeling the manifolds involved, and they even fail for Riemann-Hilbert metrics on infinite dimensional manifolds. Without being exhaustive, let us mention only that if the metric is non Riemannian then the horizontal distribution $\mathfrak k\subset TG$ is not given, therefore the section $\kappa:TM\to TG$ and the horizontal lifts $X^H$ of fields $X$ in $M$ are not given; the first item the short list above is lost, and it is unclear if there is a replacement for the Riemannian exponential. The theorem of Hopf-Rinow is false in infinite dimension (even for Riemann-Hilbert metrics, see McAlpin \cite{mcalpin}), hence the metric properties of the spaces involved cannot be derived from the extendability of geodesics. 

\smallskip 

This is also a major problem when the metric is not Riemannian, and there is no  connection or distribution: which is the right notion of geodesic, in the absence of Euler's equations? We have taken the approach of the metric geometry (or length spaces, see Burago-Burago and Ivanov \cite{bbi}): a path in a manifold is a geodesic if it is minimizing, that is if its length equals the distance among the endpoints. The distance, of course, is defined as the infima of the lengths of paths joining given endpoints (usually called the \textit{rectifiable distance}). Note that the length of a path is a notion that makes sense for any continuous distribution of tangent norms or semi-norms in the given manifold.

\medskip

Let us describe now discuss here, with a bit more detal, the contents of this paper. After giving the necessary context and definitions for manifolds and Lie groups modeled by locally convex topological vector spaces, we examine in Section \ref{intro1} of this paper the continuity of the group operations in $G$, for the topology defined by the rectifiable distance, and assuming the group has a smooth exponential, we also derive some norm inequalities for it. Regarding the action of $G$ on a manifold $M$, there is a natural tangent norm for $M$ given by (\ref{qme}) which we call the \textit{quotient metric}, and we show in Section \ref{intro1} of this paper that it is well-defined in this ample context, provided the tangent norms in $G$ are right-invariant for the action of the isotropy group $K$. Thus one is faced with two distances in $M$, the one given by the infima of lengths of paths in $M$ (measured with the quotient metric), and the one given by the distance in $G$ among the fibers $gK,hK$. It is shown (in Section \ref{hs}) that both distances agree, as in the Riemannian setting; this opens several applications to Finsler norms that will be dealt with in Section \ref{examples}. At the end of Section \ref{hs}, we discuss analogues of the second Riemannian property mentioned above, regarding geodesics of $G$ and geodesics of $M$. The main results of Section \ref{hs} are Theorem \ref{elcorodelq} and Theorem \ref{arribajo2}. If a group $G$ admits a bi-invariant metric, and the metric is not Riemannian, it is also expected that one-parameter groups will be short paths for the rectifiable distance, and this is proved in Section \ref{biinvme} of the paper. We also prove that when the norm is strictly convex, those are the unique possible short paths. For the case of non-strictly convex norms, we give a nice characterization of all short paths in the group. The main results of that section are then Theorem \ref{mini}, Theorem \ref{uniestricta} and Theorem \ref{nonstrict}. We illustrate the use of all these ideas in the final section of the paper (Section \ref{examples}), where we take a look at some known results under the point of view presented here, and we obtain new results regarding distances and geodesics in several contexts, including groups of maps from compact manifolds (diffeomorphism groups, loop groups) and groups of Banach space operators and their homogeneous spaces (invertible operators, isometries, unitary operators, positive operators).

\section{Lie groups and rectifiable metrics}\label{intro1}

Let us present in this section some general definitions and considerations that will be used throughout the paper. Manifolds in this paper will be modeled with charts in a Hausdorff locally convex topological vector space (shortly l.c.s.). The differential of a map $f:M\to N$ among smooth manifolds will be denoted by $f_*:TM\to TN$ and its specialization by $f_{*p}:T_pM\to T_{f(p)}N$, $p\in M$.

\medskip 

In this paper, a \textit{Lie group} $G$ is a manifold such that the operation $(x,y)\mapsto xy^{-1}$ is smooth (at least $C^2$) as a map $G\times G\to G$. If $g\in G$ and $L_g:h\mapsto gh$ denotes the left multiplication in $G$, with some abuse of notation we denote
$$
gv=L_gv=(L_g)_{*h}v\in T_{gh}G
$$
for $h\in G$, $v\in T_hG$. We denote $1\in G$ the identity of the group and $\lie(G)=T_1G$  its Lie algebra. The Lie bracket in $\lie(G)$ will be denoted by $[\cdot,\cdot]$: it is always a bi-linear, anti-symmetric and continuous map. If $c_g(h)=ghg^{-1}$ is the conjugation automorphism, i.e. $c_g=L_gR_g^{-1}$ for $g\in G$, we follow the standard notation $\Ad_g=(c_g)_{*1}$ with $\Ad:G\to  \gl(\lie(G))$ a group homomorphism.

\begin{rem}\label{noad}
If $\lie(G)$ is not a Banach space then $\gl(\lie(G))$ is not necessarily a Lie group, but it is a subgroup of the space of diffeomorphisms of $\lie(G)$ therefore there is a natural notion of smoothness. We denote $\ad=(\Ad)_{*1}:\lie(G)\to  \mathcal L(\lie(G))$ which is a linear Lie algebra morphism, and in fact $\ad(v)(w)=[v,w]$ for any $v,w\in \lie(G)$ (see Neeb \cite[Section II.3]{neeb}). 
\end{rem}

\subsection{Finsler metrics}

In this section we define continuous Finsler metrics over smooth manifolds. For a more radical approach that drops the smoothness assumption on the manifolds, see Andreev \cite{andreev}. See also Berestovskii \cite{beres,beres2} for a systematic account of finite dimensional homogeneous manifolds with Finsler metrics defined by distributions in the fiber bundle.

\begin{defi}[Finsler norms and semi-norms]\label{msn}
Let $E$ be a l.c.s., $\mu=|\cdot|:E\to\mathbb R_{\ge 0}$ a continuous function.
\begin{enumerate}
\item $\mu$ is a \textit{Finsler semi-norm} if it is sub-additive and positively homogeneous:  $|v+w|\le |v|+|w|$ and $|\lambda v|=\lambda |v|$ for $v,w\in E$ and $\lambda\in \mathbb R_{\ge 0}$. 
\item $\mu$ is a \textit{non-degenerate Finsler norm} if $|v|=0$ implies $v=0$.
\end{enumerate}
If  $|tv|=|t|\,|v|$ for all $t\in\mathbb R$, we obtain the standard notion of continuous vector space semi-norm (if moreover it is non-degenerate, we have a vector space norm).
\end{defi}

\begin{defi}[Finsler metrics for $TM$]
Let $M$ be a manifold modeled by a l.c.s $E$. Let $\mu:TM\to \mathbb R_{\ge 0}$ be a selection of a tangent Finsler semi-norm $\mu_p=|\cdot|_p:T_pM\to \mathbb R_{\ge 0}$, for each $p\in M$. 
\begin{enumerate}
\item $\mu$ is \textit{continuous along paths} if for each $C^1$ path $\gamma:[a,b]\to M$ the map $t\mapsto |\dot{\gamma}_t|_{\gamma_t}$ is continuous.
\item $\mu$ is continuous if $\mu:TM\to \mathbb R$ is a continuous map. 
\end{enumerate}
\end{defi}

\noindent$\S$ A word of caution: as discussed in the introduction, our definition of Finsler metric is far more general than the standard one; we are not assuming smoothness, therefore the usual machinery of Riemann-Finsler geometry \cite{bao} is not at hand.

\begin{defi}[Uniform Finsler metrics for $TG$]\label{sc}
Let $G$ be a Lie group, $\mu:TG\to\mathbb R_{\ge 0}$ a continuous Finsler metric. For each $g,h\in G$, consider the linear operators 
$$
(L_g)_{*h},\, (R_g)_{*h} :(T_hG,|\cdot|_h)\to (T_{gh}G,|\cdot|_{gh}),
$$
denote $|(L_g)_{*h}|=\sup\{ |gv|_{gh}: v\in T_hG, |v|_h=1\}$ and likewise $|(R_g)_{*h}|$. Note that when $g=1$, both norms are identically $1$ for all $h\in G$.

\smallskip

We say that $\mu$ is \textit{$L$-uniform} if there exists a $\tau_G$ upper semi-continuous function $L:G\to\mathbb R_{\ge 0}$ with $L(1_G)=1$ such that $|(L_g)_{*h}|\le L(g)$ for all $g,h\in G$.

\smallskip

Replacing $L$ with $R$ we obtain the definition of \textit{$R$-uniform}. 
\end{defi}

\begin{rem}[Left-invariant metrics]\label{norrma}
If we fix $|\cdot|$ a Finsler semi-norm in $\lie(G)$ and define $|v|_g=|(L_g)_{*1}^{-1}v|$ for $v\in T_gG$, then the group $G$ has a \textit{left-invariant Finsler metric} $|\cdot|_g:T_gG\to \mathbb R_{\ge 0}$, because if $g,h\in G$ then
$$
|hv|_{gh}=|(gh)^{-1}hv|=|g^{-1}v|=|v|_g\quad \textrm{ for }v\in T_gG,
$$
and the map $(g,v)\mapsto |v|_g=|g^{-1}v|$ is continuous as a map from $TG$ to $\mathbb R$. Any left-invariant Finsler metric in $G$ can be obtained with this procedure. 

\smallskip

Note that in this case $|L_gv|_{gh}=|h^{-1}g^{-1}gv|_1=|h^{-1}v|_1=|v|_h$ therefore automatically $|(L_g)_{*h}|=1$ for all $g,h\in G$. Note also that $|\Ad_gv|_1=|vg^{-1}|_{g^{-1}}\le  R(g^{-1})|v|_1$ when the metric is also $R$-uniform.
\end{rem}

We will present some more non-trivial examples of Finsler metrics that are not right nor left invariant in Section \ref{examples}, in particular Bures' metric for invertible operators of a $C^*$-algebra (Section \ref{bures}), and integral metrics for groups of diffeomorphisms of a compact manifold (Section \ref{grudif}).

\medskip

\begin{defi}[Rectifiable paths and length]\label{recti}
We say that a curve $\alpha:[a,b]\to G$ is \textit{rectifiable} if $\alpha$ is differentiable a.e. in some chart of $G$ and $t\mapsto |\dot{\alpha}(t)|_{\alpha(t)}$ is Lebesgue integrable. It is also possible to consider rectifiable paths in the metric sense (see \cite{andreev} and \cite[Section 2]{menucci}), though we won't be needing such machinery here. 

\smallskip

For piecewise smooth or rectifiable arcs $\alpha:[a,b]\to G$, define the \textit{length} of $\alpha$ as
$$
\LG(\alpha)=\int_a^b |\dot{\alpha}(t)|_{\alpha(t)} dt.
$$
\end{defi}

\begin{defi}
For $g,h\in G$, consider the infima of the lengths of such arcs joining $g,h$ in $G$,
$$
\dg(g,h)=\inf\{ \LG(\alpha):\alpha:[0,1]\to G\textrm{ is rectifiable }, \alpha(0)=g,\alpha(1)=h\}.
$$
Then $\dg:G\times G\to \mathbb R_{\ge 0}$ is a p.s.d. (pseudo-quasi-distance): 
\begin{enumerate}
\item it is finite in each arc-wise connected component of $G$, 
\item it obeys the triangle inequality, 
\item it is reversible if and only if $|\cdot|_g$ is homogeneous (if it is a norm) for each $g\in G$.
\end{enumerate}
More details on asymmetric distances can be found in \cite{mennu2} and the references therein. 
\end{defi}

\medskip

\begin{rem}
The matter of whether $\dg(x,y)=0$ implies $x=y$ in $G$ is more subtle. In fact, even if we start with a norm in the Lie algebra of $G$ ($\mu(v)=0$ implies $v=0$), and consider the left-invariant metric it induces, there are examples when this fails, see Michor and Mumford \cite{mm} for such an example; see also the paper by Clarke \cite{crm}. 
\end{rem}

$\S$ A sufficient condition to obtain the non-degeneracy of $\dg$ (for left-invariant metrics) is given by asking $|\cdot|$ to induce in $\lie(G)$ its original l.c.s topology. In particular, if $G$ is a finite dimensional Lie group, $\dg$ is non-degenerate for any chosen \textit{norm} in $\lie(G)$.

\begin{defi}
We will denote with $(G,\dg)$ the underlying (pseudo-quasi) metric space. Nevertheless, this distance or quasi-distance induces a topology in $G$, and we will refer to the topology induced as $\tau_{\mu}$ when needed; otherwise the topology of $G$ will always be the manifold topology denoted by $\tau_G$. Clearly $\tau_{\mu}$ will be Hausdorff if and only if $\dg$ is non-degenerate. 
\end{defi}

We now turn to the the subject of comparing topologies and the continuity of the group operations for $\tau_{\mu}$.

\begin{prop}[Continuity of the group operations]\label{contigru}
Let $\mu$ be a continuous Finsler metric in $TG$. If $\tau_G$ denotes the original topology of $G$ and $\tau_{\mu}$ the one induced by the (pseudo-quasi) metric, then
\begin{enumerate}
\item If $g_i\to g$ in $\tau_G$, then $g_i\to g$ in $\tau_{\mu}$.
\item $\dg:(G,\tau_G)\times (G,\tau_G)\to \mathbb R$ is continuous, and $\tau_{\mu}$ is finer than $\tau_G$.
\item If $\mu$ is $L$-uniform (resp. $R$), the left (resp. right) multiplication map by $g\in G$ is Lipschitz continuous for $\dg$ with constant $L(g)$ (resp $R(g)$).
\item If $\mu$ is reversible, left-invariant and $R$ uniform (or vice-versa), the product $m:G\times G\to G$, $m(g,h)=gh$  is jointly continuous for $\dg$, and so is  $i:G\to G$, $i(g)=g^{-1}$, the inversion map.
\end{enumerate}
\end{prop}
\begin{proof}
Let $(U,\varphi)$ be a chart around $1\in G$, where $U=U^{-1}$ and $V=\varphi(U)$ is a convex balanced neighborhood of $0\in \lie(G)$. If $G\ni g_i\to g$ in $\tau_G$, then $v_i=\varphi(g^{-1}g_i)\to 0$ in $\lie(G)$. Since left multiplication is $C^1$, the map $\psi:V\times V\to TG$ given by $(v,w)\mapsto (L_g)_{*\varphi^{-1}(v)}(\varphi^{-1})_{*v}(w)$ is continuous. Since $\psi(0,0)=(L_g)_{*1}(\varphi^{-1})_{*0}0=0$, if $\varepsilon>0$ is given, and shrinking $V$ if necessary, we have $|\psi(v,w)|_{g\varphi^{-1}(v)}<\varepsilon$ when $v,w\in V$. Let $\alpha(t)=g\varphi^{-1}(tv_i)$ for $t\in [0,1]$, clearly $\alpha$ joins $g$ to $g_i$ in $G$.  Now $\dot{\alpha}(t)=\psi(tv_i,v_i)$ therefore if $i\ge i_0$ is such that $v_i\in V$, then also $tv_i\in V$ and
$$
\dg(g,g_i)\le \LG(\alpha)=\int_0^1 |\dot{\alpha}|_{\alpha}=\int_0^1 |\psi(tv_i,v_i)|_{g\varphi^{-1}(tv_i)} dt<\varepsilon,
$$
which in turn implies that $\dg(g,g_i)\to 0$ when $g_i\to g$ in $\tau_G$. The second assertion is immediate from the first one. 

Now assume $\mu$ is $L$-uniform, and note that there is a bijection among piecewise smooth maps $\alpha$ joining $h,k$ and those joining $gh,gk$ given by left multiplication $\alpha\mapsto \beta=g\alpha$. Since $\dot{\beta}=(g\alpha)^{\cdot}=(L_g)_{*\alpha}\dot{\alpha}$, we have 
$$
|\dot{\beta}|_{\beta}=|(L_g)_{*\alpha}\dot{\alpha}|_{g\alpha}\le L(g)\,|\dot{\alpha}|_{\alpha}.
$$
Therefore $\LG(\beta)\le L(g) \LG(\alpha)$, which implies $\dg(gh,gk)\le L(g) \dg(h,k)$. The  proof  for $R$-uniform metrics is similar. If the metric is reversible, left-invariant and $R$-uniform, assume $\dg(g_n,g)\to 0$ and $\dg(h_n,h)\to 0$. Then 
\begin{eqnarray}
\dg(g_nh_n,gh) & \le &\dg(g_n h_n,g_n h)+\dg(g_n h,gh)\nonumber\\
&\le & \dg(h_n,h)+R(h)\dg(g_n,g)\to 0\nonumber
\end{eqnarray}
therefore the product is jointly continuous. Likewise,
$$
\dg(g_n^{-1},g^{-1})\le R(g^{-1}) \dg(g,g_n)\to 0.
$$
The proof for right-invariant, $L$-uniform reversible metrics is similar and therefore omitted.
\end{proof}

\begin{rem}[Left half-Lie groups]
A half-Lie group is a smooth manifold $G$ modeled by a Fréchet space, such that the group operations are continuous for the manifold topology, and for each $g\in G$, the map $L_g:G\to G$ is also smooth. This notion was recently systematically introduced by Neeb and Marquis in \cite{nm}. Inspection of the results in this section shows that Finsler metrics can be introduced in that context, and the results of the previous proposition then hold. Left smoothness can be replaced by right smoothness (in fact this was how it was originally introduced by Kriegl, Michor and Reiner in \cite{kmr}, to study groups of diffeomorphisms), and evidently the notions of left-invariant and $L$-uniform can be replaced by their right counterparts, to obtain the continuity of the right multiplication map for $\tau_{\mu}$ in the previous proposition.
\end{rem}

\subsubsection{The local structure}\label{regular}

\begin{rem}\label{difexp}
If $G$ carries a smooth exponential function, then its derivative can be computed explicitly: for $v,w\in \lie(G)$,
$$
\exp_{*w}(v)=e^w\int_0^1 \Ad_{e^{-sw}}v\,ds = e^w \int_0^1 e^{-s\, \ad w}v\,ds
$$
where, fixing $w$, the convergence of the integral is for each $v\in \lie(G)$; and if the Lie algebra is complete, it defines a linear operator 
$$
\kappa_w=\int_0^1 e^{-s\, \ad  w}\,ds
$$
where the convergence is in the locally convex topology of the continuous linear operators of $\lie(G)$. 

If  $\lie(G)$ is a Banach space, this operator can be computed as the functional calculus $\ad \,w\mapsto F(\ad \,w)=\kappa_w$, where $F:\mathbb{C}\to \mathbb{C}$ is the entire function given by $F(\lambda)=\lambda^{-1}(1-e^{-\lambda})$, with $F(0)=1$. See \cite[Proposition II.5.7]{neeblie} for the details. 
\end{rem}

\begin{lem}[Continuity of the exponential map]\label{explr}
If $G$ has a smooth exponential and $\mu:TG\to\mathbb R_{\ge 0}$ is left-invariant and  $R$-uniform, then the exponential map $\exp:\lie(G)\to G$ and its differential are norm-to-$\tau_{\mu}$ continuous. Explicitly if $v,w\in \lie(G)$ and $C_w=\int_0^1 R(e^{tw})$ then
$$
\dg(e^v,e^w)\le C_w|v-w| \quad \textrm{ and } \quad |e^{-w}\exp_{*w}(v)|\le C_w |v|.
$$
When $\mu$ is right-invariant and $L$-uniform the same bounds hold replacing $C_w$ by $C_v=\int_0^1 L(e^{tv})dt$.
\end{lem}
\begin{proof}
Let $\alpha(t)=e^ve^{-tv}e^{t w}$, it a smooth path in $G$ joining $e^v,e^w$, its speed is given by $\dot{\alpha}_t=e^{(1-t)v}(w-v)e^{tw}$. Thus $|\alpha^{-1}_t\dot{\alpha}_t|\le R(e^{tw})\,|w-v|$ and integrating we obtain $\dg(e^v,e^w)\le \LG(\alpha)\le C_w |w-v|$. Likewise,
$$
|\exp_{*w}(v)|_{e^w}\le \int_0^1 |e^{-sw} v e^{sw}|ds \le \int_0^1 R(e^{-sw})ds |v|.\qedhere
$$
\end{proof}

\section{Homogeneous spaces}\label{hs}

Recall that for $g\in G$, we denote with $L_g,R_g, c_g=L_g\circ R_{g^{-1}}$ the left and right multiplication maps on $G$ and the conjugation map, their differentials at $g=1\in G$ will be denoted as $L_g,R_g,\Ad_g$ respectively. 

\bigskip

\begin{defi}\label{smoothm}
We consider a $G$-homogeneous space $M$, with the following assumptions:
\begin{enumerate}
\item  $M$ is a smooth manifold where $G$ acts smoothly (at least $C^1$) and transitively with a left action $\pi:G\times M\to M$, denoted $(g,x)\mapsto g\cdot x=\pi_x(g)=\ell_g(x)$. 
\item There exists $x\in M$ such that each (piecewise) $C^1$ path $\gamma:[a,b]\to M$ starting at $x$ has a (piecewise) $C^1$ lift $\Gamma:[a,b]\to G$ starting at $1\in G$. Namely a path such that $\Gamma(a)=1$ and $\pi_x(\Gamma)=\Gamma\cdot x=\gamma$.
\end{enumerate}
\end{defi}

\begin{rem}\label{plp}By the transitivity of the action and the smoothness of the left action of $G$ on itself, for each $y=g\cdot x\in M$ and each (piecewise) $C^1$ path $\gamma:[a,b]\to M$ starting at $y$, we have a  (piecewise) smooth lift $\Gamma:[a,b]\to G$ starting at $g\in G$. Namely a path such that $\Gamma(a)=g$ and $\pi_x(\Gamma)=\Gamma\cdot x=\gamma$. Actually, we would only need to have a good definition of smooth paths in $M$ (and the lifting property), the full notion of manifold for $M$ is not used in many considerations.
\end{rem}

Note that in particular, for any $x\in M,g\in G$, the linearization $(\pi_x)_{*g}:T_gG\to T_{g\cdot x}M$ is an epimorphism.

\begin{rem}\label{openq}
Since $G$ is a Lie group, $TG$ has a natural structure of Lie group. For $x\in M$, denote $\pi_*:TG\to TM$ given by $\pi_*(g,v)=(\pi_x(g),(\pi_x)_{*g}(v))$. What we are asking that $\pi_*:TG\to TM$ is surjective and has the path-lifting property. The assumptions are immediate if $G,M$ are $C^1$ Banach manifolds and $M$ is the image of a quotient map, since $q:g\mapsto g\cdot x$ is a $C^1$ submersion from $G$ to $M$ and the action is transitive, see for instance Upmeier's book \cite[Theorem 8.19]{upmeier}. In this case we obtain a stronger assertion, $\pi_*:TG\to TM$ is a quotient surjective map (thus open).
\end{rem}

\begin{defi}\label{reguini}
Let $x\in M$, with $K_p\subset G$ we will denote the component of the identity of the closed isotropy subgroup $\{g\in G: g\cdot x=x\}$. Throughout, we will assume that $K_p$ is a \textit{regular} Lie group and that $K_p\subset G$ is an \textit{initial} Lie group (each smooth map $f:N\to G$ from a smooth manifold $N$, with values in $K_p$, is also smooth as a map $f:N\to K_p$). 
\end{defi}

\begin{rem}\label{umenos1}
Some remarks follow: 
\begin{enumerate}
\item That $K_p$ is initial implies that $\lie(K_p)=T_1K_p=\{v\in \lie(G): (\pi_x)_{*1}v=0\}$
is a closed Lie sub-algebra of $\lie(G)$ (see \cite[Proposition II.6.3]{neeblie}).
\item That $K_p$ is regular implies that if $\alpha:[a,b]\to  G$ is piecewise smooth, then $\alpha\subset K_p$ if and only if $\alpha(t_0)\in K_p$ for some $t_0\in [a,b]$ and $\alpha^{-1}\dot{\alpha}\subset \lie(K_p)$ (equivalently $\dot{\alpha}\alpha^{-1}\subset \lie(K_p)$). Moreover, for each smooth path $\xi\subset \lie(K_p)$ there exists a smooth path $\alpha\subset K_p$ with $\alpha(0)=1$ and $\dot{\alpha}\alpha^{-1}=\xi$.
\item Assume $G$ has a smooth exponential. Then clearly if $v\in \lie(G)$ and $\exp(tv)\subset K_p$ for all $t\in\mathbb R$, then $v\in T_1K_p$. On the other hand, if $v\in\lie(K_p)$ then considering $\gamma(t)=\pi_x(e^{tv})=e^{tv}\cdot x$, we have $\gamma(0)=x$ and by Lemma \ref{propac}.2 below,
$$
\dot{\gamma}(t)=(\pi_x)_{*e^{tv}}R_{e^{tv}}v=(\ell_{e^{tv}})_{*x}(\pi_x)_{*1}\Ad_{e^{-tv}}v=(\ell_{e^{tv}})_{*x}(\pi_x)_{*1}v=0,
$$
therefore $\gamma$ is constant and $T_1K=\lie(K_p)=\{v\in \lie(G):\exp(tv)\subset K_p,\,t\in\mathbb R\}$.
\end{enumerate}
\end{rem}

\subsection{Metrics invariant for the right action of the isotropy group}\label{qm}

Recall that we assume throughout that $G$ is Lie group, acting transitively on $M\simeq {\mathcal O}(x)$ with the path lifting property, and that $K_p$ (the identity component of the isotropy group) is initial and regular. Note that these hypothesis are fulfilled for Banach-Lie groups $G$.

\begin{defi}
We assume now that there exists a distinguished $p\in M$ and a Finsler norm $\mu:TG\to \mathbb R$ which is right-invariant for the action of $K_p$; that is if $g\in G$ and $u\in K_p$ then
$$
|vu|_{gu}=|v|_g \textrm{ for all }v\in T_gG.
$$
It is immediate that the length and distance in $G$ are right-invariant for the action of $K_p$:
$$
\dg(gk,hk)=\dg(g,k)\quad \forall \, g,h\in G,\; k\in K.
$$
\end{defi}

\begin{rem}[Left-invariant metrics]\label{rik}
In particular, if $\mu=|\cdot|:\lie(G)\to\mathbb R$ is a metric which is $\Ad_{K_p}$-invariant on $\lie(G)$, for the left-invariant metric induced in $TG$ we have 
$$
|vu|_{gu}=|u^{-1}g^{-1}vu|=|g^{-1}v|=|v|_g
$$
for each $g\in G$, $u\in K_p$, $v\in T_gG$. Therefore the induced left-invariant metric is bi-invariant for the action of $K_p$. Moreover $\LG$ and $\dg:G\times G\to \mathbb R$ (the left-invariant length and distance induced by $\mu$) are both left and right-invariant for the action of the subgroup $K_p$. In the next sections we establish basic properties of groups $K$ with $\Ad_K$-invariant metrics.
\end{rem}

\begin{rem}\label{exptaucont}
 Note that for such $\Ad_K$ bi-invariant metrics (see Remark \ref{difexp} and Lemma \ref{explr})
$$
|e^{-w}\exp_{*w}(v)|\le |v| \quad \forall v\in \lie(G),\, w\in\lie(K)\qquad\textrm{ and }\qquad \dg(e^v,e^w)\le |w-v|,
$$
thus $\exp:(\lie(K),|\cdot|)\to (K,\tau_{\mu})$ is Lipschitz. Clearly, if $\Gamma\subset \lie(K)$ and $\gamma=e^{\Gamma}\subset K$ then
$$
\LG(\gamma)=\int_0^1 |\gamma^{-1}\dot{\gamma}|=\int_0^1 |e^{-\Gamma}\exp_{*\Gamma}{\dot{\Gamma}}|\le \int_0^1 |\dot{\Gamma}|= \Le(\Gamma).
$$
\end{rem}

\subsection{Quotient distance}\label{qd}

We now return to the matter of giving the quotient space $M=G/K$ a metric and a distance compatible with the action of the group $G$. We fix  $p\in M$ such that $\mu:TG\to\mathbb R$ is right-invariant for the action of $K=K_p$.

\begin{defi}
Let $K\subset G$ be the (regular and initial) Lie subgroup given by the connected component of the identity of the isotropy $\{g\in G: g\cdot p=p\}$. Let $x=g\cdot p, y=h\cdot p\in M$, define the \textit{quotient distance} $\dot{d}$ in $M=G/K$ as
$$
\dot{d}(x,y)=\dg(g K,hK)=\inf\{\dg(gk_1,hk_2):\, k_1,k_2\in K\}.
$$
\end{defi}

Recall we are assuming that the metric $\mu$ in $G$ is right-invariant for the action of $K=K_p$, therefore  
\begin{equation}\label{dpunto}
\dot{d}(x,y)=\dg(gK,hK)=\dg(gK,h)=\dg(g,hK)\le \dg(g,h).
\end{equation}

Note that the name \textit{distance} is short for pseudo-distance, since it is not necessarily reversible nor separating ($\dot{d}(x,y)=0$ does not imply $x=y\in M$).

\begin{rem}\label{notpseudo}
Since the distance from a point to a closed set is positive in a metric space, this shows that if $\dg$ is in fact a metric in $G$, then $\dot{d}$ is a metric in $M$ (whether it is reversible or not depends exclusively on the fact of the   homogeneity of $\mu$).
\end{rem}

\begin{rem}[Left-invariant metrics]
If the metric $\mu$ is left-invariant in $G$, then
\begin{equation}\label{qdist}
\dot{d}(g\cdot p,h\cdot p)=\dg(K,g^{-1}h),
\end{equation}
and the metric is invariant for the action of $G$ in $G/K$.
\end{rem}

\smallskip 

Any Cauchy sequence $(g_n\cdot p)_{n\in\mathbb N}\subseteq (M,\dot{d})$ can be lifted to a Cauchy sequence in $(G,\dg)$. Thus one obtains the following:

\begin{teo}\label{comple}
If $(G,\dg)$ is a complete metric space, then $(M,\dot{d})$ is a complete metric space.
\end{teo}
The details of the proof can be found in Takesaki's book \cite[p. 109]{take}, we remark however that the hypothesis  $\tau_G=\tau_{\mu}$ stated in the textbook is unnecessary).

\begin{rem}[The group of metrically  null elements]\label{esgru}
Consider the set 
$$
K_{\mu}=\{u\in G:\dg(1,u)=0\},
$$
which is closed in $\tau_{\mu}$. Assume that the metric $\mu$ is reversible (i.e. $\mu$ is homogeneous). Then
\begin{enumerate}
\item If $\mu$ is $L$-uniform, and $u,k\in K_{\mu}$, then  $\dg(1,u^{-1})\le L(u^{-1})\dg(u,1)=0$  and
$$
\dg(1,uk)\le \dg(1,u)+\dg(u,uk)\le 0+ L(u)\dg(1,k)=0.
$$
Hence $K_{\mu}\leqslant K$ is a closed \textit{subgroup}. 
\item If $\mu$ is $L$-uniform and $u\in K_{\mu}$, $g,h\in G$ then
\begin{eqnarray}
\dg(gu,hu) & \le &  \dg(gu,g)+\dg(g,h)+\dg(h,hu)\nonumber\\
&\le&  L(g)\dg(u,1)+\dg(g,h)+L(H)\dg(1,u)=\dg(g,h),\nonumber
\end{eqnarray}
and with a similar reasoning starting with $\dg(g,h)$ we conclude that $\dg(gu,hu)=\dg(g,h)$.
\item Therefore if $\mu$ is $L$-uniform then \textit{$\dg$ is a right-$K_{\mu}$ invariant metric on $G$}, and it is plain that $\dg(g,h)=0$ if and only if $g^{-1}h\in K_{\mu}$.
\item If $\mu$ is $L$-uniform, note also that if $0=\dot{d}(gK_{\mu},hK_{\mu})=\dot{d}(g,hK_{\mu})$ we have $u_n\in K_{\mu}$ such that $\dg(g,hu_n)<1/n$. Hence $\dg(g,h)\le \dg(g,hu_n)+\dg(hu_n,h)<1/n+0=1/n$, therefore $g^{-1}h\in K_{\mu}$ and $gK_{\mu}=hK_{\mu}$. Hence the quotient distance $\dot{d}$ is non-degenerate. and $(G/K_{\mu},\dot{d})$ is a genuine metric space.
\item If $\mu$ is a $R$-uniform metric, we obtain that $K_{\mu}$ is also a closed subgroup, but now the metric $\dg$ is left-$K_{\mu}$ invariant on $G$.
\item  If $\mu$ is both $L$ and $R$-uniform, then for any $g\in G$ and $u\in K_{\mu}$
$$
\dg(1,gug^{-1})\le L(g)R(g^{-1})\dg(1,u)=0
$$
hence \textit{$K_{\mu}\triangleleft K$ is a normal subgroup}.
\end{enumerate}
\end{rem}

\subsection{Smooth action and quotient norm}

The following lemma collects some elementary identities that will be useful later, and will help us fix the notation. Its verification is left to the reader:

\begin{lem}\label{propac} Let $g,g'\in G$, $x\in M$, $u\in K_x$. Then
\begin{enumerate}
\item $\ell_g\circ \pi_x=\pi_x\circ L_g$, $ (\pi_x)_{*g}=(\ell_g)_{*x}\circ (\pi_x)_{*1}\circ (L_g)^{-1}_{*1}$
\item $(\ell_g)_{*x}\circ (\pi_x)_{*1}\circ \Ad_{g^{-1}}=(\pi_x)_{*g}\circ R_g$
\item $(\pi_{g\cdot x})_{*1}=(\pi_x)_{*g}\circ R_g$
\item $(\ell_{g^{-1}})_{*g\cdot x}\circ (\pi_x)_{*g}=(\pi_x)_{*1}\circ (L_{g^{-1}})_{*g}$
\item $(\pi_x)_{*u}=(\pi_x)_{*1}\circ (R_{u^{-1}})_{*u}$
\item $K_{g\cdot x}=c_g(K_x)$
\item $\lie(K_{g\cdot x})=\Ad_g(\lie(K_x))$.
\end{enumerate}
\end{lem}

\begin{rem}\label{adust}
Note that by the last item of the previous lemma, if $u\in K_x$ then $K_{u\cdot x}=K_x$ therefore
$$
\Ad_u(\lie(K_x))=\lie(K_x).
$$
Assume that $G$ has a smooth exponential. If $v\in \lie(K_x)$ then by Remark \ref{umenos1}.2, $e^v\in K_x$ therefore
$$
\Ad_{e^v}(\lie(K_x))=\lie(K_{e^v\cdot x})=\lie(K_x).
$$
This implies that $\lie(K_x)$ are \textit{stable} Lie sub-algebras.
\end{rem}

The other (and relevant for our purpose) consequence of the previous lemma, is the following, which will help us establish that the quotient (infinitesimal) metric is well-defined in $TM$.

\begin{lem} Let $x\in M$, $g,h\in G$,  $\dot{g}\in T_gG, \dot{h}\in T_{h}G$. Then $g\cdot x=h\cdot x$ and $(\pi_x)_{*g}\dot{g}=(\pi_{x})_{*h}\dot{h}$ if and only if 
$$
u=g^{-1}h\in K_x \quad\textrm{ and } \quad g^{-1}\dot{g}-\Ad_u(h^{-1}\dot{h})\in \lie(K_x).
$$
\end{lem}
\begin{proof}
Assume $g\cdot x=h\cdot x$ and $(\pi_x)_{*g}\dot{g}=(\pi_{x})_{*h}\dot{h}$. Clearly $u=g^{-1}h\in K_x$. Let $\gamma\subset G$ be a smooth path with $\gamma(0)=h, \dot{\gamma}(0)=\dot{h}$, and consider the identity $(g^{-1}\gamma)\cdot x=g^{-1}\cdot(\gamma\cdot x)$ written as $\pi_x(g^{-1}\gamma)=\ell_{g^{-1}}\pi_x(\gamma)$. Differentiating at $t=0$ we obtain 
$$
(\pi_x)_{*g^{-1}h}g^{-1}\dot{h}=(\ell_{g^{-1}})_{*h\cdot x}(\pi_x)_{*h}\dot{h}=(\ell_{g^{-1}})_{*g\cdot x}(\pi_x)_{*h}\dot{h}
$$
where in the last equality we just used once that $h\cdot x=g\cdot x$. Now since $u=g^{-1}h\in K_x$, the fifth item of the previous lemma shows that
$$
(\pi_x)_{*1}(g^{-1}\dot{h}h^{-1}g)=(\pi_x)_{*g^{-1}h}(g^{-1}\dot{h}).
$$
Combining the previous equalities, we obtain
$$
(\pi_x)_{*1}(g^{-1}\dot{h}h^{-1}g)=(\ell_{g^{-1}})_{*g\cdot x}(\pi_x)_{*h}\dot{h}=(\ell_{g^{-1}})_{*g\cdot x}(\pi_x)_{*g}\dot{g}
$$
where in the last equality we used the hypothesis $(\pi_x)_{*g}\dot{g}=(\pi_{x})_{*h}\dot{h}$. Applying the fourth item of the previous lemma,
$$
(\pi_x)_{*1}(g^{-1}\dot{h}h^{-1}g)=(\ell_{g^{-1}})_{*g\cdot x}(\pi_x)_{*g}\dot{g}=(\pi_x)_{*1}(g^{-1}\dot{g}).
$$
Hence $g^{-1}\dot{g}-g^{-1}\dot{h}h^{-1}g\in \ker (\pi_x)_{*1}=\lie(K_x)$ as claimed.
\end{proof}

Since $(\pi_x)_{*1}:T_1G\to T_xM$ is an epimorphism with kernel $\lie(K_x)$, we have a natural identification 
$$
T_xM\simeq \lie(G)/\lie(K_x).
$$
This motivates the following definition (for the case of left-invariant metrics in Banach manifolds, see Upmeier \cite[Proposition 12.31]{upmeier} and the references therein).

\begin{defi}[Quotient Finsler metric]\label{qfm}
Assume that the Finsler norm $\mu$ in $TG$ is $K$-right invariant for given $K=K_p$. Let $v\in T_{g\cdot p}M$, $v=(\pi_p)_{*g}(\dot{g})$ for some $\dot{g}\in T_gG$, and define
$$
\mu(v)_{g\cdot p}=\inf\{ |\dot{g}-gz|_g:\, z\in \lie(K)\}=\textrm{dist}_{\mu_g}(\dot{g},g\lie(K)),
$$
the linear distance in $T_gG$ from any lift $\dot{g}$ of $v$ to the subspace $g\lie(K)\subset T_gG$.
\end{defi}

\begin{prop}\label{eses}
Let $\mu:TM\to\mathbb R$ be defined as above. Then
\begin{enumerate}
\item It is well defined: if  $g\cdot p=h\cdot p$ and $(\pi_p)_{*g}\dot{g}=(\pi_{x})_{*h}\dot{h}=v$, then
$$
\textrm{dist}_{\mu_g}(\dot{g},g\lie(K))=\textrm{dist}_{\mu_h}(\dot{h},h\lie(K)),
$$
hence it defines a Finsler norm on each $T_xM$, $x\in M$.
\item If the metric of $G$ is left-invariant, then 
$$
\mu(v)_{g\cdot p}=\inf\{ |g^{-1}\dot{g}-z|:\, z\in \lie(K)\}=\dg(g^{-1}\dot{g},\lie(K))
$$
and
\begin{equation}\label{arribajo}
|\mu(v)_{g\cdot p}-\mu(w)_{h\cdot p}|\le |g^{-1}\dot{g}-h^{-1}\dot{h}|.
\end{equation}
\item In that case $\mu$ is left-invariant for the  action of $G$ in $M$ given by the automorphisms $(\ell_h)_{*}:TM\to TM$, $h\in G$:
$$
\mu((\ell_h)_{*g\cdot p} (\pi_p)_{*g}\dot{g})_{hg\cdot p}=\mu((\pi_p)_{*g}\dot{g})_{g\cdot p}=\mu( (\pi_p)_{*1}(g^{-1}\dot{g}))_p
$$
for all $g\in G$, $\dot{g}\in T_gG$.
\end{enumerate}
\end{prop}
\begin{proof}
If  $g\cdot p=h\cdot p$ and $(\pi_p)_{*g}\dot{g}=(\pi_{x})_{*h}\dot{h}=v$, then by the previous lemma  $u=g^{-1}h\in K$ and $g^{-1}\dot{g}-g^{-1}\dot{h}h^{-1}g\in \lie(K)$. Hence for any $z\in \lie(K)$,
\begin{eqnarray}
\textrm{dist}_{\mu_g}(\dot{g},g\lie(K)) &\le & |\dot{g}-gz|_g=|\dot{g}u-gzu|_{gu}\nonumber\\
&= & |g(g^{-1}\dot{g}-z-\Ad_u(h^{-1}\dot{h})+uh^{-1}\dot{h}u^{-1})u|_h\nonumber\\
&=& |g(-z'+uh^{-1}\dot{h}u^{-1})u|_h = |guh^{-1}\dot{h}-gz'u|_h\nonumber\\
&=& |\dot{h}-hu^{-1}z'u|_h=|\dot{h}-hz''|_h\nonumber
\end{eqnarray}
where $z',z''\in \lie(K)$ are arbitrary (the first change of variable is by addition, and the second one is by conjugation, Remark \ref{adust}). Taking the infimum over $z''\in \lie(K)$ we obtain one inequality, and by symmetry we conclude that the quotient metric is well-defined. It is easy to check that each $\mu(\cdot)_{g\cdot p}$ is a Finsler norm on each $T_{g\cdot p}M$. 

Now assume that the metric $\mu$ in $G$ is left-invariant, the first assertion is clear. Choose a minimizing sequence $z_n\in \lie(K)$ such that $|h^{-1}\dot{h}-z_n|\to \mu_{h\cdot p}(w)$. Then
$$
\mu(v)_{g\cdot p}-|h^{-1}\dot{h}-z_n|  \le  |g^{-1}\dot{g}-z_n|-|h^{-1}\dot{h}-z_n| \le |g^{-1}\dot{g}-h^{-1}\dot{h}|.
$$
Letting $n\to \infty$, and doing the same reasoning with the opposite substraction, we obtain the equation (\ref{arribajo}). The last assertion is straightforward from the definitions and the identities of Lemma \ref{propac}.
\end{proof}

\medskip

Recall the definition of $L$-uniform (Definition \ref{sc}), that guarantees that left-multiplication is continuous for $\tau_{\mu}$. If guarantees that the quotient metric is continuous along smooth paths.
 
\begin{lem}\label{mabajo}
Assume that the norm $\mu:TG\to\mathbb R$ is $L$-uniform, then
\begin{enumerate}
\item The quotient Finsler metric $\mu:TM\to \mathbb R$ of Definition \ref{qfm} is continuous along paths, i.e. if $\gamma:[a,b]\to M$ is piecewise $C^1$, then $t\mapsto \mu(\dot{\gamma}_t)_{\gamma_t}$ is continuous.
\item If $\pi_*:TG\to TM$ is a quotient map, then $\mu:TM\to\mathbb R$ is continuous.
\item For the left action of $G$ on $M$ we have $\|(\ell_h)_{*g\cdot p}\|\le L(h)$ for all $g,h\in G$.
\end{enumerate} 
\end{lem}
\begin{proof}
Let $(x,w)=(h\cdot p, (\pi_p)_{*h}\dot{h})\in TM$ and $\varepsilon>0$. Let $z_n\in\lie(K)$ be such that $|\dot{h}-hz_n|_h<\mu(w)_x+\nicefrac{1}{n}$ for each $n\in\mathbb N$. Then for each $g\in G$, and $v=(\pi_p)_{*g}\dot{g}\in TM$ with $\dot{g}\in T_gG$, we have
\begin{eqnarray}
\mu(v)_{g\cdot p}-|\dot{h}-hz_n|_h & \le &  L(gh^{-1})\, |hg^{-1}\dot{g}-hz_n|_h-|\dot{h}-hz_n|_h \nonumber\\
& \le & L(gh^{-1}\,|hg^{-1}\dot{g}-\dot{h}|_h+|L_{gh^{-1}}|\, |hz_n-\dot{h}|_h - |\dot{h}-hz_n|_h \nonumber\\
&\le & L(gh^{-1}) L(h) |g^{-1}\dot{g}-h^{-1}\dot{h}|_1+(L(gh^{-1})-1)|\dot{h}-hz_n|_h.\nonumber
\end{eqnarray}
If we let $n\to\infty$, we obtain
$$
\mu(v)_{g\cdot p}-\mu(w)_x\le L(gh^{-1})L(h) |g^{-1}\dot{g}-h^{-1}\dot{h}|_1+(L(gh^{-1})-1) \mu(w)_x.
$$
Exchanging $v,w$ we obtain a similar bound; note we can use the inequality  just obtained to bound  $\mu(v)_{g\cdot p}$ in terms of $\mu(w)_x$ (which is fixed) and other terms. The difference $|\mu(v)_{g\cdot p}-\mu(w)_x|$ is then controlled by $|g^{-1}\dot{g}-h^{-1}\dot{h}|_1$ and $L(gh^{-1})-1,L(hg^{-1})-1$ (multiplied by some bounded terms). Then if $\gamma\subset M$ is piecewise smooth and $\Gamma\subset G$ is any piecewise smooth lift, for each $s,t\in [a,b]$ the quantity
$$
|\mu(\dot{\gamma}_t)_{\gamma_t}-\mu(\dot{\gamma}_s)_{\gamma_s}|
$$
is controlled by $|\Gamma_t^{-1}\dot{\Gamma}_t-\Gamma_s^{-1}\dot{\Gamma}_s|_1$, and $L(\Gamma_t\Gamma_s^{-1})-1,L(\Gamma_s\Gamma_t^{-1})-1$ (multiplied by some bounded terms). This proves that $\mu$ is continuous along piecewise $C^1$ paths.

Now assume that $\pi_*:TG\to TM$ is a quotient map, in particular open. Then for given $\varepsilon>0$ choose a neighborhood $Z\subset TG$ of $(h,\dot{h})$ such that 
$$
L(gh^{-1})-1, \; L(hg^{-1})-1, \;|g^{-1}\dot{g}-h^{-1}\dot{h}|<\delta
$$
if $(g,\dot{g})\in Z$ (by the continuity of the norm $\mu$ and the $C^1$ smoothness of the group operations).  Then if $(g\cdot p,(\pi_p)_{*g}\dot{g} )\in \pi_*(Z)$, we have that
$$
|\mu(v)_{g\cdot p}-\mu(w)_{h\cdot p}|<\varepsilon
$$
by adjusting $\delta$. This proves that $\mu:TM\to \mathbb R$ is continuous. To prove the final assertion, recall $\ell_h(g\cdot p)=(\pi_p\circ L_h)(g)$ thus if $v=(\pi_p)_{*g}\dot{g}\in T_{g\cdot p}M$, then
$$
(\ell_h)_{*g\cdot p}v=(\pi_p)_{*hg} (L_h)_{*g}\dot{g}
$$
by Lemma \ref{propac}, which in turn implies
\begin{eqnarray}
\mu( (\ell_h)_{*g\cdot p}v)_{hg\cdot p} & = &\inf\{  |(L_h)_{*g}\dot{g}-hgz|_{hg}:z\in \lie(K)\}\nonumber\\ 
& \le & L(h)\,\inf\{|\dot{g}-gz|_g:z\in \lie(K)\}=L(h)\mu(v)_{g\cdot p}.\qedhere
\nonumber
\end{eqnarray}
\end{proof}

\bigskip

Since in any case $\mu:TM\to\mathbb R$ is continuous along smooth paths, we now define the quotient length and metric. 

\begin{defi}[Quotient Length and Distance] 
Let $\mu:TM\to\mathbb R$ be the quotient Finsler metric of $\mu:TG\to\mathbb R$ induced by the action $\pi:G\times M\to M$ (Definition \ref{qfm}). Let $\gamma\subset M$ be a piecewise $C^1$ path, and define
$$
L_M(\gamma)=\int_0^1 \mu(\dot{\gamma}_t)_{\gamma_t} dt.
$$ 
We define accordingly the quotient metric $d_M(x,y)$ as the infima of the lengths of curves in $M$ joining $x,y\in M$. 
\end{defi}

\begin{rem}
When the metric of $G$ is left-invariant, these length and distance in $M$ are left-invariant for the action of the group $G$, that is $L_M(h\cdot\gamma)=L_M(\ell_h \gamma)=L_M(\gamma)$ and
\begin{equation}\label{elem}
d_M(h\cdot x,h\cdot y)=d_M(x,y)=\inf\{L_M(\gamma):\gamma(0)=x,\gamma(1)=y\}
\end{equation}
for any $x,y\in M$, $h\in G$ and $\gamma\subset M$. 

\smallskip

When the metric of $G$ is only $L$-uniform, then reasoning as in the proof of Proposition \ref{contigru} and using the bound $\|(\ell_h)_{*g\cdot p}\|\le L(h)$ of Lemma \ref{mabajo}, we obtain 
$$
L_M(\ell_h\gamma)\le L(h)L_M(\gamma) \textrm{ an } d_M(h\cdot x,h\cdot y)\le L(h)d_M(x,y).
$$
Thus in any case $G$ acts continuously in the metric space $(M,d_M)$.
\end{rem}

This section of the paper was mostly inspired by the work of Recht et al. \cite{dmlr,dmlr2}, where the authors introduce the notion of quotient tangent norm, to study the  metric geometry in homogeneous spaces of the unitary group of a $C^*$-algebra.

\subsection{Comparison of quotient distances}

Recall $\pi_p:G\to M$ has the path lifting property (Remark \ref{plp}), i.e. each piecewise smooth curve  $\gamma:[a,b]\to M$ starting at $x=g\cdot p$ can be lifted to a piecewise smooth $\Gamma:[a,b]\to G$ such that $\Gamma\cdot p=\gamma$ and $\Gamma(a)=g$. Note that by picking $z=0\in \lie(K)$ we have
\begin{equation} \label{lasnorm}
\mu(\dot{\gamma}_t)_{\gamma_t}  =  \inf\{ |\dot{\Gamma}_t- \Gamma_t z|_{\Gamma_t}: z\in \lie(K)\} \le   |\dot{\Gamma}_t|_{\Gamma_t}
\end{equation}
for each $t$, therefore 
\begin{equation}\label{desil}
L_M(\gamma)=L_M(\Gamma\cdot p)\le \LG(\Gamma) \textrm{ and } d_M(g\cdot p,h\cdot p)\le \dg(g,\Gamma(b)).
\end{equation}
As a corollary, we have a comparison of both quotient metrics on $M$:
\begin{lem}
Let $x,y\in M$, then $d_M(x,y)\le \dot{d}(x,y)$.
\end{lem}
\begin{proof}
Let $x=g\cdot p, y=h\cdot p$. Pick $u\in K$ and any smooth curve $\Gamma\subset G$ joining $g$ to $hu$, take $\gamma=\Gamma\cdot p$ that joins $g\cdot p$ to $h\cdot p$, therefore by inequality (\ref{desil}), $d_M(g\cdot p,h\cdot p)\le \dg(g,hu)=\dg(gu^{-1},h)$. Taking infimum over $u\in K$ and recalling (\ref{dpunto}) proves the assertion.
\end{proof}

Having established the relatively simple inequalities, we now work to  obtain reversed inequalities of lengths and metrics. We start with a definition.

\begin{defi}[Isometric and $\varepsilon$-isometric lifts] Let $\varepsilon\ge 0$, let $\gamma\subset M$ be piecewise smooth, we say that $\Gamma\subset G$ is an $\varepsilon$\textit{-isometric lift} of $\gamma$ if $\Gamma\cdot x=\gamma$ and $\LG(\Gamma)\le L_M(\gamma)+\varepsilon$. If $\varepsilon=0$, we say that $\Gamma$ is an \textit{isometric lift} of $\gamma$. 
\end{defi}

\begin{teo}\label{sigual}
If each $\gamma\subset M$ admits, for each $\varepsilon>0$, an $\varepsilon$-isometric lift $\Gamma_{\varepsilon}\subset G$, then $d_M=\dot{d}$ on $M$.
\end{teo}
\begin{proof}
It suffices to prove that $d_M(g\cdot p,h\cdot p)\ge \dot{d}(g\cdot p,h\cdot p)$ for any $g,h\in G$. To this end, let $\varepsilon>0$ and pick $\gamma\subset M$ such that $L_M(\gamma)\le d_M(g\cdot p,h\cdot p)+\varepsilon$, and take its $\varepsilon$-isometric lift $\Gamma_{\varepsilon}\subset G$ starting at $g\in G$. Since $\Gamma_{\varepsilon}(b)\cdot p=\gamma(b)=h\cdot p$, then $\Gamma_{\varepsilon}(b)=hu$ for some $u\in K$. Then
\begin{eqnarray}
\dot{d}(g\cdot p,h\cdot p) & = & \dg(gK,h)\le \dg(gu^{-1},h)=\dg(g,hu)\le \LG(\Gamma_{\varepsilon})\nonumber\\
& \le & L_M(\gamma)+\varepsilon\le d_M(g\cdot p,h\cdot p)+2\varepsilon.\qedhere
\nonumber
\end{eqnarray}
\end{proof}

\medskip

We now prove a theorem that establishes the existence of almost isometric lifts. As mentioned in Definition \ref{reguini}, the isotropy group $K$ is assumed to be regular, and the Finsler metric $\mu:TG\to\mathbb R$ is assumed to be $L$-uniform (Definition \ref{sc}).

\begin{teo}\label{igu}
Let $\gamma:[a,b]\to M$ be piecewise smooth, starting at $x=g\cdot p$. Then for each $\varepsilon>0$ there exists an $\varepsilon$-isometric lift $\Gamma_{\varepsilon}:[a,b]\to G$  of $\gamma$ with $\Gamma_{\varepsilon}(a)=g$.
\end{teo}
\begin{proof}
Let $\Lambda:[a,b]\to G$ be any piecewise smooth lift of $\gamma$, starting at $g\in G$. Let $C\ge 1$ be such that $\mu(\dot{\gamma}_t)_{\gamma_t}\le C$ for all $t\in [a,b]$. Given $\varepsilon>0$, pick  $r>0$ such that $ 3r +r (3r +C)<\varepsilon$. Let $\delta>0$ be such that 
$$
|\mu(\dot{\gamma}_t)_{\gamma_t}-\mu(\dot{\gamma}_s)_{\gamma_s}|<r,\quad 
|\Lambda_s\Lambda_t^{-1}\dot{\Lambda}_t-\dot{\Lambda}_s|_{\Lambda_s}<r,\quad \quad L(\Lambda_t\Lambda_s^{-1})-1<r
$$
if $|s-t|<\delta$. Take a $\delta$-partition $\pi=\{a=t_0<t_1<\dots<t_n=b\}$ of $[a,b]$, that is $t_{i+1}-t_i<\delta$. Denote $\Lambda_i=\Lambda(t_i)$, $\gamma_i=\gamma(t_i)$ etc. and for each $i$, pick $w_i\in \lie(K)$ such that 
$$
|\dot{\Lambda}_i-\Lambda_i w_i|_{\Lambda_i}< \mu(\dot{\gamma}_i)_{\gamma_i}+r.
$$
Let $w:[a,b]\to \lie(K)$ be a polygonal path joining $\{w_i\}_{i=0,\dots,n}$ in their given order, more precisely let $w|_{[t_i,t_{i+1}]}$ be denoted by $w_t=a_t w_i+ b_t w_{i+1}$, with $0\le a_t,b_t\le 1$ and $a_t+b_t=1$ (i.e. we write $a_t=(t_{i+1}-t)(t_{i+1}-t_i)^{-1}$ and $b_t=(t-t_i)(t_{i+1}-t_i)^{-1}$ for the coefficients of the convex combination).

Let $u:[a,b]\to K$ be a smooth lift of $-w$ with $u(a)=1$, that is $\dot{u}u^{-1}=-w$ (Remark \ref{umenos1}.2). Consider $\Gamma_{\varepsilon}=\Lambda u$. Clearly $\Gamma_{\varepsilon}\cdot p=\Lambda u \cdot p=\Lambda\cdot p=\gamma$ since $u\subset K$. Hence $\Gamma_{\varepsilon}$ is piecewise smooth lift of $\gamma$, and clearly $\Gamma_{\varepsilon}(a)=g$. Now differentiating we obtain $\dot{\Gamma}_{\varepsilon}=\dot{\Lambda}u+\Lambda\dot{u}$, and if $t\in (t_i,t_{i+1})$, by the right $K$-invariance of $\mu$,
\begin{align*}
|\dot{\Gamma}_{\varepsilon}(t)|_{\Gamma_{\varepsilon}(t)}  & = |\dot{\Lambda}_t u_t +\Lambda_t \dot{u}_t|_{\Lambda_t u_t} = |\dot{\Lambda}_t+ \Lambda_t \dot{u}_tu_t^{-1}|_{\Lambda_t}= |\dot{\Lambda}_t -\Lambda_t w_t|_{\Lambda_t}\\
& =  |(a_t+b_t) \dot{\Lambda}_t -\Lambda_t (a_t w_i+ b_t w_{i+1})|_{\Lambda_t}\le a_t| \dot{\Lambda}_t -\Lambda_t  w_i|_{\Lambda_t}+ b_t| \dot{\Lambda}_t -\Lambda_t  w_{i+1}|_{\Lambda_t}\\
 & \le a_t L(\Lambda_t\Lambda_i^{-1})|\Lambda_i\Lambda_t^{-1}\dot{\Lambda}_t-\Lambda_i w_i|_{\Lambda_i}+b_t L(\Lambda_t\Lambda_{i+1}^{-1})|\Lambda_{i+1}\Lambda_t^{-1}\dot{\Lambda}_t-\Lambda_{i+1}w_{i+1}|_{\Lambda_{i+1}}\\
 & \le (1+r) (a_t |\Lambda_i\Lambda_t^{-1}\dot{\Lambda}_t-\Lambda_i w_i|_{\Lambda_i}+b_t |\Lambda_{i+1}\Lambda_t^{-1}\dot{\Lambda}_t-\Lambda_{i+1}w_{i+1}|_{\Lambda_{i+1}}).
\end{align*}
We now compute
$$
|\Lambda_i\Lambda_t^{-1}\dot{\Lambda}_t-\Lambda_i w_i|_{\Lambda_i}\le |\Lambda_i\Lambda_t^{-1}\dot{\Lambda}_t-\dot{\Lambda_i}|_{\Lambda_i}+ |\dot{\Lambda}_i-\Lambda_i w_i|_{\Lambda_i}\le r + \mu(\dot{\gamma}_i)_{\gamma_i} +r\le 3r +\mu(\dot{\gamma}_t)_{\gamma_t},
$$
and likewise 
$$
|\Lambda_{i+1}\Lambda_t^{-1}\dot{\Lambda}_t-\Lambda_{i+1} w_{i+1}|_{\Lambda_{i+1}}\le 3r +  \mu(\dot{\gamma}_t)_{\gamma_t}.
$$
We also note that both terms are bounded by $3r+C$. Therefore
\begin{align*}
|\dot{\Gamma}_{\varepsilon}(t)|_{\Gamma_{\varepsilon}(t)} & \le a_t(3r +  \mu(\dot{\gamma}_t)_{\gamma_t})+ b_t (3r +  \mu(\dot{\gamma}_t)_{\gamma_t})+ r a_t(3r +C) +r b_t(3r +C)\\
& = 3r + \mu(\dot{\gamma}_t)_{\gamma_t} + r (3r +C)=\mu(\dot{\gamma}_t)_{\gamma_t} + 3r +r (3r +C).
\end{align*}
Thus $|\dot{\Gamma}_{\varepsilon}(t)|_{\Gamma_{\varepsilon}(t)}<\varepsilon +\mu(\dot{\gamma_t})_{\gamma_t}$ for $t\in [t_i,t_{i+1}]$. Integrating in that interval, we obtain
$$
\int_{t_i}^{t_{i+1}} |\dot{\Gamma_{\varepsilon}}(t)|_{\Gamma_{\varepsilon}(t)}dt < \varepsilon(t_{i+1}-t_i)+ \int_{t_i}^{t_{i+1}} \mu(\dot{\gamma}_t)_{\gamma_t}dt,
$$
and summing over $i$, we arrive to $\LG(\Gamma_{\varepsilon})< \varepsilon(b-a) +L_M(\gamma)$.
\end{proof}

As a corollary of the existence of almost isometric lifts, we can now state the main theorem of this section.

\begin{teo}\label{elcorodelq}
Let $G,K, M\simeq G/K$ be as in Definition \ref{smoothm} and $\mu$ a $L$-uniform Finsler metric in $G$. Then 
\begin{enumerate}
\item $d_M(x,y)= \dot{d}(x,y)$ for all $x,y\in M$; equivalently if $K=K_{p}$ then
\item $d_M(g\cdot p,h\cdot p)=\dg(g,hK)$ for any $g,h\in G$.
\item If $(G,\dg)$ is complete, then $(M,d_M)$ is complete. 
\item If $\dg$ is a metric in $G$ (i.e. non-degenerate), then $d_M$ is a metric in $M$.
\item If $\pi_*:TG\to TM$ is quotient, then $\mu:TM\to\mathbb R$  and $d_M:M\times M\to\mathbb R$ are continuous for the original manifold topology $\tau_M$.
\end{enumerate}
\end{teo}
\begin{proof}
The first two items follow from  Theorems \ref{sigual} and \ref{igu}.  The third assertion then follows from Theorem \ref{comple}. The fourth assertion is immediate from the right invariance for $K$. The final assertion follows from Lemma \ref{arribajo}, combined with the fact $d_M=\dot{d}$.
\end{proof}

\subsection{Geodesics}\label{geodesics}

Recall that a lift $\Gamma\subset G$ of $\gamma\subset M$ is \textit{isometric} if $\LG(\Gamma)=L_M(\gamma)$. In this section we examine geodesics of $TM$ (minimizing paths for the quotient distance $d_M$), which can be considered as the image of certain special geodesics in $TG$: minimizing paths in $G$ that are transverse to $K$.

\begin{defi}[Minimal lifts]\label{minilift}
We say that a vector $\dot{g}\in T_gG$ is \textit{minimal} if 
$$
|\dot{g}|_g=\textrm{dist}_{\mu_g}(\dot{g},g\lie(K))=\mathop{\inf}_{z\in \lie(K)}|\dot{g}-gz|_g.
$$
That is, minimal vectors in $TG$ are isometric lifts of vectors in $TM$.
\end{defi}

\begin{rem}
Note that there might not exists minimal lifts, and if $|\cdot|_g$ is not strictly convex, they might not be unique. In the Riemannian (or smooth Finsler setting), minimal lifts are usually called \textit{horizontal lifts}, see Durán and Paiva \cite{paiva}; we prefer to keep the term minimal since horizontal distributions are rare.
\end{rem}

\medskip

We now study the relationship among short paths in $G$ and in $M\simeq G/K$. 

\begin{teo}\label{arribajo2}
Let $g,h\in G$, and let $\Gamma:[a,b]\to G$ joining them. The following are equivalent
\begin{enumerate}
\item $\LG(\Gamma)=\dg(g,hK)$ (in particular $\Gamma$ is short in $G$ and the distance from $g$ to $hK$ is realized).
\item $\gamma=\Gamma\cdot p\subset M$ is a short path joining  $g\cdot p,h\cdot p$, and $\Gamma$ is an isometric lift of $\gamma$.
\end{enumerate}
\end{teo}
\begin{proof}
Let us prove that $1\Rightarrow 2$. Let $\beta\subset M$ be any other path joining $g\cdot p,h\cdot p$, let $\varepsilon>0$ and take an $\varepsilon$-isometric lift $\Lambda_{\varepsilon}\subset G$ of $\beta$  starting at $g$ (Theorem \ref{igu}). Then $\Lambda_{\varepsilon}(b)=hu$ for some $u\in K$, therefore
$$
L_M(\gamma)\le \LG(\Gamma)=\dg(g,hK)\le \dg(g,hu)\le \LG(\Lambda_{\varepsilon})<L_M(\beta)+\varepsilon.
$$
This proves that $L_M(\gamma)\le L_M(\beta)$, therefore $\gamma$ is short. Now
$$
\LG(\Gamma)=\dg(g,hK)=\dot{d}(g\cdot p,h\cdot p)=d_M(g\cdot p,h\cdot p)=L_M(\gamma)
$$
because $\gamma$ is short in $M$, therefore $\Gamma$ is an isometric lift of $\gamma$.

We now prove that $2\Rightarrow 1$. Let $\varepsilon>0$ and pick $u\in K$ such that $\dg(g,hu)<\dg(g,hK)+\varepsilon$. Pick a path $\Lambda\subset G$ joining $g,hu$ such that $\LG(\Lambda)<\dg(g,hu)+\varepsilon$. Since $\Lambda(a)\cdot p=g\cdot p=\gamma(a)$ and $\Lambda(b)\cdot p=hu\cdot p=h\cdot p=\gamma(b)$,
$\Lambda\cdot p$ has the same endpoints than $\gamma$ in $M$, therefore
$$
\LG(\Gamma)=L_M(\gamma)\le L_M(\Lambda\cdot p)\le \LG(\Lambda)
$$
since $\gamma$ is short and we have inequality (\ref{lasnorm}). Then
$$
\LG(\Gamma)\le \LG(\Lambda)<\dg(g,hu)+\varepsilon<\dg(g,hK)+2\varepsilon.
$$
This proves that $\LG(\Gamma)\le\dg(g,h K)$, ant the other inequality  always holds.
\end{proof}

\begin{rem}\label{necemin}
If $\Gamma\subset G$ is as in the previous theorem, then \textit{$\Gamma$ is an isometric lift of $\gamma=\Gamma\cdot p$ restricted to any sub-interval $I\subset [a,b]$, and $\dot{\Gamma_t}\in T_{\Gamma_t}G$ is  a minimal vector for all $t\in [a,b]$}.  To prove it, it suffices to observe that
$$
\dg(\Gamma_t,h)=\dg(\Gamma_t,hK)\quad \textrm{ for all }t\in [a,b].
$$
Otherwise it would exists $u_0\in K$ such that $\dg(\Gamma_t,hu_0)<\dg(\Gamma_t,h)$. Now being minimizing, distance is additive along $\Gamma$, therefore
\begin{eqnarray}
\dg(g,hK) &\le & \dg(g,hu_0)\le \dg(g,\Gamma_t)+ \dg(\Gamma_t,hu_0)\nonumber\\
&<& \dg(g,\Gamma_t)+\dg(\Gamma_t,h)=\dg(g,h)=\dg(g,hK)\nonumber
\end{eqnarray}
a contradiction. Then 
$$
\LG(\Gamma)_t^1=\dg(\Gamma_t,h)=\dg(\Gamma_t,hK)=d_M(\gamma_t,h\cdot p)=L_M(\gamma)_t^1
$$
for all $t\in [a,b]$. Differentiating with respect to $t$ we obtain the claim
$$
|\dot{\Gamma}_t|_{\Gamma_t}=\mu(\dot{\gamma}_t)_{\gamma_t}=\mathop{\inf}_{ z\in \lie(K)} |\dot{\Gamma}_t-\Gamma_tz|_{\Gamma_t}\quad \forall \,t\in [a,b].
$$
\end{rem}

\begin{rem}\label{geococ}
For a minimizing path $\Gamma$ of $(G,\dg)$, it would be natural to conjecture that if $\dot{\Gamma}$ is a minimal vector along $\Gamma$, then $\gamma=\Gamma\cdot p$ is (at least locally) a minimizing path of $(M,d_M)$. As mentioned in the introduction of this paper, for Riemannian (or smooth Finsler metrics) this follows from the Euler equations (see for instance \cite[Theorem 3.1]{paiva}).

\medskip

$\S$ For certain actions of the unitary group $G=\mathcal U(H)$ of a Hilbert space (flag manifolds) this last minimality property was proved by Recht et. al in \cite{dmlr} -for the Finsler metric induced by the uniform norm-. Then in \cite{alr} we proved it for the $p$-Schatten norms of compact operators, and in \cite{acl} it was proved for the $p$-norms of finite von Neumann algebras.
\end{rem}

\section{Bi-invariant metrics}\label{biinvme}

In this section we take a closer look at the geometry of a Lie group $K\subset G$ with a metric which is  invariant for the adjoint action of $K$; the metric considered in this section is the left-invariant metric induced by the tangent $\Ad_K$-invariant Finsler metric (therefore it is a bi-invariant metric in $K$). Typically, one could consider these groups $K$ as generalizations of isometry groups. Their Lie algebras are analogous of compact Lie algebras in the finite dimensional setting, and are called \textit{elliptic} (see \cite[p. 344]{neeblie}).

\begin{rem}[Bi-invariant intrinsic metrics in locally compact groups]
It is worth recalling here that due to a result from Berestovskii, any locally compact, locally contractible topological group $K$ with a bi-invariant intrinsic metric $d$ (i.e. a $(K,d)$ is a \textit{length space}), is a Lie group with a bi-invariant Finsler metric (see \cite[Theorem 7]{beres}). Thus the results of this section imply that one-parameter groups are short paths for the intrisic metric of $K$.
\end{rem}

In what's left of this section we assume that the Lie group $K$ has a smooth (at least $C^2$) exponential map. Thus for fixed $v\in\lie(G)$, $w\in \lie(K)$, the path $f(s)=e^{s \ad w}v=\Ad_{e^{sw}}v$ is smooth.
\medskip

$\S$ In classical Finsler geometry, where the norm $\mu$ is twice differentiable and its Hessian $g_{\mu}=D^2\mu^2$ is positive definite, there is a nice characterization of \textit{geodesic vectors}, that is, vectors such that $t\mapsto e^{tw}$ is a geodesic of the auxiliary metric $g_{\mu}$ of the Finsler metric in $G$. It was obtained by Latifi \cite{latifi}, and it states that $w\in\lie(G)$ is geodesic if and only if $g_{\mu}(\ad w(v),v)=0$ for all $v\in \lie(G)$. It generalizes the well-known fact that for a Riemannian bi-invariant metric on a Lie group,  $\ad w$ is skew-adjoint for the metric $g_{\mu}$ (see Milnor \cite[Section 7]{milcur}). In what follows we extend this result to bi-invariant semi-norms.

\begin{defi}[Norming functionals]\label{adop} We denote $E_{\mu}=(\lie(G),\mu)$  the linear space $\lie(G)$ with the topology induced by $\mu$. When $|\cdot|$ is degenerate, we denote $\tilde{v}$ to the class of $v$ in the quotient $\lie(G)/N$ by the closed subspace $N=\{z\in\lie(G): |z|=0\}$; then $\lie(G)/N$ has a non-degenerate quotient Finsler norm. When $\varphi$ is $\mathbb R$-linear and continuous in $E_{\mu}$ (that is, there exists $c>0$ such that $|\varphi(z)|\le c \mu(z)$ for all $z\in \lie(G)$), it must be $\varphi|_N=0$.

Let us denote
$$
E_{\mu}^*=\{\varphi:\lie(G)\to\mathbb R\,  \textrm{ s.t. } \varphi|_N\equiv 0, \, \varphi \textrm{ is linear and continuous for }\mu\}.
$$
We also we denote $\|\varphi\|=\sup\{|\varphi(v)|: |v|=1\}$ for $\varphi\in E_{\mu}^*$. We say that $\varphi\in E_{\mu}^*$ is a \textit{norming functional} of $v\in\lie(G)$ if $\varphi(v)=|v|$; by the Hahn-Banach theorem each vector $v$ admits at least one \textit{unit norm} norming functional, we shall denote it by $\varphi_v$. 
\end{defi}

Since $|v|\le |v-z|+|z|=|v-z|\le |v|+|z|=|z|$ when $z\in N$, it is apparent that $|v|=\|\tilde{v}\|:=\inf\{|v-z|:z\in N\}$ and there is a natural isometric identification
$$
\left(\lie(G)/N\right)^*\simeq E_{\mu}^*,
$$
given by $\varphi(v)=\tilde{\varphi}(\tilde{v})$. 

\begin{prop}\label{diss}
Let $\mu=|\cdot|$ be an $\Ad_K$-invariant semi-norm in $\lie(G)$, let $v\in \lie(G)$. Then
\begin{enumerate}
\item If $w\in \lie(K)$ and $v\in N$, then $\ad w(v)=[w,v]\in N$.
\item $N\cap\lie(K)$ is a closed stable ideal of $\lie(K)$.
\item  If $w\in \lie(K)$ then for each norming functional $\varphi$ of $v$ (or of $-v$), we have $\varphi([w,v])=0$. Equivalently, $\varphi \circ \ad v\equiv 0$ on $\lie(K)$ when $\varphi(\pm v)=\pm |v|$.
\end{enumerate}
\end{prop}
\begin{proof}
 If $w\in \lie(K)$ then $|e^{\ad w}v|=|v|$ for all $v\in \lie(G)$, in particular if $v\in N$ this shows that $N$ is stable. Moreover $\alpha(t)=e^{t\ad w}v\subset N$ therefore $e^{\ad w}(v)\in N$ and also $[w,v]=\ad w(v)=\dot{\alpha}(0)\in N$, and this proves the first assertion. The second assertion is immediate from the first one. Consider the expansion in $\lie(G)$ 
$$
\Ad_{e^{sw}}v=e^{s \ad  w}v=v+s \, \ad  w(v)+o(s^2),
$$
which gives $e^{s \ad  w} v-v=s[w,v]+o(s^2)$. For each (unit norm) norming functional $\varphi$ of $v$
$$
\varphi(e^{s\ad w}v)-\varphi(v)\le |e^{s\ad w}v|-|v|= 0.
$$
Note that the convergence of the series is in the original topology of the lie algebra of $K$, thus the assumption on the continuity of $\mu$ (and hence of any $\varphi\in E_{\mu}^*$), allows the folllowing: divide by $s>0$ and make $s\to 0^+$, to obtain $\varphi([w,v])\le 0$. Replacing $w$ with $-w$ proves that $\varphi([w,v])=0$. The proof when $-\varphi(v)=|-v|$ is similar and therefore omitted.
\end{proof}

\begin{rem}
If $\varphi$ is a norming functional for $v\in\lie(G)$ and $w\in\lie(K)$, then 
$$
|v|=\varphi(v)=\varphi(v-\ad(w) v)\le |v-\ad (w)(v)|.
$$
In particular $1-\ad w$ is injective for each $w\in \lie(K)$. 
\end{rem}

The next little lemma contains essentially the same information as the Gauss' Lemma of Riemannian geometry: the differential of the exponential map along a geodesic preserves angles with the geodesic speed vector.

\begin{lem}[Gauss' Lemma]\label{normiexp}
Let $v,w\in\lie(K)$ and  $\varphi$ be a norming functional for $w$ of $-w$. Then
\begin{enumerate}
\item  $\varphi(e^{\lambda \ad  w}v)=\varphi(v)$ for all $\lambda\in\mathbb R$,
\item  $\varphi( e^{-w}\exp_{*w} v)=\varphi(v)$.
\end{enumerate}
\end{lem}
\begin{proof}
The functional $\varphi$ is also continuous for the locally convex topology of $\lie(G)$ (since $\mu$ is continuous), therefore the map $f(\lambda)=\varphi(e^{\lambda \ad  w}v)$ is smooth (Remark \ref{noad}). Moreover  $f'(\lambda)=\varphi([ w,e^{\lambda \ad  w}v])=0$ by the previous proposition. Since $f(0)=\varphi(v)$ we obtain the first assertion. Using Remark \ref{difexp}, we have
\begin{eqnarray}
\varphi( e^{-w}\exp_{*w} v) & = & \varphi\left(\int_0^1 e^{-s\,\ad  w}v ds\right)\nonumber\\
&= &\int_0^1 \varphi(e^{-s \,\ad  w}v)ds=\int_0^1 \varphi(v)=\varphi(v).\qedhere
\nonumber
\end{eqnarray}
\end{proof}

Since they will play a fundamental role in bi-invariant metrics, let's define segments and polygonal paths:

\begin{defi}\label{defiseg}
A path $\delta$ in $K$ is a \textit{segment} if it is a left translation of a one-parameter group, i.e. $\delta(t)=ue^{tz}$. We say that $\delta$ \textit{polygonal path} if it is a continuous concatenation of segments.
\end{defi}

Now we can establish the main result of this section. 

\begin{lem}\label{gauss} Let $w:[a,b]\to \lie(K)$ be a piecewise $C^1$ path, let $\gamma=e^w$. Then
$$
|w(b)|-|w(a)|\le \int_a^b |\gamma_t^{-1}\dot{\gamma}_t|\, dt =\LG(\gamma).
$$
\end{lem}
\begin{proof}
We start by assuming that the path $w$ is of class $C^2$. Note that for each $t$ we have  $\gamma^{-1}(t)\dot{\gamma}(t)=e^{-w(t)}\exp_{*w(t)}\dot{w}(t)$. Denote $w_t=w(t)$, $\dot{w}_t=\dot{w}(t)$, and for each $t\in [a,b]$ pick $\varphi_t$ a unit norm, norming functional of $w_t$, i.e. $\|\varphi_t\|=1$, $\varphi_t(w_t)=|w_t|$ (if $|w_t|=0$ pick $\varphi_t=0$). Then by the previous lemma
\begin{eqnarray}\label{cotaa}
 \varphi_t(\dot{w}_t) &= & \varphi_t(e^{-w_t}\exp_{*w_t}\dot{w}_t) \le |e^{-w_t}\exp_{*w_t}\dot{w}_t|=|\gamma_t^{-1}\dot{\gamma}_t|.
\end{eqnarray}
Now we pick a partition $\pi=\{t_0<t_1<\dots<t_n\}$ of the interval $[a,b]$, and denote $w_i=w(t_i)$, etc. Since $w_t$ is of class $C^2$, 
$$
w_i=w_{i-1}+\Delta_i \dot{w}_{i} +o(\Delta_i^2),
$$
where $\Delta_i=t_i-t_{i-1}$. Hence 
$$
|w_{i-1}|\ge \varphi_i(w_{i-1})=\varphi_i(w_i)-\varphi_i(\dot{w}_i\Delta_i+o(\Delta_i^2))
=|w_i|-\varphi_i(\dot{w}_i\Delta_i+o(\Delta_i^2)).
$$
We now compute a telescopic sum, 
\begin{eqnarray}
|w_b|-|w_a| &= &\sum_i |w_i|-|w_{i-1}|\le \sum_i \varphi_i(\dot{w}_i)\Delta_i+  \varphi_i(o(\Delta_i^2))\nonumber \\
& \le & \sum_i |\gamma_{t_i}^{-1}\dot{\gamma}_{t_i}|\Delta_i+  o(\Delta_i^2)\nonumber
\end{eqnarray}
by equation (\ref{cotaa}) and the fact that $\|\varphi_i\|=1$. Refining the partition, it follows that
$$
|w(b)|-|w(a)|\le \int_a^b |\gamma_t^{-1}\dot{\gamma}_t| dt = \LG(\gamma).
$$
Clearly the inequality holds if $w$ is only piecewise $C^2$. If $w$ is only piecewise $C^1$, take $\varepsilon>0$ and approximate $\gamma=e^w$ with a polygonal path $P\subset K$ such that $|P^{-1}_t\dot{P}_t-\gamma_t^{-1}\dot{\gamma}_t|<\varepsilon$ for all $t$. Since $P$ is piecewise $C^2$ we obtain
$$
|w(b)|-|w(a)|\le \int_a^b |P_t^{-1}\dot{P_t}| dt \le \LG(\gamma)+\varepsilon(b-a).
$$
Letting $\varepsilon\to 0^+$ proves the assertion for the path $w$.
\end{proof}

\subsection{Locally exponential Lie groups, normed neighbourhood}\label{loca}

In this section and the next one we will assume that the exponential map of $K$ is a local diffeomorphism for \textit{some open ball of the given Finsler semi-norm $\mu$}. Adapting the technique from Riemannian geometry we obtain local minimality of segments.

\begin{defi}\label{localymuexp}
Let $K$ be a Lie group, $\mu=|\cdot|$ a bi-invariant Finlser semi-norm in $\lie(K)$, and for $r>0$ let 
$$
B_r=\{v\in \lie(K): |v|<r\}, \quad V_R=\exp(B_R).
$$
We assume in this section that \textit{there exists $R>0$ such that $V_R$ is $\tau_K$ open in $K$, and $\exp:B_R\to V_R$ is a $C^1$-diffeomorphism}.
\end{defi}

\begin{rem}
Actually, it will suffice to assume that $\exp|_{B_R}$ is a bijection with its image $V_R$, and every piecewise $C^1$ path $\gamma\subset V_R$ starting at $g=1$ has a unique piecewise $C^1$ lift $\Gamma\subset B_R$ starting at $v=0$ (that is $e^{\Gamma}=\gamma$). If $\lie(K)$ is a Fréchet space and $\exp^{-1}|_{V_R}$ maps smooth curves to smooth curves, then $\exp|_{B_R}$ is a diffeomorphism (see \cite[Remark II.2.13.(b)]{neebmss}). For other locally convex spaces, the  weaker hypothesis can be useful.
\end{rem}

\begin{rem}
The straightforward example of our definition is given by a Banach-Lie group $K$ and a bi-invariant norm $\mu=|\cdot|$ that is equivalent to the original norm modeling the Banach space $\lie(K)$. Then the radius $R>0$ is given by the fact that $\exp$ is a local diffeomorphism.
\end{rem}

\begin{teo}[Local minimality of segments]\label{mini}
Let $u_0,u_1=u_0e^z\in K$ with $|z|<R$.
\begin{enumerate}
\item  Let $\gamma$ be a piecewise $C^1$ path joining $u_0,u_1$. If $\gamma$ leaves $u_0V_R$, then $\LG(\gamma)\ge R$. 
\item If $\delta(t)=u_0e^{tz}$, $t\in [0,1]$, then $\delta$ is shorter that any other piecewise $C^1$ path $\gamma$ in $K$ joining $u_0,u_1$ and $\dg(u_0,u_1)=|z|$. 
\end{enumerate}
\end{teo}
\begin{proof}
By the left invariance of the metric it suffices to consider $u_0=1$, $u_1=u=e^{z}\in K$. Let $\gamma\subset K$ be a piecewise $C^1$ path with $\gamma(0)=1,\gamma(1)=u$. Then $I=\gamma^{-1}(V_R)$ is open and since $\gamma$ leaves $V_R$ there exists $0<t_0\le 1$ such that $[0,t_0)\subset [0,1]$ is the connected component containing $0\in I$. Clearly $\gamma|_{[0,t_0)}\subset V_R$ therefore the hypothesis in Definition \ref{localymuexp} gives us a path $\Gamma:[0,t_0)\to \lie(K)$ which is piecewise $C^1$, and such that $e^{\Gamma}=\gamma|_{[0,t_0)}$. For each $n\in\mathbb N$ there exists $0<t_n<t_0$ such that $|\Gamma_{t_n}|\ge R-\nicefrac{1}{n}$ (otherwise $\Gamma_{t_0}\in B_R$ and then $\gamma(t_0)\in V_R$; by the continuity of $\gamma$, the component would be strictly larger). Now differentiating  $e^{-\Gamma}\exp_{*\Gamma}\dot{\Gamma}=\gamma^{-1}\dot{\gamma}$ and by Lemma \ref{gauss} we obtain
$$
R-\nicefrac{1}{n}\le |\Gamma_{t_0}|\le \int_0^{t_0} |\gamma^{-1}\dot{\gamma}|\le \LG(\gamma).
$$
Letting $n\to +\infty$ proves the first assertion of the theorem.  Since $\LG(\delta)=|z|=r<R$, we can now assume that $\gamma$ does not leave $V_R$. Let $w$ be a piecewise $C^1$ lift of $\gamma$ to $B_R\subset \lie(K)$, $w(0)=0$. Again by Lemma \ref{gauss}, recalling $w(1)=z$ by the injectivity of $\exp|_{B_R}$, we get
$$
|z|-|0|\le \int_0^1 |\gamma_t^{-1}\dot{\gamma}_t| dt \le \LG(\gamma),
$$
and this proves the second assertion of the theorem $\LG(\delta)=|z|\le \LG(\gamma)$.
\end{proof}

By a standard approximation argument, it is clear that

\begin{coro}
If $u_1=u_0e^z$ with $|z|\le R$, the path $\delta(t)=u_0e^{tz}$ is shorter than any other piecewise smooth path in $K$ joining them.
\end{coro}

Covering the image of a path $\gamma$ with finite translations of $V_R$, we obtain
\begin{coro}\label{poly}
If $\gamma$ is a piecewise smooth path joining given endpoints in $K$, then $\LG(\gamma)$ is larger than a \textit{certain} polygonal path joining the same endpoints, 
$$
\LG(\gamma)\ge \sum_{i=1}^n |v_i|,\qquad v_i\in B_R.
$$
\end{coro}

\begin{coro}
If $u_0,u\in K$ and $\dg(u_0,u_1)<R$, then there exists $z\in \lie(K)$ with $|z|<R$ such that $u=u_0e^z$ therefore $|z|=\dg(u_0,u)=\LG(\delta)$, where $\delta$ is the segment generated by $z$. In particular 
$$
\exp(B_r)=V_r=\{u\in K: \dg(u,1)<r\}
$$
holds for all $r\le R$.
\end{coro}
\begin{proof}
As usual, we can assume that $u_0=1$. Let $\gamma$ be a piecewise smooth path joining $1,u$ with $\LG(\gamma)\le \dg(1,u)+\frac{\varepsilon}{2}=R-\frac{\varepsilon}{2}<R$ (hence we can assume that $\gamma$ does not leave $V_R$). In particular $\gamma(1)=e^z$ with $|z|<R$. The minimality theorem (Theorem \ref{mini}) settles the proof.
\end{proof}

\paragraph{Metrization and uniform structure} If $\lie(K)$ is a Fréchet space, let $(|\cdot|_n)_{n\in\mathbb N}$ be a directed family of semi-norms defining the topology of $\lie(K)$, with 
$$
|\cdot|_1\le |\cdot|_2\le \dots \le |\cdot|_n\le |\cdot|_{n+1}\le \dots
$$
Let us denote $B_R^{n}=\{z\in\lie(K): |z|_n<R\}$ to the open ball of each semi-norm,  $B_R^{n+1}\subset B_R^n\subset \dots \subset B_R^1$, and denote $V_R^n=\exp(B_R^n)$. Let us denote $d_n$ to the pseudo-distance induced by $|\cdot|_n$ in $K$. Let 
$$
d_K(g,h)=\sum_n \frac{1}{2^n} \frac{d_n(g,h)}{1+d_n(g,h)}
$$
\textit{Assume $\exp|_{B_R^1}$ is a diffeomorphism onto $V=V_R^1$}. Then the same holds for all $n\in\mathbb N$. The minimality theorem implies that \textit{the manifold topology of $K$ is the same than the topology induced by the metric $d_K$}. 

\subsubsection{Uniqueness and the EMD property}

For strictly convex Finsler norms we now establish the uniqueness of segments as minimizing paths; note that if $\mu$ is strictly convex then in particular it is non-degenerate therefore $K_{\mu}=\{1\}$ (see Lemma \ref{nilel} below) and $\dg$ is also non-degenerate. 

\begin{teo}[Uniqueness for strictly convex norms]\label{uniestricta}
Assume that $\mu=|\cdot|$ is strictly convex. Let $u_0,u_1\in K$ and $\gamma:[a,b]\to K$ be a short rectifiable path joining $u_0,u_1$. Then there exists $z\in \lie(K)$ such that $u_1=u_0e^z$ and $\gamma$ is a reparametrization of the segment $\delta(t)=u_0e^{tz}$. If $\dg(u_0,u_1)<R$, then this segment is unique.
\end{teo}
\begin{proof}
We can assume $u_0=1,u_1=u$. Let $\gamma$ be a smooth path joining $1,u$ with $\LG(\gamma)=\dg(1,u)$. Assume first that $\LG(\gamma)<R$, then there exists a unique $z$ such that $u=e^z$ and $|z|=\dg(1,u)<R$. We can assume that $\gamma$ does not leave $V_R$, so if $t_1\in (a,b)$, then $\gamma(t_1)=e^x$ with $|x|=\dg(1,\gamma(t_1))$. This is because $\gamma$, being short, is also short along its path (and so is the segment). Since $\dg(\gamma(t_1),\gamma(1))<R$ there also exists $y$ with $|y|<R$ and $\gamma(1)=\gamma(t_1)e^y$ and $|y|=\dg(\gamma(t_1),1)$. Hence $\gamma(1)=e^z=e^xe^y$ and
$$
\dg(1,u)\le |x|+|y|= \LG(\gamma|_{[a,t_1]})+\LG(\gamma_{[t_1,1})=\LG(\gamma)=\dg(1,u).
$$
Consider the auxiliary path $\beta$ given by $\beta(t)=e^{tx}e^{ty}$. Clearly $\beta$ joins $1,u$ in $K$. Now
$$
\beta_t^{-1}\dot{\beta}_t=e^{-t \ad  y}(x+y)
$$
therefore
$$
\dg(1,u)  \le  \LG(\beta)=|x+y|\le |x|+|y| =\dg(1,u).
$$
This implies $|x+y|=|x|+|y|$, and since $|\cdot|$ is strictly convex, it must be $y=\lambda x$ with $\lambda=|y|/|x|$. Then
$$
u=e^z=e^xe^y=e^xe^{\lambda x}=e^{(1+\lambda)x},
$$
with $|(1+\lambda)x|= |x|+|y|=\dg(1,u)=|z|<R$. By the injectivity of the exponential map in $B_R$, it must be $(1+\lambda)x=z$, hence $x=sz$ with $0<s=(1+\lambda)^{-1}<1$ and $\gamma(t_1)=e^x=e^{sz}$, thus $\gamma$ is a reparametrization of $\delta(t)=e^{tz}$. This proves the local result. Now assume $\gamma$ is any short path joining given endpoints in $K$, again we can assume $\gamma(a)=1$. Partitioning $\gamma$ in small pieces, each of these of length strictly shorter than $R$, we conclude that $\gamma$ is a (reparametrization of) a polygonal path. If we call $x_i$ to the speeds of each segment, it suffices to show that there exists $\mu_i>0$ such that $x_{i+1}=\mu_i x_i$ for each $i$, to conclude that $\gamma$ is a segment. If $t_1\in [a,b]$ is the first cusp, after a reparametrization of $\gamma$, we can assume that $\gamma|_{[t_1-\varepsilon,t_1+\varepsilon]}$ (wich is still short) is the segment $t\mapsto e^{t  x_1}$, followed by the segment $t\mapsto e^{x_1} e^{ (t-1) x_2}$, parametrized in the interval $[0,2]$. Renaming $a=0,t_1=1, b=2$, we are in the previous situation where $\gamma|_{[a,b]}$ is a short path joining its endpoints, of length strictly shorter than $R$. Hence the proof above tells us that $x_2=\mu_1 x_1$ for some $\mu_1>0$. Iterating this argument at each vertex, we conclude that $\gamma$ is a segment.
\end{proof}

\begin{rem}\label{atta}
The previous theorem shows than in the case of strictly convex Finsler norms, there  exists a short path $\gamma$ joining $1,u$ in $K$ only if $u=e^z$ is in the range of the exponential map, and $\gamma$ is then a segment. Moreover, if $u$ is close to $1$, this segment is unique. It is unclear however what is the maximal neighborhood of $0\in\lie(K)$ where the segments are short paths, though one would expect that this would be the case when the exponential map is a diffeomorphism along the segment (a set which can be much larger than the $\mu$-ball of radius $R$).
\end{rem}

We now compare the distance in the manifold with the tangent distance and give criteria for equality to hold. 

\begin{teo}[The Exponential Metric Decreasing (EMD) property]\label{emd}
 If $v,w\in \lie(K)$ then 
$$
\dg(e^v,e^w)\le |w-v|.
$$
\begin{enumerate}
\item If $w,v$ commute and $|w-v|\le R$, then equality holds.
\item If equality holds and the norm is strictly convex, then $w,v$ commute.
\end{enumerate}
\end{teo}
\begin{proof}
Consider the path $\beta(t)=e^{tw}e^{-tv}$ with length $|w-v|$. Clearly $\beta$ joins $1,e^we^{-v}$ and $\beta_t^{-1}\dot{\beta}_t=e^{t\ad v}(w-v)$, therefore
$$
\dg(e^v,e^w)=\dg(1,e^we^{-v})\le \LG(\beta)=|w-v|.
$$
If $w,v$ commute then $e^{tw}e^{-tw}=e^{t(w-v)}$ which is short with the additional hypothesis that $|w-v|\le R$, and this proves the first assertion.

If equality holds and the norm is strictly convex, by the previous theorem there exists a reparametrization $f:[0,1]\to [0,1]$ and $z\in \lie(K)$ such that $\delta(t)=e^{f(t)z}=\beta(t)$. For small $t$, we have $f(t)z=\exp^{-1}(e^{tw}e^{-tv})$. Taking a norming functional $\psi$ for $z$ we have $f(t)|z|=\psi(\exp^{-1}(e^{tw}e^{-tv}))$, therefore $f$ is smooth for small $t$. Differentiating $\delta=\beta$ we have $e^{f(t)z}f'(t)z=e^{tw}(w-v)e^{-tv}$ therefore $f'(t)z=e^{t \ad v}(w-v)$ for small $t$. Applying a norming functional $\varphi$ for $w-v$, we have 
$$
f'(t)\varphi(z)=\varphi(w-v)=|w-v|
$$
by Lemma \ref{normiexp}. Thus $f(t)=at$ and $az=e^{t\ad v}(w-v)$ for small $t$. Differentiating at $t=0$ we obtain $[v,w]=\ad v(w)=\ad v(w-v)=0$ as claimed.
\end{proof}

\begin{rem}For manifolds of non-positive curvature, one obtains a reversed inequality known as the EMI (exponential metric increasing property). See  \cite[Sections 3 and 4.1.4]{cl} and the references therein. Therefore our EMD speaks of the non-negative nature of the curvature of bi-invariant metrics on Lie groups.
\end{rem}

In the case of a Riemannian metric, it is immediate using comparison triangles that if $\alpha,\beta,\gamma$ are the inner angles of a geodesic triangle in $K$, then
$$
\alpha+\beta+\gamma\ge \pi.
$$
Moreover, it is not hard to check that the EMD property implies that these Lie groups with bi-invariant metrics have non-negative curvature condition in the sense of Alexandrov (see \cite{bh}).

\smallskip

At the Lie algebra level, the EMD property indicates a contraction property for the local Lie group structure product:

\begin{coro}[The BCH contraction property]\label{bchcontra}
If $v,w\in\lie (K)$ are such that $e^ve^w\in V_R$ (for instance if $|v+w|<R$), then
$$
|BCH(v,w)|=|v\star w|=|\exp^{-1}(e^ve^w)|\le |v+w|.
$$
\end{coro}
\begin{proof}
The hypothesis tells us that $z=\exp^{-1}(e^ve^w)$ is in $B_R$ therefore by the minimality of segments and the EMD property $|z|=\dg(1,e^z)=\dg(e^v,e^{-w})\le |v+w|$.
\end{proof}

\paragraph{Unit spheres with faces}

We will now establish some facts for semi-norms that are not strictly convex. There is a nice characterization of short paths in linear spaces, generalized  below in Lemma \ref{minivect} for semi-norms (for lack of a better reference we include a proof).

\begin{rem}\label{normifun}
For $\varphi$ a unit norm functional in $E_{\mu}^*$, consider $\{v_i\}_{i=1\dots n}\in \lie(K)$ such that $\varphi(v_i)=\mu(v_i)$ for all $i$. Then 
$$
|\sum_i v_i|\le \sum_i |v_i|=\sum_i \varphi(v_i)=\varphi(\sum_i v_i)\le  |\sum_i v_i|,
$$
and in fact $\varphi$ norms any convex combination of the $v_i$ (the set of vectors normed by one fixed $\varphi$ is a convex cone at $v=0$). Reciprocally if the norm is additive on a set of vectors $x_i$, pick a unit norm $\varphi\in E^*_{\mu}$ such that $\varphi$ returns the norm of the sum $ \sum_ix_i$; it is easy to check that this $\varphi$ returns the norm of each $x_i$. By the last item of Proposition \ref{diss}, any commutator of these $x_i$ is in the kernel of $\varphi$.
\end{rem}

In the following lemma, rectifiable should be interpreted in the sense of inner metric spaces \cite{bbi}. In particular a rectifiable path taking values in a vector space is continuous  and Lebesgue differentiable, so it makes sense to compute the integral of the differential; since the path is also absolutely continuous (having finite length), the integral of the differential is the path itself.

\begin{lem}\label{minivect}
Let $E$ be a Finlser semi-normed space and let $\Gamma:[a,b]\to E$ be a rectifiable path. Let $\varphi\in E^*$, $\|\varphi\|=1$ such that $\varphi(\Gamma_b-\Gamma_a)=|\Gamma_b-\Gamma_a|$. Then the following are equivalent:
\begin{enumerate}
\item $\varphi(\dot{\Gamma}_t)=|\dot{\Gamma}_t|$ for almost all $t\in [a,b]$.
\item $\varphi(\Gamma_t-\Gamma_s)=|\Gamma_t-\Gamma_s|$ for almost all $a\le s\le t\le b$.
\item $\int_a^b |\dot{\Gamma}|dt = |\int_a^b \dot{\Gamma}dt|$.
\item The path $\Gamma$ is short joining its given endpoints.
\end{enumerate}
\end{lem}
\begin{proof}
If the first assertion holds, then 
$$|\Gamma_t-\Gamma_s|\ge \varphi(\Gamma_t-\Gamma_s)=\int_s^t\varphi(\dot{\Gamma})=\int_s^t|\dot{\Gamma}|=L_s^t(\Gamma)\ge |\Gamma_t-\Gamma_s|,
$$
thus the second assertion holds. Reciprocally, if the second assertion holds, dividing by $t-s$ and making $t\to s^+$ gives $\varphi(\dot{\Gamma}(s))=|\dot{\Gamma}(s)|$. Now assume $1.$, then 
$$
|\int_a^b \dot{\Gamma}|\le \int_a^b |\dot{\Gamma}|=\int_a^b \varphi(\dot{\Gamma})=\varphi (\int_a^b \dot{\Gamma} )\le |\int_a^b \dot{\Gamma}|,
$$
thus $3.$ holds. If $3.$ holds, $|\Gamma_b-\Gamma_a|=|\int_a^b \dot{\Gamma}|=\int_a^b |\dot{\Gamma}|=L_a^b(\Gamma)$ therefore $\Gamma$ is a short path joining its given endpoints and $4.$ holds. 

Now assume that $4.$ holds, if $\varphi(\dot{\Gamma})<|\dot{\Gamma}|$ in a set of positive measure $I\subset [a,b]$, then
$$
|\Gamma_b-\Gamma_a|=\varphi(\Gamma_b-\Gamma_a)=\int_a^b \varphi(\dot{\Gamma})<\int_a^b |\dot{\Gamma}|dt =L_a^b(\gamma)
$$
which contradicts the minimality of $\gamma$. Thus $1.$ holds.
\end{proof}

For strictly convex Finsler norms, the lemma is trivial since there are no faces thus the only short regular paths in $E$ are the straight segments. Now we extend the previous lemma to our context of Lie groups with $\Ad$-invariant metrics. Recall $V_R=\exp(B_R)$ for $R$ the injectivity radius of the exponential map.

\begin{teo}[Non strictly convex Finsler norms in $K$]\label{nonstrict} Let $z \in \lie(K)$ with $|z|\le R$, let $\gamma:[a,b]\to K$ be a piecewise $C^1$  path joining $1,e^z$ in $K$. The following are equivalent
\begin{enumerate}
\item $\gamma$ is a short path in $K$, i.e. $\LG(\gamma)=\dg(1,e^z)=z$.
\item $\gamma=e^{\Gamma}\subset V_R$ and for any norming functional $\varphi$ of $z$, $\varphi(\dot{\Gamma}_t)=|\gamma^{-1}_t\dot{\gamma}_t|$ for all $t$. 
\item $\gamma=e^{\Gamma}\subset V_R$, there exists a unit norm functional $\varphi$ such that $\varphi(\Gamma_t)=|\Gamma_t|$ and $\varphi(\gamma^{-1}_t\dot{\gamma}_t)=|\gamma^{-1}_t\dot{\gamma}_t|$ for all $t$. Thus $z,\gamma^{-1}\dot{\gamma}$ (normalized) sit inside a face of the unit sphere of $E_{\mu}$.
\end{enumerate}
\end{teo}
\begin{proof}
If $\gamma$ is short, since $\gamma(b)=e^z$, $\gamma$ does not leave $V_R$ therefore $\delta=e^{\Gamma}$ for a piecewise $C^1$ path joining $0,z$ and $\Gamma\subset B_R$, hence $|\Gamma_t|=\dg(1,e^{\Gamma_t})=\dg(1,\gamma_t)=\int_0^t |\gamma^{-1}\dot{\gamma}|$ for all $t$. Note that since $\Gamma$ starts at $0$, $\dot{\Gamma}_a=\dot{\gamma}_a$; we first claim that for any such short path, and any norming functional of $z$, we have $\varphi(\dot{\Gamma}_a)=|\dot{\Gamma}_a|$. To prove it, note that 
$$
\dg(e^{\gamma_t},e^z)\le \dg(1,e^{\Gamma_t})+ \dg(e^{\Gamma_t},e^z)=\ell_a^t(\gamma)+\ell_t^b(\gamma)=\ell_a^b(\gamma)=\dg(\gamma_a,\gamma_b)=|z|\le R
$$ 
thus there exists $y_t$ with $|y_t|\le R$ such that $e^{\Gamma_t}e^{y_t}=e^z$, and by the minimality theorem again $|y_t|=\dg(e^{\Gamma_t},e^z)$. Since $e^{y_t}=e^{-\Gamma_t}e^z$, $y_t$ is in fact piecewise $C^1$, $y_t=y_a+\dot{y_c}(t-a)=z+(t-a)\dot{y}_c$ for some $c\in (a,t)$ and $t$ near $a$. We now differentiate $e^{\Gamma_t}=e^ze^{-y_t}$ in $t=a$ to obtain
$$
\dot{\Gamma}_a=e^z\exp_{*-z}(-\dot{y}_a).
$$
On the other hand
$$
|y_t|\ge \varphi(y_t)=\varphi(z)+(t-a)\varphi(\dot{y}_c)=|z|+(t-a)\varphi(\dot{y}_c)=|\Gamma_t|+|y_t|+ (t-a)\varphi(\dot{y}_c),
$$
which implies that $0\ge |\frac{\Gamma_t}{t-a}|+ \varphi(\dot{y}_c)$. Letting $t\to a^+$ we obtain $0\ge |\dot{\Gamma}_a|+\varphi(\dot{y}_a)$. Now since $\varphi$ norms $z$, if $v=-z$ then $|-v|=|z|=\varphi(z)=\varphi(-v)$ thus by Gauss' Lemma
$$
|\dot{\Gamma}_a|\ge \varphi(\dot{\Gamma}_a)=-\varphi(e^z\exp_{*-z}\dot{y}_a)=-\varphi(e^{-v}\exp_{*v}\dot{y}_a)=-\varphi(\dot{y}_a)\ge |\dot{\Gamma_a}|,
$$
proving the claim $\varphi(\dot{\Gamma}_a)=|\dot{\Gamma}_a|$. Now for fixed $t$, let $\beta(s)=e^{s\Gamma_t}e^{sy_t}$, note that $\beta(0)=1$ and $\beta(1)=e^z$ and also
$$
|z|\le \LG(\beta)=|e^{-s\ad y_t}(\Gamma_t+y_t)|=|\Gamma_t+y_t|\le |\Gamma_t|+|y_t|=|z|
$$
therefore $\beta$ is also short, and applying the previous assertion to $\beta$ we obtain $\varphi(\Gamma_t+y_t)=\varphi(\dot{\beta}_0)=|\dot{\beta}_0|=|\Gamma_t+y_t|$. Thus $|z|=|\Gamma_t+y_t|=\varphi(\Gamma_t+y_t)=\varphi(\Gamma_t)+\varphi(y_t)\le |\Gamma_t|+|y_t|=|z|$, implying that $\varphi(\Gamma_t)=|\Gamma_t|$. This in turn implies, if $s>t\ge a$,  that
$$
\varphi(\Gamma_s-\Gamma_t)=|\Gamma_s|-|\Gamma_t|=\dg(1,\gamma_s)-\dg(1,\gamma_t)=\dg(\gamma_t,\gamma_s)=\int_t^s |\gamma^{-1}\dot{\gamma}|.
$$
If $t$ is a smooth point of $\gamma$, then dividing by $s-t>0$ and letting $s\to t^+$ gives $\varphi(\dot{\Gamma_t})=|\gamma^{-1}_t\dot{\gamma}_t|$: this proves that if $\gamma$ is short, the second condition of the theorem holds. Now assume the second assertion holds, then
$$
|z|=\varphi(z)=\int_a^b \varphi(\dot{\Gamma})=\int_a^b |\gamma^{-1}\dot{\gamma}|=\LG(\gamma)
$$
therefore $\gamma$ is short. This proves that the first and second assertions are equivalent. Now assuming the second assertion, the path $\gamma$ is short and examining the proof above (of $1\Rightarrow 2$) we conclude that $\varphi(\Gamma_t)=|\Gamma_t|$ for all $t$. Thus by Gauss' Lemma
$$
\varphi(\gamma_t^{-1}\dot{\gamma}_t)=\varphi(e^{-\Gamma_t}\exp_{*\Gamma_t}\dot{\Gamma}_t)=\varphi(\dot{\Gamma}_t)=|\gamma_t^{-1}\dot{\gamma}_t|,
$$
and the third assertion holds. Finally, assuming the third assertion it is immediate again by Gauss Lemma that $\varphi(\dot{\Gamma})=|\gamma^{-1}\dot{\gamma}|$, and again 
$$
|z|= \varphi(z)=\int_a^b \varphi(\dot{\Gamma})=\int_a^b |\gamma^{-1}\dot{\gamma}|=\LG(\gamma)\ge |z|
$$
shows that $\gamma$ is short.
\end{proof}

\begin{coro} Let $\Gamma:[a,b]\to E_{\mu}$ be a piecewise $C^1$ short path joining $0,z$, let $\gamma=e^{\Gamma}$. Then  $|\gamma^{-1}_t\dot{\gamma}_t|=|\dot{\Gamma}_t|$ for (almost) all $t\in [a,b]$. If $|z|\le R$ then $\gamma$ is short in $K$ with the same length than $\Gamma$. Moreover if $\varphi(z)=|z|$ is a unit norm functional, then
$$
\varphi(\gamma^{-1}_t\dot{\gamma}_t)=|\gamma^{-1}_t\dot{\gamma}_t|=|\dot{\Gamma}_t|=\varphi(\dot{\Gamma}_t)\quad \forall t\in [a,b],
$$
thus $z$, $\gamma^{-1}\dot{\gamma}$ and $\dot{\Gamma}$ (normalized) sit inside the same face of the unit sphere of $E_{\mu}$.
\end{coro}
\begin{proof}
By Remark \ref{exptaucont} and Lemma \ref{gauss}, if $a\le s\le t\le b$ then 
$$
|\Gamma_t-\Gamma_s|=\Le(\Gamma)_s^t\ge \int_s^t|\gamma^{-1}\dot{\gamma}|\ge |\Gamma_t|-|\Gamma_s|=|\Gamma_t-\Gamma_s|
$$
where the last equality is because $\Gamma$ is short therefore its length is additive (or use the previous lemma and the fact that here $\Gamma_a=0$). Dividing by $t-s$ and taking limits gives the equality of speeds $|\dot{\Gamma}|=|\gamma^{-1}\dot{\gamma}|$. Then $\LG(\gamma)=\Le(\Gamma)=|z|=\dg(1,e^z)$ if $|z|\le R$ thus $\gamma$ is also short. By the previous theorem and lemma, if $\varphi$ is a unit norm functional with $\varphi(z)=|z|$, then  
$$
|\dot{\Gamma}|=\varphi(\dot{\Gamma})=|\gamma^{-1}\dot{\gamma}|=|\dot{\Gamma}_t|  \quad \textit{ for all }t\in [a,b].\qedhere
$$
\end{proof}

\begin{rem}Let us call a piecewise $C^1$ path $\gamma\subset K$ \textit{quasi-autonomous} if there exist a unit norm functional $\varphi$ such that $\varphi(\gamma^{-1}\dot{\gamma})=|\gamma^{-1}\dot{\gamma}|$ for all $t$. This terminology is borrowed from the study of the Hofer metric in the groups of   symplectomorphisms of a compact symplectic manifold $(M,\omega)$, see \cite{hofer} and the references therein. By Theorem \ref{nonstrict}, any short path in $K$ is quasi-autonomous. For finite dimensional Lie groups, it can be proved that if $\gamma$ is quasi-autonomous, then for a lift $e^{\Gamma}=\gamma$ of $\gamma$, it holds $\varphi(\dot{\Gamma})=\varphi(\gamma^{-1}\dot{\gamma})=|\gamma^{-1}\dot{\gamma}|$, this will be developed elsewhere; thus in that case and again by Theorem \ref{nonstrict}, any quasi-autonomous path in a finite dimensional group $K$ is short for the bi-invariant metric.
\end{rem}

With these tools we take a closer look at the EMD inequality, with interest in when exactly it turns into an equality.

\begin{coro}
If $v,w\in\lie(K)$,  are such that $|w-v|\le R$, then $\dg(e^v,e^w)=|w-v|$ if and only if there exists a unit norm functional $\varphi$ such that $\varphi(w-v)=|w-v|$ and $\varphi$ vanishes on each term of the $BCH$ expansion of $\exp^{-1}(-w,v)$. In that case $\dg(e^{sv},e^{sw})=s|w-v|$ for all $s\in [0,1]$.
\end{coro}
\begin{proof}
We note that  $\dg(e^v,e^w)=\dg(1,e^{-v}e^w)\le |w-v|=\LG(\beta)$ with $\beta(t)=e^{-tv}e^{tw}=e^{B_t}$ (note that $B_t=BCH(-tv,tw)$ is the Baker-Campbell-Hausdorff expansion of $-tv,tw$). Therefore equality holds if and only if $\beta$ is short, and this by Theorem \ref{nonstrict} occurs if and only if $\beta=e^{\Gamma}$, and there exists $\varphi$ such that
$$
\varphi(e^{-t\ad w}(w-v))=\varphi(\beta^{-1}_t\dot{\beta}_t)=|\beta^{-1}_t\dot{\beta}_t|=|w-v|=\varphi(\dot{B}_t)=\varphi(\frac{d}{dt}BCH(-tv,tw)).
$$
Evaluating the derivatives of the above equality at $t=0$ gives that $\varphi(w-v)=|w-v|$ and that $\varphi$ is zero in every term of the BCH expansion. 
Assume now that $\varphi$ is a unit norm functional such that $\varphi(w-v)=|w-v|$ and $\varphi$ vanishes in each $B(w,-v)-w-v$. Then 
$$
0\le |w-v|-\dg(e^v,e^w)=|w-v|-|BCH(w,-v)|\le \varphi(w-v)-\varphi(BCH(-v,w))=0,
$$
thus $\dg(e^v,e^w)=|w-v|$.
\end{proof}

\noindent $\S$ Is the above condition equivalent to the assertion: $\varphi$ gives the norm of $w-v$, and $\varphi$ vanishes on the \textit{Lie ideal generated} by $v,w$? That is, $\varphi$ vanishes on each $z=[v,x]+[w,y]$ for $x,y$ in the Lie algebra generated by $v,w$?

\begin{coro}\label{uniextremal} If $|z|\le R$ and $z/|z|$ is an extremal point of the unit sphere of $E_{\mu}$, then the only short piecewise $C^1$ path joining $1,u=e^z$ in $K$ is (a reparametrization of) the segment $t	\mapsto e^{tz}$.
\end{coro}
\begin{proof}
Extremal points of the unit sphere are characterized as the intersection of all the maximal faces, which in turn are described by the norming functionals of $z$. If $\gamma:[a,b]\to K$ is a short path joining $1,e^z$ then by (the second part of) Theorem \ref{nonstrict}, the normalized speed $\gamma^{-1}\dot{\gamma}$ sits inside the intersection of all these faces, which in this case is the point $z/|z|$. Thus $\gamma^{-1}_t\dot{\gamma}_t=|\gamma^{-1}_t\dot{\gamma}_t|z/|z|=g(t)z$ where $g(t)=|\gamma^{-1}_t\dot{\gamma}_t|/|z|$ is continuous and positive. Let $f(t)=\int_a^t g(s)ds$, which is increasing with $f(a)=0$, $f(b)=1$, and note that $\delta(t)=e^{f(t)z}$ is the only solution of the differential equation $\dot{\gamma}_t=\gamma_t g_t z$ with $\delta(a)=1$ (in the case of locally convex groups, this can be restated as the uniqueness of paths with a given logarithmic derivative, see \cite[Lemma II.3.5]{neeblie}). Then $\gamma=\delta$ thus $\gamma$ is a reparametrization of the segment.
\end{proof}

\medskip

Since Theorem \ref{nonstrict} shows that the geodesic of $K$ are exactly the exponentials of the geodesics of its Lie algebra, when $K=U_n(\mathbb C)$ is the group of unitary matrices, this can be used to give a different description of the non-commutative Horn inequalities (also known as the quantum Horn inequalities \cite{belkale}), which describe the relation between the spectrum of $x,y,z$ skew-Hermitian matrices when $e^xe^y=e^z$. This will be discussed in another paper.

\medskip

This is related to the results obtained by Bialy and Polterovich et al. \cite{bialyp,hofer,lm1} for the Hofer metric in the group of symplectomorphisms of a manifold $M$ (Remark \ref{symplec} below). Since the exponential map is not a diffeomorphism in any neighbourhood of the identity of the group $\textrm{Diff}(M)$ of diffeomorphisms of $M$, direct applications to the study of the Hofer metric should be developed using ad-hoc density techniques; this will appear in detail elsewhere; see also Section \ref{weak} below.

\medskip

\paragraph{The nil-elements for the distance} As a small ending note of this section, we now expand the results on the closed normal subgroup $K_{\mu}=\{u\in K:\dg(1,u)=0\}$ considered in Remark \ref{esgru}, applying the minimality results of this section. Let $q:K\to K/K_{\mu}$ denote the quotient map.

\begin{lem}\label{nilel}
Let $N=\{z\in \lie(K):|z|=0\}$, $K_{\mu}$ as before. Then
\begin{enumerate}
\item $K_{\mu}=\exp(N)=\{e^z: |z|=0\}$ thus $K_{\mu}=\{1\}$ iff $\mu$ is non-degenerate.
\item $K_{\mu}$ is an embedded locally exponential Lie subgroup in $K$.
\item The quotient tangent Finsler norm $\tilde{\mu}$ in $K/K_{\mu}$ is non-degenerate and bi-invariant, and it induces the quotient non-degenerate distance in the quotient group $K/K_{\mu}$.
\item If $K$ is BCH then $K/K_{\mu}$ is BCH and $\tilde{\exp}\circ q_{*1}=q\circ\exp$ where $\tilde{\exp}$ is the exponential map of the quotient Lie group.
\item If $K$ is BCH in the $\mu$-ball $B_R$, and $\tilde{B}_R=\{\tilde{v}\in \lie(K)/N: \tilde{\mu}(\tilde{v})<r\}$, then $\tilde{\exp}|_{\tilde{B}_R}$ is a diffeomorphism onto its open image.
\end{enumerate}
\end{lem}
\begin{proof}
 Clearly $\dg(1,e^z)\le |z|=0$ if $z\in N$, but also note that if $u\in K_{\mu}$ then $u\in V_R$ otherwise all paths joining $1,u$ would have length bigger than $R$, a contradiction. Hence $0=\dg(1,e^z)=|z|$ by the minimality theorem and $z\in N$. That is $K_{\mu}=\exp(N)$. Since $\exp^{-1}|_{V_R}$ is smooth, it can be used to define a smooth (global) chart for $K_{\mu}=\exp(N)=\exp(B_R\cap N)=V_R\cap K_{\mu}$. This gives $K_{\mu}$ the structure of a locally exponential Lie subgroup of $K$ \cite[Theorem IV.3.3]{neeblie}. In that case we have just shown that the Lie algebra $N$ (a closed ideal of $\lie(K)$) is the Lie algebra of the normal subgroup $K_{\mu}$, and therefore the Lie group $K/K_{\mu}$ inherits the tangent quotient norm which is now non-degenerate (and bi-invariant), and by Theorem \ref{elcorodelq} it induces the quotient distance there. 

The fourth assertion is due to the quotient theorem for BCH groups (see \cite[Theorem 2.20]{glock}, the identity $\tilde{\exp}\circ q_{*}=q\circ \exp$ is apparent and in fact it is used to define the quotient exponential map for small $v\in\lie(K)$. Note that since $q$ is open then $q\circ\exp|_{B_R}$ is open therefore $\tilde{\exp}|_{\tilde{B}_R}$ is open and smooth. Assume $\tilde{v}=q_{*}v,\tilde{w}=q_{*}w\in \tilde{B}_r$, then $v,w\in B_r$ and if $\tilde{\exp}(\tilde{v})=\tilde{\exp}(\tilde{w})$ then $q(e^v)=q(e^w)$ or equivalently there exists $z\in N$ such that $e^v=e^we^z$ by the first item of this lemma. Then by the BCH contraction property \ref{bchcontra} and the fact that $z\in N$, $|v|\le |w+z|=|w|$ and reversing $|v|=|w|=|w\star z|<R$ thus $v=BCH(w,z)=w+z'$ for some $z'\in N$. This proves that $\tilde{v}=\tilde{w}$ hence $\tilde{\exp}$ is injective in $\tilde{B}_R$.
\end{proof}

The lemma above tells us that if we start with a semi-norm in $\lie(K)$, we can divide by its kernel $N$, and the quotient norm induces a reasonable and non-degenerate metric in $K/\exp(N)$, which is locally exponential and still complies with the requirements of Definition \ref{localymuexp}, thus $K/\exp(N)$ keeps the nice local minimality properties of segments stated above in this section. Can the additional assumption that $\exp$ is BCH in the ball $B_R$ be dropped?

\subsection{Relative (or local) minimality}\label{weak}

In this section we discuss the applications of Lemma \ref{gauss} for a wider set of Lie groups, that have a smooth exponential, but it is not injective restricted to any open ball of the tangent norm $\mu$. The discussion at the beginning of Section \ref{biinvme} applies here (from Definition \ref{adop} to Lemma \ref{gauss}). To simplify the exposition, we will assume that $\mu$ is homogeneous and non-degenerate (i.e. $\mu$ is a norm).

\medskip

By Lemma \ref{gauss} and Remark \ref{exptaucont}, if $\Gamma:[a,b]\to \lie(K)$ is a short rectifiable path with $\Gamma(a)=0$, and $\gamma=e^{\Gamma}$, then
\begin{equation}\label{seigual}
|\Gamma(b)|\le \LG(\gamma)\le \int |\dot{\Gamma}|=|\Gamma(b)|
\end{equation}
where the last equality is by the assumption that $\Gamma$ is short in the linear space $\lie(K)$. Therefore $\gamma=e^{\Gamma}$ has the same length than $\Gamma$ and the following is immediate (here we do not ask for $\exp$ to be a local diffeomorphism):

\begin{teo}[Exponential variations]
Let $z\in\lie(K)$, $\mu$ a bi-invariant norm in $\lie(K)$. Let $X:(-\epsilon,\epsilon)\times [0,1]\to \lie(K)$ be a fixed-endpoints variation of the segment $X_0=tz$, let $\gamma_s=\exp(X_s)$. Assume that $X_s$ is piecewise $C^1$ for each $s$. Then 
$$
\LG(\gamma_s)\ge \LG(\gamma_0)=|z|\quad \forall s\in (-\epsilon,\epsilon).
$$ 
In fact, if $\Gamma:[a,b]\to \lie(K)$ is any rectifiable short path starting at  $0$, $X_s$ is a fixed-endpoints variation of $\Gamma$, and $\gamma_s=e^{X_s}$, then
$$
\LG(\gamma_s)\ge \LG(\gamma_0)=L(\Gamma)=|\Gamma(b)|\quad \forall s\in (-\epsilon,\epsilon).
$$ 
\end{teo}

\medskip

The notion of \textit{local minimality} was introduced as \textit{minimality among non-wandering curves} by Recht et al. in \cite{dmlr0}. The idea is that a path is short with respect to those which are close in some topology (but not necessarily with respect to \textit{all} paths joining the given endpoints). 

\medskip

Instead of neighborhoods in the path space, another possible setting is that of a locally exponential Lie group, where the neighborhood is given by the manifold topology:

\begin{defi}[Locally exponential groups]\label{lexp}
Let $0\in U\subset \lie(K)$ be a balanced open neighborhood, $V=\exp(U)$ and assume $\exp|_U:U\to V$ is a $C^1$ diffeomorphism.
\end{defi}

What we are loosing here is that $\exp$ is a local diffeomorphism restricted to some open \textit{ball} $B_R$ of the tangent norm. 

\begin{teo}[Locally minimizing paths]\label{miniexp}
Let $\mu$ a bi-invariant norm in $\lie(K)$. Assume $K$ is locally exponential and $z\in U$. If $\beta$ is piecewise $C^1$ and joins $1,e^z$ in $K$, then $\LG(\beta)\ge |z|$, provided $\gamma$ does not leave $V$. Thus if $\Gamma\subset \lie(K)$ is a short rectifiable path joining $0,z$ with $z\in U$, then $\gamma=e^{\Gamma}$ is locally minimizing, i.e. is short with respect to paths that do not leave $V=\exp(U)$. 
\end{teo}
\begin{proof}
If $\beta\subset V$ is a path with the given endpoints, then $\beta=e^{\Lambda}$ for a piecewise $C^1$ path $\Lambda\subset U$ joining $0,z$. by Lemma \ref{gauss},
$$
|z|=|\Lambda(1)|\le \LG(\beta),
$$
proving the first assertion. The second one is immediate from (\ref{seigual}).
\end{proof}

In particular, note that segments $\delta$ in $K$ are short \textit{among paths that are close to $\delta$}, with the possibility of shorter paths leaving the neighborhood $V$.

\bigskip

Assuming strict convexity, the proof of Theorem \ref{uniestricta} can be adapted in this context to show that if there exists a short path in $K$ (minimal among \textit{all} paths in $K$) then there must exist a segment joining the same endpoints:

\begin{teo}Assume that $K$ is locally exponential and $\mu$ is strictly convex. If $\gamma:[a,b]\to K$ is short among all paths joining its given endpoints, then there exists $z\in \lie(K)$ such that $\gamma(1)=\gamma(0)e^z$, and $|z|=\LG(\gamma)$.
\end{teo}
\begin{proof}
We can safely assume that $\gamma(a)=1$, fix $t\in (a,b)$. Let $U,V=\exp(U)$ be as in  Definition \ref{lexp}. Pick a partition $\{a=t_0<t_1<\cdots <t_n=b\}$ of the interval $[a,b]$ such that $t=t_{i_0}$ for some $i_0$. Denote $\Delta_i=[t_i,t_{i+1}]$ and make sure that $\gamma|_{\Delta_i}\subset \gamma(t_i)V$ for all $i=0\dots {n-1}$. Then there exists $x_i\in U$ such that $\gamma(t_{i+1})=\gamma(t_i)e^{x_i}$ and by the local minimality of segments, $|x_i|\le  \LG(\gamma|_{\Delta_i})$. Therefore
$$
\sum_i |x_i|\le \sum_i \LG(\Gamma|_{\Delta_i})=\LG(\gamma)=\dg(1,u)\le \sum_i |x_i|.
$$
Considering $\beta(t)=e^{tx_1}\dots e^{tx_n}$ and arguing as in the proof of Theorem \ref{uniestricta}, we conclude that the $x_i$ are aligned and thus $\gamma(b)=e^z$ for some $z=\lambda x_1\in \lie(K)$. Since $|z|=\sum_i |x_i|$, the segment is also short.
\end{proof}

\begin{rem}[The no-small subgroups property]
It is well-known that $\mu$ if is a continuous norm or a Finsler norm (i.e. non-degenerate) on $\lie(G)$, and $G$ is locally exponential, then there exists a neighborhood $W$ of $0$ in $\lie(G)$ such that $\exp(W)$ contains no nontrivial subgroup. For lack of a reference, we include a proof: shrinking, we can assume that there exists $0\in U\subset \lie(G)$ such that $U$ is open and convex, and $\exp|_U:U\to \exp(U)$ is a diffeomorphism. Fix any $R>0$ and consider the open set $W=\frac{1}{2}U\cap B_R\subset U$ (since $U$ is convex). Assume $K\subset \exp(W)$ is a subgroup, let $k\in K$ then there exists $w\in W$ such that $k=e^w$. Since $K$ is a group, $e^{2w}=k^2\in H\subset \exp(W)$ therefore there exists $w'\in W$ such that $e^{2w}=e^{w'}$. Note that $2w\in 2W\subset U$ and the same holds for $w'$, therefore $2w=w'$ and this proves that $2w\in W\subset B_R$. Iterating, we obtain $2^n w\in B_R$ for all $n\ge 1$ or equivalently $2^n|w|<R$. This is only possible if $w=0$, therefore $k=1$ and $K=\{1\}$.

\end{rem}

$\S$ Assuming the existence of $z\in \lie(G)$ with $|z|<R$ and $e^z=1$, the loop $\delta(t)=e^{tz}$ cannot be minimizing thus a major obstruction for the local minimality theorem (even for locally exponential groups) is the existence of arbitrary (norm) small $z$ outside the ``flat'' neighborhood $U\subset\lie(K)$.

\subsection{Normal subgroups}

In this section we examine the geodesics of a group obtained by a quotient of a Lie group $K$ with a bi-invariant metric and a locally exponential Lie subgroup.  A closed subgroup $H\subset K$ is a Lie subgroup if  there exists a $0$-neighborhood $U\subset \lie(K)$ such that 
\begin{equation}\label{subgr}
\exp(U\cap \lie(H))=\exp(U)\cap H
\end{equation}
(hence using the exponential chart, the subgroup is an embedded submanifold). The following equivalence will be useful (see \cite[Section IV]{neeblie}), recall $\exp_{*v}(w)=e^v\kappa_v(w)$ (Remark \ref{difexp}):

\begin{lem} Let $H\subset K$ be a closed normal Lie subgroup of a locally exponential Lie group. Then, shrinking $U$ if necessary, the following are equivalent:
\begin{enumerate}
\item $K/H$ is a locally exponential Lie group.
\item $H$ is a locally exponential Lie subgroup and $\kappa_v(\lie(H))=\lie(H)$ for $v\in U$.
\end{enumerate}
\end{lem}

Then in the situation of the previous lemma, if $\pi:K\to K_0=K/H$ is the quotient map and $K$ has a bi-invariant metric, we obtain the continuous quotient norm (Definition \ref{qfm}). 
Let $o=\pi(1)$ be the identity of $K_0$., then for $\pi_{*k}(v)\in T_{kH}K_0$ with $v_0=\pi_{*k}(v)$ we have
$$
|\pi_{*k}(v)|_{kH}=\inf\{|k^{-1}v-z|: z\in \lie(H)\}=|\pi_{*1}(k^{-1}v)|_o.
$$
By the naturality of the exponential maps, we have $\pi\circ \exp_K= \exp_{K_0}\circ\pi_{*1}$, and if $\exp_K|_{B_R}$ is a diffeomorphism onto its image $V_R\subset K$ for some $R>0$, then $\exp_{K_0}|_{\widetilde{B_r}}$ is a diffeomorphism onto its image $\widetilde{V_r}=\pi(V_r)\subset K_0$, where $\widetilde{B_r}=\pi_{*1}(B_r)$ is the open quotient ball, for some smaller $r\le R$ Therefore the minimality results of Section \ref{loca}  apply. In what follows we estimate this radius with a lower bound $r=R/2$.

\smallskip

If the groups are $BCH$, and $w\in\lie(H)$, then locally 
$$
e^z=e^ve^w=v+w+\nicefrac{1}{2}[v,w]+\dots=v+h
$$
with $h\in\lie(H)$, since the Lie algebra of a normal subgroup is a closed ideal. This is the content of the following lemma.
\begin{lem}
Let $U\subset \lie(K)$ be as in (\ref{subgr}), with $H\subset K$ a normal subgroup and $K/H$ locally exponential. If $e^z=e^ve^w$ with $v,w,z\in U$ and $w\in \lie(H)$, then $z-v\in \lie(H)$.
\end{lem}
\begin{proof}
Let $z_t=\exp^{-1}(e^ve^{tw})\subset \lie(K)$ which is a smooth path with $z_0=v$, $z_1=z$. Then differentiating $e^{z_t}=e^{v}e^{tw}$ we obtain $e^{z_t}\kappa_{z_t}(\dot{z_t})=e^{z_t}w$ (Remark \ref{difexp}). Hence $\kappa_{z_t}(\dot{z_t})=w$ for all $t\in [0,1]$. Since we are assuming that the exponential is invertible for small $z$, and by the previous lemma, we have $\dot{z_t}=\kappa_{z_t}^{-1}(w)\in \lie(H)$ for all $t\in [0,1]$. Hence
$$
z=z_1=z_0+\int_0^1 \dot{z_t}\, dt=v+\int_0^1 \kappa_{z_t}^{-1}(w)dt=v+h\in v+\lie(H).\qedhere
$$
\end{proof}

\begin{teo}
Let $K$ be a Lie group with a bi-invariant metric, $R>0$ as in Definition \ref{localymuexp}. Let $H\subset K$ be a normal Lie subgroup such that $K_0=K/H$ is locally exponential, assume $B_R\subset U$ where $U\subset\lie(K)$ is as in (\ref{subgr}). Let $v\in \lie(K)$ with $|v|<R/2$. Then
$$
d_{K_0}(o,e^v\cdot o)=d_K(H,e^v)=\inf\{|v-z|: z\in \lie(H)\}=|\pi_{*1}(v)|_0,
$$
and the segment 
$$
t\mapsto \gamma(t)=\exp_{K_0}(t\pi_{*1}v)=e^{tv}\cdot o=\pi(e^{tv})
$$
is a short path in $K_0$.
\end{teo}
\begin{proof}
Note that for the quotient metric, $\gamma$ has constant speed
$$
\mu(\dot{\gamma}_t)_{\gamma_t}=|\pi_{*e^{tv}}v|_{e^{tv}H}=|\pi_{*1}(v)|_o=\inf\{|v-z|: z\in \lie(H)\}\le |v|,
$$
with equality if and only if $v$ is a minimal vector in $\lie(K)$. Let $h_n\in H$ be a minimizing sequence, $\dg(h_n,e^v)\to \dg(H,e^v)\le \dg(1,e^v)=|v|<R/2$. Let $n_0$ be such that $\dg(1,e^vh_n^{-1})=\dg(h_n,e^v)<R/2$ for all $n\ge n_0$. Then  $e^v h_n^{-1}=e^{z_n}$ for some $z_n\in \lie(K)$ with $|z_n|<R/2$. Since 
$$
\dg(1,h_n^{-1})\le \dg(1,e^v)+\dg(e^v,h_n^{-1})=|v|+\dg(e^v,e^ve^{z_n})=|v|+|z_n|<R/2+R/2=R,
$$
it must be $h_n^{-1}=e^{w_n}$ for some $w_n\in \lie(K)$ with $|w_n|<R$, with $w_n\in \lie(H)$ by our assumptions. Thus $e^{z_n}=e^ve^{w_n}$ for each $n\ge n_0$ and by the previous lemma, $z_n=v+x_n$ for some $x_n\in \lie(H)$, and
$$
|\pi_{*1}v|_o\le |v+x_n|=|z_n|=\dg(1,e^{z_n})=\dg(h_n,e^v)\to \dg(H,e^v).
$$
Hence
$$
|\pi_{*1}v|_o\le \dg(H,e^v)=d_{K_0}(o,e^v\cdot o)\le L_{K_0}(\gamma)=|\pi_{*1}(v)|_0.\qedhere
$$ 
\end{proof}

\section{Applications and examples}\label{examples}

In this section we present examples of the theory developed in Section \ref{hs} and Section \ref{biinvme}. Some of these examples are new and others are well-known; we have included them hoping that this new viewpoint will be useful to obtain a better comprehension. To help a bit the exposition, we have chosen to use $\di(\cdot,\cdot)$ for distances in the group $G$ and $d(\cdot,\cdot)$ for distances in the quotient space $M=G/K$, with suffixes indicating the Finsler norm used when necessary.

\subsection{Isometries of a Banach space}\label{isos}

Let $X$ be a Banach space and consider the uniform norm in $\mathcal L(X)$ denoted by $\|\cdot\|$, and denote $\di_{\infty}$ the induced left-invariant metric in $GL(X)$, the group of invertible linear operators on $X$. Let
$$
K=\mathcal U(X)=\{U\in GL(X): \|Ux\|=\|x\| \quad \forall \,x\in X\},
$$
the group of isometries of $X$. It is a Banach-Lie group with Banach-Lie algebra
$$
\lie(K)=\{V\in \mathcal L(X): \exp(tV)\subset K\quad \forall\, t\in\mathbb R\},
$$
the skew-Hermitian operators of $\mathcal L(X)$. Clearly the uniform norm is bi-invariant for $K$. It is well-known that, when restricted to the open ball $B_{\pi}=\{z:\|z\|<\pi\}$, the exponential map $\exp:\lie(K)\to K$ is a diffeomorphism onto its image $V_{\pi}\subset K$. 

\smallskip

Therefore for every $u\in V_{\pi}$ there exists $z\in \lie(K)$ with $\|z\|<\pi$ such that $u=e^z$, and by Theorem \ref{mini} the path $\delta(t)=e^{tz}$ is a short path for the bi-invariant metric $d_{\infty}$ induced by the the uniform norm, and $\di_{\infty}(1,u)=\|\log(u)\|=\|z\|$.  % neeb CH Lema 5.3

\begin{rem}
Since $K$ is not necessarily a Lie subgroup of $GL(X)$, the set $V_{\pi}$ is not necessarily open for the \textit{norm} topology. For \textit{algebraic} Banach-Lie subgroups $K$ of $GL(X)\cap \mathcal U(X)$ however, one can estimate better the minimality radius in terms of the uniform norm: if $K$ is algebraic of degree $n$ then $K$ is a submanifold of $GL(X)$ and $\exp$ is a diffeomorphism onto $\{g\in K: \|g-1\|<\sin(\pi/n)$ (see Harris and Kaup \cite{harris}).
\end{rem}

\subsubsection{Unitary group of a Hilbert space}

For $X=H$ a Hilbert space, the situation is much nicer though; recall that  
$$
\lie(\mathcal U(H))=\mathcal L(H)_{ah}=\{x\in \mathcal L(H): x^*=-x\},
$$
the anti-Hermitian operators of $\mathcal L(H)$. It is well-known that $\exp$ is a diffeomorphism between the sets $\{ z\in \lie(\mathcal U(H)): \|z\|<\pi\}$ and $\{u\in \mathcal U(H): \|u-1\|<2\}$, which happens to be norm dense un $\mathcal U(H)$. Besides, each $u\in \mathcal U$ admits a Borelian logarithm $z\in \lie(\mathcal U(H))$ such that $\|z\|\le \pi$ (unique and with norm $\|z\|<\pi$ if $\|u-1\|<2$). Therefore Theorem \ref{mini} implies that $\di_{\infty}(1,e^z)=\|z\|$ and the geodesical radius of $\mathcal U(H)$ is $\pi$. 

\medskip

This a known result, obtained independently by Atkin \cite{atkin1,atkin2}, later by Porta and Recht \cite{portarecht}, both using a different technique than ours, involving representations of the one-parameter groups (the short paths) as geodesics in the sphere of the Hilbert space $H$.

\subsubsection{Matrix groups} Let $K=U_n$ stand for the orthogonal or the unitary group of $n\times n$ matrices. Let $|\cdot|$ stand for any weakly unitarily invariant norm (that is $|uzu^*|=|z|$ for any matrix $z$ and any $u\in K$, see Bhatia's book \cite[Chapter IV]{bhatia}). Then the discussion of the previous paragraph on the unitary group of a Hilbert space applies, with more freedom to choose the norm -for example, the $p$-norms given by the trace, 
$$
|z|_p=\tr((x^*x)^{\nicefrac{p}{2}}))^{\nicefrac{1}{p}}
$$
for $1\le p<\infty$, or the Ky-Fan norms, etc. For all these norms Theorem \ref{mini} implies that one-parameter groups are short (minimizing) paths for the Finsler metric. Moreover, when the norm is strictly convex -for instance, the $p$-norms for $p\ne 1$- those are the only minimizing paths (Theorem \ref{uniestricta}), and equality in the EMD property $\di(e^v,e^w)\le |v-w|$ only occurs if $v,w$ commute (Theorem \ref{emd}).

\medskip

For unitarily invariant norms, these results were obtained by Antezana et. al in \cite{alv} using a different technique, involving a fundamental result of Thompson regarding the eigenvalues of a sum of Hermitian matrices. However, the minimality results seem to be new for weakly unitarily invariant norms, such as the $c$-numerical radius
\begin{equation}\label{cnum}
\omega(v)_c=\max\{|\tr(cuvu^*)|:u\in U_n\}
\end{equation}
for a fixed matrix $c$ such that $\tr(c)\ne 0$ and $c$ is not a scalar multiple of the identity (see \cite[Section IV.4]{bhatia}). The usual numerical radius of $v$ can be recovered by picking the one-dimensional projection $c=e_1\otimes e_1$.

\subsection{Sphere of a Banach space}

Let $X$ be a smooth \textit{transitive} Banach space (the group of isometries $\mathcal U(X)$ acts transitively on the unit sphere $S_X=\{x\in X:\|x\|=1\}$, see \cite{guerre}). Pick any $p\in S_X$ and consider the natural transitive, smooth action ${\mathcal U}(X)\times S_X\to S_X$ given by $\pi(u,x)=u(x)$. The isotropy group is then $K=\{u\in \mathcal U(X): u(p)=p\}$ which, being algebraic, is a Banach-Lie subgroup of $\mathcal U(X)$. Moreover, $u\mapsto u(p)$ is a smooth submersion, therefore the requirements on $M=S_X$ in Definition \ref{smoothm} are fulfilled; since $\di_{\infty}$ is a complete metric in $\mathcal U(X)$ then the distance
$$
d_{S_X}(p,u(p))=\di_{\infty}(K_p,u)
$$
is a complete metric in $S_X$. The discussion in Section \ref{isos} shows that $\Gamma(t)=e^{tz}$ is a short geodesic of $\mathcal U(X)$ for sufficiently small $z\in \lie(\mathcal U(H))=i\textrm{Herm}(X)$. Therefore by Theorem \ref{arribajo2}, if one can find $z$ such that $\Gamma$ gives the distance from $1$ to $uK_p$, then $\gamma(t)=e^{tz}(p)$ would be a short geodesic in $S_X$ for the quotient distance $d_{S_X}$ (Remark \ref{geococ}). Note that by Remark \ref{necemin}, $z$ must be a minimal vector, that is
$$
\|z\|=\inf \{\|z-v\| : p\in \ker(v)\}.
$$

\smallskip
 
If $X=H$ is a Hilbert space, these metrics have been studied for the action of the $p$-Schatten groups (see \cite{aa} and the next section). In particular, it was shown in \cite[Remark 4.8]{aa} that when the Hilbert space is real, the action of the Hilbert-Schmidt unitary group induces in the unit sphere $S_H$ the usual Riemannian metric (the ambient metric of the inclusion $S_H\subset H$).

\paragraph{Sphere of a $C^*$-Hilbert module} Quotient metrics for the sphere $S_{\mathcal X}$ of Hilbert $C^*$-module over a $C^*$-algebra  $\mathcal A$ were considered by Andruchow and Varela in \cite{avhm}, in connection with the ideas of the quotient metric for the quotients of unitary groups of $C^*$ algebras $U_{\mathcal A}/U_{\mathcal B}$ developed by Recht et al. in \cite{dmlr,dmlr2}.

\subsection{Classical (restricted) groups of Hilbert space operators}\label{dlh}

Let $H$ be a complex Hilbert space, and let $\bh$ stand for the bounded linear operators acting on $\mathcal H$. We will denote with $\kh\subset \bh$ the ideal of compact operators. For an operator $x\in \bh$, we denote with $\sigma(x)\subset\mathbb C$ the spectra of $x$, with $\|x\|\in\mathbb R_+$ the usual operator norm and with $x^*\in\bh$ the Hilbert space adjoint of $x$. For $x\in \bh$ we will denote $|x|=\sqrt{x^*x}\in \bh^+$ the modulus of operators, and $\re(x)=\nicefrac{1}{2}(x+x^*)$, $\textrm{Im}(x)=\nicefrac{1}{2i}(x-x^*)$ their real and imaginary parts, which are self-adjoint operators. 

\medskip

Let $|\cdot|_{\mathcal I}:\bh\to \mathbb R_+\cup \{+\infty\}$ stand for a norm on $\bh$, consider
\begin{equation}\label{I}
\mathcal I=\{x\in \bh:|x|_{\mathcal I}<\infty\}.
\end{equation}
We assume that the norm is \textit{symmetric}, that is $|uzw|_{\mathcal I}\le \|u\|\, |z|_{\mathcal I}\|w\|$ for $v\in \mathcal I$, $u,w\in \bh$. Then $\mathcal I\subset \bh$ is a proper ideal of compact operators unless $|\cdot|_{\mathcal I}$ is equivalent to the uniform norm, in that case we can take $\mathcal I=\bh$ or $\mathcal I=\kh$. For convenience \textit{we will always assume that the space (\ref{I}) is complete with respect to the metric induced by the  $|\cdot|_{\mathcal I}$}, thus $\mathcal{I}$ is a Banach algebra.

\smallskip

In particular the norm is $\Ad$ invariant for the action of the unitary group $K=\mathcal U(H)$ (and any of its subgroups). Classical examples of non-equivalent symmetric norms are the $p$-Schatten norms $1\le p<\infty$, see de la Harpe \cite{harpe}. Of these norms, the $2$-norm is also known as the Hilbert-Schmidt norm, and it comes from the real inner product $\langle v,w\rangle_2=\re \tr(vw^*)$. In finite dimension, it is also known as the Frobenius norm and it can be also computed as $\|v\|_2^2= \sum_{i,j} |v_{ij}|^2$,  where $v_{ij}$ are the entries of the square matrix $v\in M_n$.

\medskip

Take $G=\gl(\mathcal H)$ and any symmetric norm $|\cdot|_{\mathcal I}$, consider the group
$$
GL_{\mathcal I}=\{g\in \gl(\h): g-1\in \mathcal I\}
$$
the usual exponential map $\exp(z)=\sum\limits_{n\ge 0}\frac{1}{n!}z^n$ and
$$
\lie(GL_{\mathcal I})=\{x\in \bh :\exp(tx)\subset GL_{\mathcal I} \; \forall t\in \mathbb R\}.
$$
it is easy to check that $\lie(GL_{\mathcal I})=\mathcal I$, and with the norm $|\cdot|_{\mathcal I}$ it becomes a normed Banach-Lie algebra. We provide $GL_{\mathcal I}$ with the manifold structure as an immersed submanifold of $\gl(\h)$, via its Lie algebra and the usual exponential map. We therefore obtain a Banach-Lie group with a topology that is possible non-equivalent to the uniform topology. 

\subsubsection{Restricted unitary operators}\label{restriu}

With the obvious modifications we consider the Banach-Lie subgroup 
$$
{\mathcal U}_{\mathcal I}=\{u\in \mathcal U(H): u-1\in{ \mathcal I}\}\subset GL_{\mathcal I}
$$
and its Banach-Lie sub-algebra $\lie(\mathcal {\mathcal U}_{\mathcal I})\subset \lie(GL_{\mathcal I})$ given by
$$
\lie(\mathcal {\mathcal U}_{\mathcal I})={\mathcal I}_{ah}=\{v\in \mathcal I: v^*=-v\},
$$
the anti-Hermitian operators of $\mathcal I$. For these groups of restricted unitary operators, we obtain the minimality of one-parameter groups given by Theorem \ref{mini}: to be more specific, if one considers $z$ in the ball of radius $R=\pi$ in $\mathcal I_{sh}$, $\|z\|\le |z|_{\mathcal I}<\pi$ therefore $\exp|_{B_R}$ is diffeomorphism onto its image.  Therefore $d_{\mathcal I}(1,u)=|z|_{\mathcal I}$ with the segment $t\mapsto e^{tz}$ as a short path, unique if the norm is strictly convex.

\medskip

This result however is not optimal, since using other techniques it can be shown that any $u\in \mathcal {\mathcal U}_{\mathcal I}$ can be reached by a minimizing one-parameter group, therefore the geodesic radius of $\mathcal {\mathcal U}_{\mathcal I}$ with $d_{\mathcal I}$ is \textit{infinite} (see \cite{alhr} where it was proved for the $p$-Schatten norms, $p\ge 2$, and \cite{alv} for a proof for all symmetric norms).

$\S$ It would be of interest to investigate further into weakly unitarily invariant norms such as the $c$-numerical radius, defined as in (\ref{cnum}) for $c$ a trace class operator, that is $c$ in the $1$-Schatten class of compact operators.

\subsection{The group of invertible operators in a Banach space}\label{gk}

Let $X$ be Banach space, $\mathcal L(X)$ the algebra of bounded operators and $GL(X)$ the group of invertible operators as before. Note that $GL(X)$ is open in $\mathcal L(X)$ therefore $TG\simeq GL(X)\times \mathcal L(X)$. 

\smallskip

Let $|\cdot|$ be a weakly symmetric norm in $\mathcal L(X)$, that is $|uvu^{-1}|=|v|$ for any $v\in \mathcal L(X)$ and any isometry $u\in \mathcal U(X)$. We assume that $|\cdot|$ is equivalent to the supremum norm $\|\cdot\|$, and that it is a Banach algebra norm ($|ab|\le |a|\,|b|$ for all $a,b\in \mathcal L(X)$). We further normalize so that $|1_{\mathcal L(X)}|=1$. The obvious example is $|\cdot|=\|\cdot\|$, the supremum norm.

\smallskip

There are two essential metrics that we can consider in the group $G=GL(X)$, which are right-invariant for the action of $\mathcal U(X)$. Let $g\in GL(X)$ and $v\in T_gGL(X)\simeq \mathcal L(X)$. Then we can consider
\begin{enumerate}
\item The left-invariant metric, $|v|_g=|g^{-1}v|$.
\item The flat metric, $|v|_g=|v|$.
\end{enumerate}

It is relevant to note that when restricted to $\mathcal{U}(H)$ \textit{both metrics are equal and give a bi-invariant metric} as constructed in Section \ref{isos}.

\medskip

The left-invariant metric is also $R$-uniform with $R(g)=|\Ad_g|\le \|g\|\|g^{-1}\|$, because 
$$
|(R_g)_{*h}v|_h=|vg|_{hg}=\|(hg)^{-1}vg\|=\|\Ad_g(h^{-1}v)\|\le \|g\|\, \|g^{-1}\|\, \|h^{-1}v\|=\|g\|\,\|g^{-1}\|\,|v|_h.
$$
Thus from Proposition \ref{contigru} we know that \textit{the group operations are continuous for the left-invariant metric in $GL$}. From Lemma \ref{explr}, we also know that the exponential map and its differential are locally Lipschitz, with a rough bound around $w\in \lie(G)=\mathcal L(X)$ given by the formula $C_w\le \nicefrac{1}{2} (e^{2\|w\|}-1)\|w\|^{-1}$.

\medskip

We claim that the flat metric is $L$ and $R$ uniform. To prove it, note that if $v\in\mathcal L(X)$ is regarded as a tangent vector at $h\in GL(X)$, then $(L_g)_{*h}:T_hGL(X)\to TgGL(X)$ is given by $v\mapsto gv$. Therefore 
$$
|(L_g)_{*h}|=\sup \{ |gv|_{gh}: |v|_h=1\}=\sup \{|gv|: |v|=1\}\le |g|
$$
because we are assuming it is a Banach algebra norm, therefore choosing $L(g)=|g|$ we have $|(L_g)_{*h}|\le L(g)$ for all $g,h\in G$ and $L$ is continuous, thus the metric is $L$-uniform. With the same argument, we can take $R(g)=|g|$ and the flat metric is $R$-uniform in $GL(X)$.

\begin{rem}[Left-invariant metric in $GL$]\label{ligl}
The geometry of the left-invariant metric in $GL(X)$ is not well-understood, but there are partial results: for instance if $X$ is a finite dimensional space and $|\cdot|$ is the Frobenius norm $|v|=\sqrt{\tr(v^*v)}$. We are then in the realm of classical Riemannian geometry and the unique geodesic (which is locally minimizing) starting at $g\in GL_n$, with initial speed $v$, is given by $t\mapsto ge^{tv^*}e^{t(v-v^*)}$. See \cite{alrv} for these and related results for the Hilbert-Schmidt norm and other $p$-Schatten norms of compact operators in $GL(H)$. See also Theorem \ref{arripos} below.

By the results detailed in Section \ref{isos}, it follows that at least locally, the group $\mathcal U(H)$ is geodesically convex in $GL(H)$ with the left-invariant metric, since the minimizing geodesics for $\mathcal U(H)$ are geodesic segments $t\mapsto ge^{tz}=ge^{tz^*}e^{t(z-z^*)}$ because $z\in \lie(\mathcal U(H))$, therefore $z^*=-z$.
\end{rem}

\begin{rem}[Flat metric in $GL$]\label{flat}
The geometry of the flat metric in $GL(X)$ is of course, much simpler. Since $GL(X)$ is open in $\mathcal L(X)$, for given $g\in GL(X),v\in T_gGL(X)=\mathcal L(X)$ there exists an interval $(t_-,t_+)$ around $0\in\mathbb R$ such that the flat segment $\delta(t)=g+tv$ is contained in $GL(X)$ for $t\in (t_+,t_+)$. It is easy to see that flat segments are short for the flat metric (and in fact, if the norm is strictly convex, they are the unique short paths among given endpoints). Note that in contrast with the left-invariant metric, the group of isometries is not geodesically convex as a submanifold of $GL(X)$, since rarely $e^{tz}=1+tz$ (unless $z^2=0$, which cannot happen when $X=H$ is a Hilbert space).
\end{rem}

Let $K$ be an immersed Banach-Lie group  $K\subset \mathcal U(X)$, and consider the quotient space $M=GL(X)/K$. Then $M$ can be given the quotient metric $d_M$ (with any of the two metrics) and the theory of Section \ref{hs} applies. 

\medskip

When $X=H$ is a Hilbert space and $K$ is the full group of isometries (the unitary group), much more can be said, since the quotient can be realized as the manifold $GL^+$ of positive invertible operators in $H$. In Section \ref{pio} we will give details on the quotient left-invariant metric in $GL^+$ and in Section \ref{bures} we delve into details of the the quotient flat metric in $GL^+$.

\subsection{Positive invertible Hilbert space operators}

We use the notation of Section \ref{dlh}, with $\mathcal I$ an ideal of bounded operators in $\mathcal L(H)$. Consider the action $\pi:G\times M\to M$ of $G=GL_{\mathcal I}$ on $M=GL^+_{\mathcal I}$, where
$$
GL^+_{\mathcal I}=\{a\in GL_{\mathcal I}: a^*=a, \sigma(a)\subset (0,+\infty)\}
$$
and $\pi(g,a)=gag^*$. Being open in $\mathcal I_{sa}$ (the self-adjoint operators of $\mathcal I$), there is a natural identification $TGL_{\mathcal I}^+\simeq \mathcal I_{sa}$. 

\medskip

Let $x=1\in M$, therefore $K={\mathcal U}_{\mathcal I}$.  Note that each $a\in GL^+_{\mathcal I}$ has a unique logarithm $\ln(a)\in\mathcal I_{sa}$. Therefore $GL^+_{\mathcal I}=\{gg^*:g\in \mathcal I\}=\pi_1(GL_{\mathcal I})$ and the action is transitive.   For given  $g\in GL_{\mathcal I}$ on $p=hh^*\in M=GL_{\mathcal I}^+$ the action is then given by
\begin{equation}\label{accion}
\ell_g(p)=ghh^*g^*=gpg^*,
\end{equation}
and we have $\pi_1(g)=gg^*$ therefore 
\begin{equation}\label{deriac}
(\pi_1)_{*g}\dot{g}=\dot{g}g^*+g \dot{g}^*.
\end{equation}

\subsubsection{Left-invariant metric in $GL^+$}\label{pio}

We give $GL_{\mathcal I}$ the left-invariant metric induced by the norm of the ideal. We are in the situation of the Section \ref{gk}, regarding the Banach algebra ${\mathcal I}$ acting on itself by left multiplication (therefore $X={\mathcal I}$ and $GL_{\mathcal I}=GL(X)$). In particular the action (\ref{accion}) is isometric in $GL_{\mathcal I}^+$ with the quotient metric $\mu$.

\medskip

We now characterize the quotient of the left-invariant metric for symmetric norms. In the Riemannian setting, these are well-know results (see also \cite{larro} for the extension for Hilbert-Schmidt operators). There is a precedent (in some sense) of this result for the uniform norm, in a paper by Corach and Maestripieri \cite[Theorem 2.6]{cm}, where it is introduced as a variational characterization of the horizontal lifts on the cone of positive invertible operators of $\bh$. 

\bigskip

We start with a lemma of a well-known result for matrices (see Fan and Hoffman \cite{fh}); we include a proof to illustrate the role of unitary invariance. 

\begin{lem}\label{order}
Assume that $|\cdot|_{\mathcal I}$ is a symmetric norm as in Section \ref{dlh}. Then for any $v\in \mathcal I$,
$$
\mathop{\inf}_{z\in {\mathcal I}_{ah}} |v-z|_{\mathcal I}=|\re(v)|_{\mathcal I}.
$$
\end{lem}
\begin{proof}
Taking $z=i \im(v)\in {\mathcal I}_{ah}$ the inequality $|\re(v)|_{\mathcal I}\ge \mathop{\inf}_{z\in {\mathcal I}_{ah}} |v-z |_{\mathcal I}$ is immediate. On the other hand, for any $z\in {\mathcal I}_{ah}$ there exists a unitary $u\in \bh$ such that
$$
|\re(v)|=|\re(v-z)|\le u|v-z|u^*
$$
by Thompson's inequality \cite{thompson} for the partial order of symmetric operators. It is well-known that symmetric norms preserve the spectral order, that is $0\le v\le w\in \bh$ implies $|v|_{\mathcal I}\le |w|_{\mathcal I}$, and also that the norm only depends on the singular values, i.e. $|v|_{\mathcal I}=|v^*|_{\mathcal I}=|\,|v|\,|_{\mathcal I}$. Hence $|\re(v)|_{\mathcal I}\le |u|v-z|u^*|_{\mathcal I}=|v-z|_{\mathcal I}$ for any $z\in {\mathcal I}_{ah}$.
\end{proof}

\begin{teo}\label{teoq}
Let $p=g\cdot 1=gg^*\in M=GL^+_{\mathcal I}$ and  let $v\in T_{gg^*}M\simeq \mathcal I_{sa}$. If $|\cdot|_{\mathcal I}$ is symmetric and $\mu:TM\to\mathbb R$ is the quotient metric of the left-invariant metric in $GL_{\mathcal I}$, then
\begin{equation}\label{metripos}
\mu(v)_{g\cdot 1}=|\re(g^{-1}\dot{g})|_{\mathcal I}=\nicefrac{1}{2}|p^{- \nicefrac{1}{2}}vp^{- \nicefrac{1}{2}}|_{\mathcal I}.
\end{equation}
The shortest path among $p=gg^*,q=hh^*\in M$ is given by 
$$
\gamma_{p,q}(t)=p^{ \nicefrac{1}{2}}(p^{- \nicefrac{1}{2}}qp^{- \nicefrac{1}{2}})^tp^{ \nicefrac{1}{2}}
$$
with constant speed 
$$
\mu(\dot{\gamma}_t)_{\gamma_t}=\nicefrac{1}{2} |\ln(p^{- \nicefrac{1}{2}}q p^{- \nicefrac{1}{2}})|_{\mathcal I}=d_{\mathcal I}(q^{-\nicefrac{1}{2}}p^{\nicefrac{1}{2}},{\mathcal U}_{\mathcal I}).
$$
If the norm is strictly convex, $\gamma_{p,q}$ is the unique short path in $GL_{\mathcal I}^+$ joining the given endpoints.
\end{teo}
\begin{proof}
Note that  $g=p^{ \nicefrac{1}{2}}u$ for some $u\in {\mathcal U}_{\mathcal I}$. Take a smooth path $g_t\in GL$ such that $g_0=g$, then 
$$
v=\frac{d}{dt}\Big|_{t=0}\pi_1(g_t)=\frac{d}{dt}\Big|_{t=0} g_tg_t^*=\dot{g}g^*+g\dot{g}^*=\dot{g}u^*p^{ \nicefrac{1}{2}}+p^{ \nicefrac{1}{2}}u\dot{g}^*.
$$
Then
$$
p^{- \nicefrac{1}{2}}vp^{- \nicefrac{1}{2}}=p^{- \nicefrac{1}{2}}\dot{g}u^*+u\dot{g}^*p^{- \nicefrac{1}{2}}=\Ad_u\left(g^{-1}\dot{g} + (g^{-1}\dot{g})^*\right)=2\Ad_u\left(\re(g^{-1}\dot{g})\right).
$$
Therefore   
$$
|p^{- \nicefrac{1}{2}}vp^{- \nicefrac{1}{2}}|_{\mathcal I}=2 |\re(g^{-1}\dot{g})|_{\mathcal I}=2\mathop{\inf}_{z\in {\mathcal I}_{ah}} |g^{-1}\dot{g}-z|_{\mathcal I}
$$
by the previous lemma, and we obtain the first equality of the theorem. The proof of the formula for the geodesic distance can be adapted from the classical theory of geometry in the positive cone, see the seminal paper by Mostow \cite{mostow} for the Frobenius norm, and \cite{cpr,cl} for the extensions to the uniform norm and symmetric norms in operator algebras. The assertion on the speed of $\gamma$ is immediate from the definitions  and Theorem \ref{elcorodelq}. The assertion on the uniqueness can be proved using the same idea as in \cite[Proposition 3.6]{cl}.
\end{proof}

We now obtain corollaries of the previous theorem and the results in Section \ref{geodesics}, that gives us new information regarding short paths in the general linear group with the left-invariant metric (Remark  \ref{ligl}):

\begin{teo}\label{arripos}
Let $\delta:[0,1]\to GL(H)$ be the segment $\delta(t)=ge^{tv}$ with $g\in GL(H),  v=v^*\in \mathcal L(H)$ (or any of its restricted versions in $GL_{\mathcal I}$, for a symmetric norm $|\cdot|_{\mathcal I}$). Then 
\begin{enumerate}
\item $\delta$ is shorter than any other path joining its given endpoints, for the left-invariant metric in $GL$ induced by the symmetric norm $|\cdot|_{\mathcal I}$.
\item Moreover $|v|_{\mathcal I}=L_{\mathcal I}(\delta)=d_{\mathcal I}(g,ge^v)=d_{\mathcal I}(e^v, {\mathcal U}_{\mathcal I})$ for any $g\in GL$.
\end{enumerate}
\end{teo}
\begin{proof}
Let $h=\delta(1)=ge^v$, $p=gg^*$ therefore $p^{\nicefrac{1}{2}}=\sqrt{gg^*}=|g^*|$. Let $g=|g^*|u$ be the right-polar decomposition of $g$, with $u$ a unitary operator. Let $q=hh^*=ge^ve^{v*}g^*=ge^{2v}g^*$. A straightforward computation shows that $2uvu^*=\ln(p^{-\nicefrac{1}{2}} q p^{-\nicefrac{1}{2}})$, therefore 
$$\delta_t\cdot 1=\delta_t\delta_t^*=p^{\nicefrac{1}{2}} e^{2tuvu^*} p^{\nicefrac{1}{2}}=\gamma_{p,q}(t)$$
and $\delta$ is a lift of the minimizing path $\gamma=\gamma_{p,q}\in GL^+$. Moreover, 
$$
|\delta_t^{-1}\dot{\delta}|_{\mathcal I}=|v|_{\mathcal I}=\nicefrac{1}{2}|\ln(p^{-\nicefrac{1}{2}} q p^{-\nicefrac{1}{2}})|_{\mathcal I}=\mu(\dot{\gamma}_t)_{\gamma}
$$
therefore $\delta$ is an isometric lift of $\gamma$. By Theorem \ref{arribajo2} and the left-invariance of the metric, our claims are proved.
\end{proof}

This result is particularly useful for the trace norm (the $1$-Schatten norm) and for the uniform norm of $GL(H)$, since no geodesics were known for the left-invariant metric of those norms. Even for the $p$-Schatten norms this improves significantly the results of \cite[2.1.2]{alrv}; there we obtained that these segments were critical points of the length functional, but it was unclear whether they were locally minimizing (except for the case of $p=2$, where Riemannian techniques applied).

\medskip

Note that the second assertion of the previous theorem says in particular that if $v=v^*$ then $\di_{\mathcal I}(1,e^v)=\di_{\mathcal I}({\mathcal U}_{\mathcal I},e^v)=|v|_{\mathcal I}$ in $GL_{\mathcal I}$.

\begin{lem}\label{unigl}
If the norm is strictly convex then for self-adjoint $v$, the segment $\delta(t)=ge^{tv}$ is the unique short path in $GL_{\mathcal I}$ realizing the left-invariant distance from $g$ to $ge^v$.
\end{lem}
\begin{proof}
By left-invariance we can assume that $g=1$. Let $\gamma=\delta\cdot 1$, it joins $1,e^{2v}$ in $GL^+$ and it is short with length $|2v|_{\mathcal I}$. If $\Gamma$ is another short path in $GL$ joining $1,e^v$ then $\beta=\Gamma\cdot 1$ joins $1,e^{2v}$ in $GL^+$. Using (\ref{desil}) and the previous theorem, we have
\begin{eqnarray}
L_{GL^+}(\beta) & = & L_{GL^+}(\Gamma \cdot 1)\le \Le_{GL}(\Gamma)=\di_{GL}(1,e^v) \nonumber \\
&= & \di_{GL}(\mathcal U_{\mathcal I},e^v)=d_{GL^+}(1,e^{2v})\le L_{GL^+}(\beta).\nonumber
\end{eqnarray}
Therefore $\beta$ is also short in $GL^+$ and Theorem \ref{teoq} tells us that $\gamma=\beta$. This implies that there exists a smooth path $u_t\in {\mathcal U}_{\mathcal I}$ such that $\Gamma_t=\delta_t u_t$ for $t\in [0,1]$. Therefore $\dot{\Gamma}_t=\dot{\delta}_t u_t+\delta_t \dot{u}_t$ and
$$
|\Gamma_t^{-1}\dot{\Gamma}_t|_{\mathcal I}=|u_t^*vu_t+u^*_t \dot{u}_t|_{\mathcal I}=|v+u^*_t\dot{u}_t|_{\mathcal I}\quad \forall t\in [0,1].
$$
But by Remark \ref{necemin} we also must have for each $t$
$$
|\Gamma_t^{-1}\dot{\Gamma}_t|_{\mathcal I}=\inf\{|v+u^*_t\dot{u}_t-z|_{\mathcal I}: z^*=-z\}=|v|_{\mathcal I}
$$
since $u^*_t\dot{u}_t\in \lie({\mathcal U}_{\mathcal I})$ and $v^*=v$ (see Lemma \ref{order}). Therefore for each $t$,
$$
|v+u^*_t\dot{u}_t|_{\mathcal I}=|v|_{\mathcal I}=d=|v-0|_{\mathcal I}.
$$
Where $d=\di(v, \lie(\mathcal U_{\mathcal I}))$. Since we are assuming the norm is strictly convex, there can only be one minimizer in $\lie({\mathcal U}_{\mathcal I})$ of the distance from $v$ to that subspace. Hence $u^*_t\dot{u}_t=0$, and $u_t$ is constant. Now $1=\Gamma_0=\delta_0 u_0=1 u_0$ implies that $u_t=1$ for all $t\in [0,1]$, thus $\Gamma=\delta$.
\end{proof}

\begin{rem}
It is worth noting that if we give $GL$ the left-invariant metric and $GL^+$ the induced quotient metric by $\pi_1(g)=gg^*$, in the inclusion $GL^+\hookrightarrow GL$ the metric of $GL^+$ is \textit{not} the restriction of the left-invariant metric because of (\ref{metripos}). For $p,q\in GL^+$, consider $v=\log(p^{-1}q)$ the unique bounded operator determined by the principal branch of the analytic logarithm in $\mathbb C\setminus \mathbb R_{\le 0}$. Note that $v=v^*$ if and only if $p,q$ commute. We can rewrite $pv=p\log(p^{-1}q)=p^{ \nicefrac{1}{2}}\ln(p^{- \nicefrac{1}{2}}qp^{- \nicefrac{1}{2}}) p^{ \nicefrac{1}{2}}$ and
$$
\delta(t)=pe^{tv}=pe^{t\log(p^{-1}q)}= p^{\nicefrac{1}{2}}(p^{- \nicefrac{1}{2}}qp^{- \nicefrac{1}{2}})^t p^{ \nicefrac{1}{2}}=\gamma_{p,q}(t).
$$
Therefore Theorem \ref{arripos} tells us that $p$ is a geodesic point of the inclusion $GL^+\hookrightarrow GL$ when $p$ commutes with all $q\in GL^+$, and this is the case only if $p=\lambda 1$ for some $\lambda\in\mathbb R_{>0}$.  Moreover if the norm is strictly convex, this condition is also necessary by the previous lemma. 
\end{rem}

\paragraph{Weakly symmetric norms.} In the finite dimensional setting, we can give $GL_{\mathcal I}^+$ the left-invariant metric using any weakly symmetric, since all norms are equivalent. But Lemma \ref{order} is false for such norms in general (it is false that they depend only on the modulus of a matrix, and it false that $0\le a\le b$ implies $|a|\le |b|$ for weakly unitarily invariant norms, see Bhatia and Holbrook \cite[p. 80]{bh}).  Therefore we obtain a new geometry in $GL_{\mathcal I}^+$ which for which little is known.

\smallskip

In particular if $c$ is a normal, trace class operator on an infinite dimensional Hilbert space $H$, it is known that for the $c$-numerical radius (\ref{cnum}) it holds
$$
\nicefrac{1}{2}\|c\|\,\|v\|\le \omega(v)_c\le (\tr|c|) \,\|v\|=\|c\|_1\|v\|
$$
for all $v\in \mathcal L(H)$ (see Goldberg and Straus \cite[Lemma 7]{gs}), therefore the $c$-numerical radius is equivalent to the uniform norm, and it is bi-invariant for the action of the unitary group of $H$. It is then possible to study the geometry of the left-invariant metric induced by the $c$-numerical radius in $GL$, and the quotient norm it induces in $GL^+$ using the techniques introduced here (no results are known to the author).

\subsubsection{Quotient of the flat metric in $GL^+$ (Bures' metric)}\label{bures}

We now turn to the quotient $GL^+=GL/U$ but we give $GL^+$ the quotient metric of the flat metric of $GL$ (Remark \ref{flat}), for any weakly unitarily invariant norm $|\cdot|_{\mathcal I}$. Now the flat metric is not left-invariant, therefore the action (\ref{accion}) is not isometric in $GL^+$. However, since the flat metric is trivial in $GL$, we have that, for $p=\pi_1(p^{\nicefrac{1}{2}}), q=\pi_1(q^{\nicefrac{1}{2}})\in GL^+$
\begin{equation}\label{diu}
d_{GL^+}(p,q)=\di_{\mathcal I}(p^{\nicefrac{1}{2}},q^{\nicefrac{1}{2}}\mathcal U_{\mathcal I})=\inf\{ |p^{\nicefrac{1}{2}}-q^{\nicefrac{1}{2}}u|_{\mathcal I}:u\in \mathcal{U}_{\mathcal{I}}\}.
\end{equation}
For sufficiently close $p,q\in GL^+$ the minimizing paths are the image of segments,
$$
\gamma(t)=\pi_1(p^{\nicefrac{1}{2}}+tv)=(p^{\nicefrac{1}{2}}+tv)(p^{\nicefrac{1}{2}}+tv^*)=p+t(p^{\nicefrac{1}{2}}v^*+vp^{\nicefrac{1}{2}})+t^2vv^*.
$$
To obtain a better description of the tangent metric in $GL^+$, recall that $\mathcal I_{ah}$ is the Lie algebra of the isotropy group $\mathcal U_{\mathcal I}$, and that $\mathcal I=\mathcal I_{sa} \,  \oplus  \mathcal I_{ah}$. The following will be useful:

\begin{lem}
For each $p\in GL^+$, and $w=x+y$ with $x^*=x,y^*=-y$ (a generic vector in $\mathcal I$), the linear map $T_p:w\mapsto xp+py$ is an isomorphism of $\mathcal I$.
\end{lem}
\begin{proof}
Clearly $T_p$ is linear; by the open mapping theorem $x+y\mapsto (x,y)$ with $|(x,y)|=|x|_{\mathcal I}+|y|_{\mathcal I}$ is an isomorphism. Note that, for symmetric norms, $|xp|_{\mathcal I}\le |x|_{\mathcal I}\|p\|$ for $x\in\mathcal I$ follows from the definition, and  $|xp|_{\mathcal I}\le C|x|_{\mathcal I}\|p\|$ for some constant $C>0$ if $|\cdot|_{\mathcal I}$ is weakly symmetric but equivalent to $\|\cdot\|$ in $\mathcal L(H)$. Therefore in any case $|xp+py|_{\mathcal I}\le (|x|_{\mathcal I}+|y|_{\mathcal I})C\|p\|$, thus $T_p$ is bounded. Let $xp+py=\hat{x}p+p\hat{y}$ then $mp=pn$ with $m=x-\hat{x}$ and $n=\hat{y}-y$. By the theorem of Putnam-Fuglede \cite[Theorem 12.16]{rudin} we also have $m^*p=pn^*$, thus $mp=-pn$. Therefore $mp=0=pn$ and this implies $x=\hat{x},y=\hat{y}$, proving that $T_p$ is a monomorphism. Now write $p=e^s$ with $s^*=s$ and consider 
$$
C_s:\xi\mapsto p\xi p^{-1}+p^{-1}\xi p=e^s\xi e^{-s}+e^{-s}\xi e^s=e^{\ad s}\xi+ e^{-\ad s}\xi=2\cosh(\ad s)\xi.
$$
it is well-known that for self-adjoint $s$, the linear operator $C_s$ is an isomorphism of $\mathcal I$ (see \cite[Theorem 12]{lineq}). If $z\in \mathcal I$, it is left to the reader to verify that if
$$
x=\nicefrac{1}{2}\cosh^{-1}(\ad s)\left[ zp^{-1}+p^{-1}z^*\right],\quad y=\nicefrac{1}{2}\cosh^{-1}(\ad s)\left[ zp^{-1}-p^{-1}z^*\right]
$$
then $x^*=x,y^*=-y$ and $T_p(x+y)=z$.
\end{proof}

Then we have
$$
\mathcal I=\mathcal I_{sa} \,p^{\nicefrac{1}{2}}\,  \oplus p^{\nicefrac{1}{2}}\, \mathcal I_{ah}
$$
Let $\dot{g}\in T_{p^{\nicefrac{1}{2}}}GL$, which is identified with $\mathcal I$. Then we can write $\dot{g}=x p^{\nicefrac{1}{2}}+ p^{\nicefrac{1}{2}}y$ 
where $x^*=x$ and $y^*=-y$. Then by (\ref{deriac}) a generic $v\in T_pGL^+$ can be written as
$$
v=(\pi_1)_{*g}\dot{g}=  xp+px , \quad x^*=x,
$$
therefore $x=(L_p+R_p)^{-1}v$. Since $p= p^{\nicefrac{1}{2}}\cdot 1=\pi_1( p^{\nicefrac{1}{2}})$, the tangent quotient norm can be computed as
\begin{eqnarray}
\mu(v)_p &= &\inf\{ |x p^{\nicefrac{1}{2}}+ p^{\nicefrac{1}{2}}y- p^{\nicefrac{1}{2}}z|_{\mathcal I}:z^*=-z\}=\inf\{ |x p^{\nicefrac{1}{2}}- p^{\nicefrac{1}{2}}z|_{\mathcal I}:z^*=-z\}\nonumber\\
&=&  \inf\{ | (L_p+R_p)^{-1}(v p^{\nicefrac{1}{2}})- p^{\nicefrac{1}{2}}z|_{\mathcal I}:z^*=-z\}.\nonumber
\end{eqnarray}

By the change of variables $z\mapsto w=p^{-\nicefrac{1}{2}} z p^{-\nicefrac{1}{2}}$ that preserves $\mathcal I_{ah}$, this can be rewritten as
\begin{equation}\label{bure}
\mu(v)_p =  \inf_{w\in \mathcal I_{ah}} \left| \left[(1+\Ad_p)^{-1}v- w\right]p^{-\nicefrac{1}{2}}\right|_{\mathcal I}.
\end{equation}

\begin{rem}In particular, for the Frobenius norm $|\cdot|_2$ that comes from trace inner product, it is easy to check that since $x^*=x$, then $xp^{\nicefrac{1}{2}}$ is orthogonal to the subspace $p^{\nicefrac{1}{2}} I_{ah}$. Therefore 
$$
\mu(v)_p^2 =|x p^{\nicefrac{1}{2}}|_2^2=\tr(xpx)=\nicefrac{1}{2}\, \tr(x(xp+px))=\nicefrac{1}{2}\, \tr(v \,(L_p+R_p)^{-1}v)
$$
which is the Bures metric, of particular relevance in quantum information theory (see Dittmann \cite{dittmann} and the references therein). That the Bures distance can be computed using the infimum in (\ref{diu}) was proved recently, see Bhatia, Jain and Lim \cite{bjl}, were they also show that the minimum is attained for $u_0$ the unitary operator in the polar decomposition of $q^{\nicefrac{1}{2}}p^{\nicefrac{1}{2}}$, that is $q^{\nicefrac{1}{2}}p^{\nicefrac{1}{2}}=u_0 |q^{\nicefrac{1}{2}}p^{\nicefrac{1}{2}}|$.
\end{rem}

\medskip

$\S$ It would be of interest to obtain better expressions for the quotient tangent metric (\ref{bure}) for other unitarily invariant norms.

\subsubsection{Positive operators in a $C^*$-algebra}

It is worth mentioning that the results of the previous sections, for the particular case of the uniform norm in the full group $GL(H)$, can be adapted verbatim to a $C^*$-algebra $\mathcal A$. 

\begin{prop}Let $GL(\mathcal A)$ be the full group of invertible operators of $\mathcal A$, $GL^+(\mathcal A)$ be the manifold of positive invertible elements and $\mathcal U(\mathcal A)$ the group of unitary elements. 
\begin{enumerate}
\item Left invariant metric in $GL(\mathcal A)$: let $\delta(t)=ge^{tv}$ with $g\in GL(\mathcal A)$,  $v\in \mathcal A_h$. Then $\delta$ is shorter than any other path joining its given endpoints, for the left-invariant metric in $GL(\mathcal A)$. Moreover $\|v\|=d_{\infty}(e^v, {\mathcal U}(\mathcal A))$ in $GL(\mathcal A)$.
\item Bures' metric in $GL^+(\mathcal A)$: for $p\in GL^+(\mathcal A)$ and $v^*=v$, the generalized Bures metric is given by
$$
\|v\|_p = \inf\{ \| (L_p+R_p)^{-1}(v p^{\nicefrac{1}{2}})- p^{\nicefrac{1}{2}}z\| :z^*=-z\}.
$$ 
For sufficiently close $p,q\in GL^+(\mathcal A)$, minimizing paths are $\gamma(t)=p+t(p^{\nicefrac{1}{2}}v^*+vp^{\nicefrac{1}{2}})+t^2vv^*$ and
$$
d_{GL^+}(p,q)=\inf\{ \| p^{\nicefrac{1}{2}}-q^{\nicefrac{1}{2}}u\|:u^*=u^{-1}\}.
$$ 
\end{enumerate}
\end{prop}

\subsection{Groups of maps}

In this section we discuss some applications to groups of smooth functions, we start by considering groups given by the composition of maps.

\subsubsection{The group of diffeomorphisms of a compact manifold}\label{grudif}

For a compact manifold $M$, consider the group of its $C^{\infty}$ diffeomorphisms $G=\difm$, under the composition of maps. Then $G$ is a Lie group with locally convex Lie algebra $\lie(G)\simeq \mathcal X(M)$, the smooth vector fields on $M$ (for the details on the manifolds structure see \cite[Example II.3.14]{neeblie}).

\smallskip

The exponential map is given by the $1$-time evaluation of the flow of a field $X$: if $\varphi^X(t,p)$ is the flow of $X\in \mathcal X(M)$, then $\exp_G(X)(p)=\varphi^X(1,p)$, shortly $\exp(X)=\varphi^X_1$. This exponential map is smooth and $\exp_{*0}=id$, but in locally convex spaces there is no general inverse function theorem and usually $\exp_G$ is not a local diffeomorphism so it is unfit as a chart. The group $\dm$ is however, a regular Lie group. If $M$ is not compact, $\dm$ is not a  manifold but there is a notion of smooth map $\varphi:N\to \dm$ for smooth manifolds $N$  (see \cite[Definition II.3.1, p. 329]{neeblie}).

\smallskip

For a given (smooth, classical) Finsler metric $F$ on $M$, it is possible to define a tangent norm in $TG$ with the formula
\begin{equation}\label{metrim}
|X|_{\eta}=\int_M |X(p)|_{\eta(p)} d\mu(p),
\end{equation}
 recalling $X\in C^{\infty}(M,TM)$ belongs to $T_{\eta}G$ if and only if $X_p\in T_{\eta(p)}M$ for all $p\in M$. The measure used here can be the Hausdorff/Busemann measure, the Holmes-Thompson metric, or the metric induced by the auxiliary Riemannian metric (see Berck and Paiva \cite{paiva2} and the references therein). If $f,\eta\in \difm$ then $L_f\eta=f\circ\eta$ and $R_f\eta=\eta\circ f$ thus
$$
(L_f)_{*}X=Df X\textrm{ and } (R_f)_{*}X=X\circ f\quad \forall X\in T_{\eta}G.
$$

\begin{rem}
Let $L(f)=\max\{|Df_pv|_p:|v|_p=1, p\in M\}$, note that $L(1)=1$. Since $M$ is compact, the sphere bundle of $M$, $SM=\{(p,v):p\in M, |v|_p=1\}$ is also compact in $TM$. For $\varepsilon>0$, consider the open set in $TM$
$$
TU=\{(q,w):q\in M, |w|_q<1+\varepsilon\},
$$
then $\mathcal F(SM,TU)$ is a neighborhood of the identity map $1_*:TM\to TM$, $1_*(p,v)=(p,v)$, since $\dm$ has the compact open topology. If $f_i\to 1$ in $\dm$, then there exists $i_0$ such that $i\ge i_0$ implies that $(f_i)_*(SM)\subset TU$, that is $|(Df_i)_pv|_{f(p)}<1+\varepsilon$ for all $p\in M$, $|v|_p=1$. Thus $L(f_i)<1+\varepsilon$ for all $i\ge i_0$, and this proves that $L$ is upper-semicontinuous. Hence \textit{these metrics (\ref{metrim}) in $\dm$ are $L$-uniform}. 
\end{rem}

On the other hand if the metric in $M$ is Riemannian then 
$$
f\mapsto R(f)=\max\{|\det(Df^{-1}_p)|: p\in M\}
$$
is continuous and can be used to show that the metric is $R$-uniform using the change of variables formula for the integral. What are the possible choices of $R$ for measures compatible with a Finsler geometry on $M$, such as the Holmes-Thompson metric?  This would provide some control over the group operations (Proposition \ref{contigru}) and the group exponential map (Lemma \ref{explr}).

\smallskip

\begin{rem}
By considering the subgroup $K=\difm_{\mu}$ of diffeomorphisms preserving the measure $\mu$, one obtains that that the metric is right-invariant for the action of $K$, therefore one can give the space of densities $G/K$ the quotient metric of Section \ref{hs}. Extensions of the classical groups of symplectomorphisms, and Hamiltonian symplectomorphisms (see Kriegl and Michor \cite{km}) can also be regarded as (pseudo)-metric groups and homogeneous spaces, in this fashion. The groups of Sobolev diffeomorphisms studied by Ebin and Marsden are only half-lie groups (from the right) but nevertheless the methods described here in Sections \ref{intro1} and \ref{hs} apply. In particular we have the equivalent characterization of quotient metrics (Theorem \ref{elcorodelq}) and the characterization of geodesics (Theorem \ref{arribajo2}).
\end{rem}

\begin{rem}\label{symplec} When $K=Ham(M,\omega)$, the group of Hamiltonian symplectomorphisms of a compact symplectic manifold $M$, it is natural to consider the bi-invariant metric given by Hofer. To this end we recall that for a time dependent Hamiltonian $H_t:M\to \mathbb R$, its norm is computed as $|H_t|=\max H_t-\min H_t$, and if $X_{H_t}$ is the Hamiltonian vector field induced by $H_t$ (i.e. the symplectic gradient of $H_t$ defined by the relation $\omega(X_{H_t},\cdot)=dH_t$), then the flow of $H_t$ is given by the equation $\dot{\phi}_t\circ \phi_t^{-1}=X_{H_t}$. Thus if $\phi_t$ is any path of symplectomorphisms, its length is given by $L(\phi)=\int_0^1 |H_t|dt$. As we mentioned earlier, Theorem \ref{nonstrict} above regarding a characterization of short paths provides another approach to the characterization of geodesics with this metric (related to the quasi-autonomous Hamiltonians, see \cite{lm1}). However, since the hypothesis on the exponential map being a local diffeomorphism fails for groups of symplectomorphisms, our approach seems to need further refinement to fully understand the relation.
\end{rem}

\subsubsection{Gauge groups, Loop groups}

Now we discuss groups of maps $G=C_c^{\infty}(M,K)$ where $M$ is a $\sigma$-compact (finite dimensional) manifold and $K$ is a locally convex Lie group, with the point-wise operations. The lie algebra of $C_c^{\infty}(M,K)$ is identified with $C_c^{\infty}(M,\mathfrak k)$, where $\mathfrak k=\lie(K)$, and if $K$ has a smooth exponential and $v\in C_c^{\infty}(M,\mathfrak k)$ then 
$$
\exp(v)(z)=\exp_K(v(z)),\quad z\in M
$$
is the smooth exponential of $C_c^{\infty}(M,K)$. Moreover, if $K$ is a locally exponential then $C_c^{\infty}(M,K)$ is locally exponential  \cite[Theorem IV.1.12]{neeblie}. If in addition, $K$ is regular, then $C_c^{\infty}(M, K)$ is regular, and if $K$ is $BCH$, then so is $C_c^{\infty}(M,K)$. In particular $C^{\infty}(M,K)$ is a $BCH$-Lie group for any compact (finite dimensional) manifold $M$ and any Banach-Lie group $K$. These considerations can be specialized to loop groups $LK=C^{\infty}(S^1,K)$, or else generalized to gauge groups $Gau_c(P)$ where $q:P\to M$ is a smooth $K$-principal bundle over $M$ (the gauge group is the group of compactly supported gauge transformations of the bundle).

\smallskip

If $|\cdot|:TK\to \mathbb R$ is a Finsler metric in $K$, there are several ways to induce  metrics in $TC_c^{\infty}(M,K)$. Note that if $f\in C_c^{\infty}(M,K)$, and  $w\in T_f C_c^{\infty}(M,K)$, then we can take 
$$
|w|_{f,\infty}=\max_{z\in M} |w(z)|_{f(z)},
$$
or the $p$-norms 
$$
|w|_{f,p}=\left(\int_M |w(z)|^p_{f(z)} d\mu(z)\right)^{\nicefrac{1}{p}},\quad 1\le p<\infty
$$
where $\mu$ is some reasonable measure in $M$ (the volume form of a given Riemannian metric on $M$, or $\omega^n$ for a symplectic form $\omega$ in $M$ when $\dim(M)=2n$, etc.). 

\smallskip

There are also interesting variations of these groups using Sobolev spaces where we can apply those ideas, for example let $K$ be a compact connected Lie groups and let $L^{s_0}K=H^{s_0}(M,K)$ be the Sobolev $H^{s_0}$ maps. This is a Hilbert Lie group when $s_0>\dim(M)/2$, and the Laplacian $\Delta$ of $M$ induces a Sobolev metric in $L^{s_0}K$ (see Freed \cite{freed}).

\begin{rem}\label{lindul}
It is immediate from these constructions that if the metric in $K$ is left (or right) invariant, the induced metric is left (resp. right) invariant.  If the metric in $K$ is $L$-uniform, let
$$
\mathbb L(f)=\max_{z\in M}L(f(z)),
$$
note that $\mathbb L(\mathbbm 1)=1$ (we denote $\mathbbm 1$ to the identity of $C^{\infty}(M,K)$). For $\varepsilon>0$, there exists a neighborhood $U$ of $1\in K$ such that $L(k)-1<\varepsilon$ for $k\in K$. Since $C^{\infty}(M,K)$ has the compact open topology, if $f_i\to \mathbbm 1$ there exists $i_0$ such that $i\ge i_0$ guarantees that $f_i(M)\subset U$, which implies $L(f_i(z))-1<\varepsilon$ for all $i\ge i_0$ and all $z\in M$. Taking the maximum over $z\in M$ implies $\mathbb L(f_i)-1\le\varepsilon$ for all $i\ge i_0$. Thus $\mathbb L:C^{\infty}(M,K)\to\mathbb R_{\ge 0}$ \textit{is also  upper semi-continuous}. If $g,h\in C_c^{\infty}(M,K)$ and $v\in T_hC_c^{\infty}(M,K)$ then 
$$
|(L_g)_{*h}v (z)|_{gh(z)}=|(L_{g(z)})_{*h(z)}v(z)|_{g(z)h(z)}\le L(g(z))|v(z)|_{h(z)}\le \mathbb L(g) |v(z)|_{h(z)}
$$
therefore $|(L_g)_{*h}v|_{gh,\infty}\le  \mathbb L(g) \, |v|_{h,\infty}$. This shows that the maximum norm is $L$-uniform in $C^{\infty}_c(M,K)$. Likewise, since
$$
\int_M |(L_g)_{*h}v(z)|_{gh(z)}^p d\mu(z)\le \mathbb L(g)^p \int_M  |v(z)|_{h(z)}^p d\mu(z)
$$
the $p$-norms are also $L$-uniform. The same considerations hold for a $R$-uniform metric in $K$. Hence the results of Section 1 and 2 apply to all these constructions. This can be used to compare the metric in the full group of free loops $LK=C^{\infty}(S^1,K)$ with the metric in the group $\Omega K=LK/K$ of based loops; there are also interesting actions of $G=LGL_n$ in the space of harmonic maps (see \cite{bg} and the references therein).
\end{rem}

In particular, if $|\cdot|$ is a bi-invariant metric in $K$, then those constructions give bi-invariant metrics in $C^{\infty}_c(M,K)$, and if $K$ is locally exponential the results of Section \ref{biinvme} regarding the local minimality of one-parameter groups also apply. Taking into account the examples of Section \ref{dlh}, there are several possible choices of $K$ to combine these ideas.

\medskip

To end this paper, we illustrate this observations with a particular loop group. Let $K=\mathcal U(H)$, the unitary group of a separable complex Hilbert space $H$, so $\mathfrak k=\mathcal L(H)_{sh}$. Let $M=S^1$ so $G=L \mathcal U=C^{\infty}(S^1,\mathcal U (H))$ is the loop group of unitary operators. Being a Banach-Lie group, $\mathcal U (H)$ is locally exponential; more precisely if $B_{\pi}=\{z=-z^*: \|z\|<\pi\}$ then the operator exponential is a diffeomorphism from $B_{\pi}\subset\lie( \mathcal U (H))$ onto $V_{\pi}=\{u\in \mathcal U (H):\|u-1\|<2\}$ (recall that we use $\|\cdot\|$ to indicate the usual supremum operator norm). So we give $L\mathcal U$ the bi-invariant metric
$$
|v|_g=|g^{-1}v|_{\infty}=\max\{ \|g(z)^{-1}v(z)\|=\|v(z)\|: z\in S^1\},
$$
and we note that if $v\in \tilde{B_{\pi}}=\{v\in\lie(G): \max\{\|v(z)\|:z\in S^1\}<\pi$ then the exponential map $\exp_G|_{\tilde B}\to \tilde{V}=\exp_G(\tilde B)$ is a diffeomorphism. Hence by Theorem \ref{mini}

\begin{coro}Let $\gamma\in L\mathcal U$ with 
$$
|1-\gamma|_{\infty}=\max\{\|1-\gamma(z)\|:z\in S^1\}<2.
$$
Then there exists a unique $v\in \tilde{B_{\pi}}$ such that $\gamma=\exp_G(v)$ (that is $\gamma(z)=e^{v(z)}$ for all $z\in S^1$), and $\delta_t=\exp_G(tv)$ is shorter than any other piecewise smooth path joining its given endpoints $1,g$ in $L\mathcal U$. In particular for the intrinsic metric induced in  the loop group,
$$
\di_{\infty}(1,\gamma)=|v|_{\infty}=\max\{\|v(z)\|:z\in S^1\}.
$$ 
\end{coro}

Variations of this statement for the restricted loop groups of compact operators (Section \ref{restriu}) are straightforward.

\section*{Acknowledgements} This research was supported by CONICET (PIP 2014 11220130100525) and ANPCyT (PICT 2015 1505).

\bigskip

{\sc Gabriel Larotonda.}\\
{Instituto Argentino de Matem\'atica, ``Alberto P. Calder\'on'' (CONICET), and \\  
Departamento de Matem\'atica, Facultad de Cs. Exactas y Naturales, Universidad de Buenos Aires. Ciudad Universitaria (1428) CABA, Argentina }

\noindent e-mail: {\sf glaroton@dm.uba.ar }


\begin{thebibliography}{XX}

\bibitem{paiva2} J. C. \'Alvarez Paiva, G. Berck: What is wrong with the Hausdorff measure in Finsler spaces. Adv. Math. 204 (2006), no. 2, 647--663.

\bibitem{paiva} J. C. \'Alvarez Paiva, C. E. Dur\'an: Isometric submersions of Finsler manifolds. Proc. Amer. Math. Soc. 129 (2001), no. 8, 2409--2417.

\bibitem{andreev} P. Andreev: Foundations of singular Finsler geometry. Eur. J. Math. 3 (2017), no. 4, 767--787.

\bibitem{aa} E. Andruchow, A. C. Antunez: Quotient $p$-Schatten metrics on spheres. Rev. Un. Mat. Argentina 58 (2017), no. 1, 21--36.

\bibitem{acl} E. Andruchow, E. Chiumiento, G. Larotonda: Homogeneous manifolds from noncommutative measure spaces. J. Math. Anal. Appl. 365 (2010), no. 2, 541--558.

\bibitem{alhr} E. Andruchow, G. Larotonda: Hopf-Rinow theorem in the Sato Grassmannian. J. Funct. Anal. 255 (2008), no. 7, 1692--1712.

\bibitem{alr} E. Andruchow, G. Larotonda, L. Recht: Finsler geometry and actions of the $p$-Schatten unitary groups. Tran. Amer. Math. Soc. 362 (2010) no. 1, 319--344.

\bibitem{alrv} E. Andruchow, G. Larotonda, L. Recht, A. Varela: The left invariant metric in the general linear group. J. Geom. Phys. 86 (2014), 241--257.

\bibitem{cocoeste} E. Andruchow, L. Recht: Sectional curvature and commutation of pairs of selfadjoint operators. J. Operator Theory 55 (2006) no. 2, 225--238.

\bibitem{avhm} E. Andruchow, A. Varela: Metrics in the sphere of a $C^\ast$-module. Cent. Eur. J. Math. 5 (2007), no. 4, 639--653.

\bibitem{alv} J. Antezana, G. Larotonda, A. Varela: Optimal paths for symmetric actions in the unitary group. Comm. Math. Phys. 328 (2014), no. 2, 481--497.

\bibitem{atkin1} C. J. Atkin: The Finsler geometry of groups of isometries of Hilbert space. J. Austral. Math. Soc. Ser. A 42 (1987) no. 2, 196--222.

\bibitem{atkin2} C. J. Atkin: The Finsler geometry of certain covering groups of operator groups. Hokkaido Math. J. 18 (1989) no. 1, 45--77.

\bibitem{bao} D. Bao, S.-S. Chern, Z. Shen. An introduction to Riemann-Finsler geometry. Graduate Texts in Mathematics, 200. Springer-Verlag, New York, 2000.

\bibitem{guerre} J. Becerra Guerrero, A. Rodríguez-Palacios: Transitivity of the norm on Banach spaces. Extracta Math. 17 (2002), no. 1, 1--58.

\bibitem{belkale} P. Belkale: Quantum generalization of the Horn conjecture. J. Amer. Math. Soc. 21 (2008), no. 2, 365--408.

%\bibitem{beltita} D. Belti\c{t}\u a: Smooth homogeneous structures in operator theory. Chapman \& Hall/CRC Monographs and Surveys in Pure and Applied Mathematics, 137. Chapman \& Hall/CRC, Boca Raton, FL, 2006.

\bibitem{beres} V. N. Berestovski\u i: Homogeneous manifolds with an intrinsic metric. I. (Russian) ; translated from Sibirsk. Mat. Zh. 29 (1988), no. 6, 17--29 Siberian Math. J. 29 (1988), no. 6, 887--897 (1989).

\bibitem{beres2} V. N. Berestovski\u i: Homogeneous manifolds with an intrinsic metric. II. (Russian) ; translated from Sibirsk. Mat. Zh. 30 (1989), no. 2, 14--28, 225 Siberian Math. J. 30 (1989), no. 2, 180--191.

\bibitem{bg} M. J. Bergvelt, M. A. Guest: Actions of loop groups on harmonic maps. Trans. Amer. Math. Soc. 326 (1991), no. 2, 861--886.

\bibitem{bhatia} R. Bhatia: Matrix analysis. Graduate Texts in Mathematics 169. Springer-Verlag, New York, 1997.

\bibitem{bh} R. Bhatia, J. A. R. Holbrook: A softer, stronger Lidskii theorem. Proc. Indian Acad. Sci. Math. Sci. 99 (1989), no. 1, 75-83.

\bibitem{bjl} R. Bhatia, T. Jain, Y. Lim: On the Bures-Wasserstein distance between positive definite matrices. Expositiones Mathematicae (2018) in press.

\bibitem{bialyp} M. Bialy, L. Polterovich: Geodesics of Hofer's metric on the group of Hamiltonian diffeomorphisms. Duke Math. J. 76 (1994), no. 1, 273--292.

\bibitem{bbi} D. Burago, Y. Burago, S. Ivanov: A course in metric geometry. Graduate Studies in Mathematics, 33. American Mathematical Society, Providence, RI, 2001.

\bibitem{crm} B. Clarke: The metric geometry of the manifold of Riemannian metrics over a closed manifold. Calc. Var. Partial Differential Equations 39 (2010), no. 3-4, 533--545.

\bibitem{cl} C. Conde, G. Larotonda: Manifolds of semi-negative curvature. Proc. Lond. Math. Soc. (3) 100 (2010), no. 3, 670--704.

\bibitem{cm} G. Corach, A. Maestripieri: Positive operators on Hilbert space: a geometrical view point. Colloquium on Homology and Representation Theory (Spanish) (Vaquerías, 1998). Bol. Acad. Nac. Cienc. (Córdoba) 65 (2000), 81--94. 

\bibitem{cpr} G. Corach, H. Porta, L. Recht: Geodesics and operator means in the space of positive operators. Internat. J. Math. 4 (1993), no. 2, 193--202. 

\bibitem{dittmann} J. Dittmann: Some properties of the Riemannian Bures metric on mixed states. J. Geom. Phys. 13 (1994), no. 2, 203--206.

\bibitem{menucci} A. Duci, A. C. G. Mennucci: Banach-like distances and metric spaces of compact sets. SIAM J. Imaging Sci. 8 (2015), no. 1, 19--66.

\bibitem{dmlr0} C. E. Durán, L. E. Mata-Lorenzo, L. Recht: Natural variational problems in the Grassmann manifold of a $C^*$-algebra with trace. Adv. Math. 154 (2000), no. 1, 196--228.

\bibitem{dmlr} C. E. Durán, L. E. Mata-Lorenzo, L. Recht: Metric geometry in homogeneous spaces of the unitary group of a $C^*$-algebra. I. Minimal curves. Adv. Math. 184 (2004), no. 2, 342--366.

\bibitem{dmlr2} C. E. Durán, L. E. Mata-Lorenzo, L. Recht: Metric geometry in homogeneous spaces of the unitary group of a $C^\ast$-algebra. II. Geodesics joining fixed endpoints. Integral Equations Operator Theory 53 (2005), no. 1, 33--50.

\bibitem{fh} K. Fan, A. J. Hoffman: Some metric inequalities in the space of matrices. Proc. Amer. Math. Soc. 6, (1955). 111--116. 

\bibitem{freed} D. S. Freed: The geometry of loop groups. J. Differential Geom. 28 (1988), no. 2, 223--276.

\bibitem{gallot} S. Gallot, D. Hulin, J. Lafontaine: Riemannian geometry. Third edition. Universitext. Springer-Verlag, Berlin, 2004.

\bibitem{glock} H. Gl\"ockner: Lie group structures on quotient groups and universal complexifications for infinite-dimensional Lie groups. J. Funct. Anal. 194 (2002), no. 2, 347--409.

\bibitem{gs} M. Goldberg, E. G. Straus: Norm properties of $c$-numerical radii. Linear Algebra Appl. 24 (1979), 113--131.

\bibitem{harpe} P. de la Harpe: Classical Banach-Lie Algebras and Banach-Lie Groups of Operators in Hilbert Space. Lecture Notes in Mathematics 285, Springer, Berlin, 1972.

\bibitem{harris} L. A. Harris, W. Kaup: Linear algebraic groups in infinite dimensions. Illinois J. Math 21 (1977) no. 3, 666--674.

\bibitem{hofer} H. Hofer, E. Zehnder: Symplectic invariants and Hamiltonian dynamics. Reprint of the 1994 edition. Modern Birkh\"auser Classics. Birkh\"auser Verlag, Basel, 2011.

\bibitem{km} A. Kriegl, P. W. Michor: The convenient setting of global analysis. Mathematical Surveys and Monographs, 53. American Mathematical Society, Providence, RI, 1997.

\bibitem{kmr} A. Kriegl, P. W. Michor, A. Rainer, Armin: An exotic zoo of diffeomorphism groups on $\mathbb R^n$. Ann. Global Anal. Geom. 47 (2015), no. 2, 179--222.

\bibitem{lm1} F. Lalonde, D. McDuff: Hofer's $L^\infty$-geometry: energy and stability of Hamiltonian flows. I, II. Invent. Math. 122 (1995), no. 1, 1--33, 35--69.

\bibitem{latifi} D. Latifi: Homogeneous geodesics in homogeneous Finsler spaces. J. Geom. Phys. 57 (2007), no. 5, 1421--1433.

\bibitem{larro} G. Larotonda: Nonpositive Curvature: a Geometrical approach to Hilbert-Schmidt Operators. Differential Geom. Appl. 25 (2007), no. 6, 679--700.

\bibitem{lineq} G. Larotonda: Norm inequalities in operator ideals. J. Funct. Anal. 255 (2008), no. 11, 3208--3228.

%\bibitem{menger} K. Menger: Zur Dimensions-und Kurventheorie. (German) Unveröffentlichte Aufsätze aus den Jahren 1921-23. Monatsh. Math. Phys. 36 (1929), no. 1, 411--432.


\bibitem{nm}  T. Marquis, K.-H. Neeb: Half-Lie groups. arXiv preprint (2018), arXiv: 1607.07728.


\bibitem{mennu2} A. C. G. Mennucci: On asymmetric distances. Anal. Geom. Metr. Spaces 1 (2013), 200--231.

\bibitem{mcalpin} J. McAlpin: Infinite dimensional manifolds and Morse theory.  Thesis (Ph.D.)--Columbia University. ProQuest LLC, Ann Arbor, MI, 1965.

\bibitem{mm} P. W. Michor, D. Mumford: Riemannian geometries on spaces of plane curves. J. Eur. Math. Soc. (JEMS) 8 (2006), no. 1, 1--48.

%\bibitem{milnor} J. Milnor: Morse theory. Based on lecture notes by M. Spivak and R. Wells. Annals of Mathematical Studies \textbf{51}, Princeton University Press, Princeton, N.J. 1963.

\bibitem{milcur} J. Milnor: Curvatures of left invariant metrics on Lie groups. Advances in Math. 21 (1976), no. 3, 293--329.

\bibitem{mostow} G. D. Mostow: Some new decomposition theorems for semi-simple groups. Mem. Amer. Math. Soc. 1955, (1955) no. 14, 31--54.

\bibitem{neeb} K.-H. Neeb: A Cartan-Hadamard Theorem for Banach-Finsler manifolds. Geom. Dedicata 95 (2002) 115--156.

\bibitem{neebmss} K.-H. Neeb: Monastir summer school: Infinite dimensional Lie groups.  TU Darmstadt Preprint 2433 (2006).

\bibitem{neeblie} K.-H. Neeb: Towards a Lie theory of locally convex groups. Jpn. J. Math. 1 (2006), no. 2, 291--468.

\bibitem{portarecht} H. Porta, L. Recht: Minimality of geodesics in Grassmann manifolds. Proc. Amer. Math. Soc. 100 (1987) no. 3, 464--466.

\bibitem{repovs} D. Repov\v{s}, P. V. Semenov: Continuous selections of multivalued mappings. Mathematics and its Applications, 455. Kluwer Academic Publishers, Dordrecht, 1998.

\bibitem{rudin} W. Rudin: Functional analysis. Second edition. International Series in Pure and Applied Mathematics. McGraw-Hill, Inc., New York, 1991.

\bibitem{take} M. Takesaki. Theory of operator algebras. III. Springer-Verlag, New York-Heidelberg, 1979.

\bibitem{thompson} R.C. Thompson: Matrix type metric inequalities. Linear and Multilinear Algebra 5 (1977/78), no. 4, 303--319.

%\bibitem{tumpach} A.B. Tumpach: Mostow's Decomposition Theorem for $L^*$-groups and Applications to affine coadjoint orbits and stable manifolds, preprint (2006).

\bibitem{upmeier} H. Upmeier: Symmetric Banach manifolds and Jordan $C^*$-algebras. North-Holland Mathematics Studies, 104. Notas de Matem\'atica [Mathematical Notes], 96. North-Holland Publishing Co., Amsterdam, 1985.

%\bibitem{wald} A. Wald: Axiomatik des Zwischenbegriffes in metrischen Räumen. (German) Math. Ann. 104  (1931), no. 1, 476-484.
\end{thebibliography}
\end{document}